\documentclass[11pt]{amsart}

\usepackage{fullpage}

\usepackage{amsmath,amssymb,mathrsfs}
\usepackage{amsthm}

\usepackage{hyperref}
\usepackage{url}

\usepackage{graphicx}
\usepackage{enumitem}

\renewcommand\labelenumi{\roman{enumi})} 
\renewcommand\theenumi\labelenumi

\newcommand{\eps}{\varepsilon}
\newcommand{\Ups}{\Upsilon}

\newcommand{\wdt}[1]{\widetilde #1}

\newcommand{\ovl}[1]{\overline #1}

\newcommand{\pl}{\partial}
\newcommand{\bs}{\backslash}

\newcommand{\vol}{\operatorname{vol}}
\newcommand{\im}{\operatorname{im}}
\newcommand{\tr}{\operatorname{tr}}
\newcommand{\otr}{\operatorname{Tr}}

\newcommand{\Real}{\operatorname{Re}}

\newcommand{\sym}{\operatorname{Sym}}
\newcommand{\End}{\operatorname{End}}
\renewcommand{\hom}{\operatorname{Hom}}
\newcommand{\inj}{\operatorname{inj}}
\newcommand{\id}{\operatorname{id}}
\newcommand{\sys}{\operatorname{sys}}

\newcommand{\PSL}{\mathrm{PSL}}
\newcommand{\SL}{\mathrm{SL}}
\newcommand{\SU}{\mathrm{SU}}
\newcommand{\SO}{\mathrm{SO}}

\newcommand{\so}{\mathcal{O}}
\newcommand{\sh}{\mathcal{H}}
\newcommand{\sB}{\mathcal{B}}

\newcommand{\LP}{\mathfrak{p}}
\newcommand{\LSL}{\mathfrak{sl}}
\newcommand{\frg}{\mathfrak{g}}
\newcommand{\frk}{\mathfrak{k}}

\newcommand{\fri}{\mathfrak{I}}

\newcommand{\abs}{\mathrm{abs}}

\newcommand{\res}{\mathrm{res}}
\newcommand{\cusp}{\mathrm{cusp}}
\newcommand{\cont}{\mathrm{cont}}
\newcommand{\disc}{\mathrm{disc}}
\newcommand{\eis}{\mathrm{Eis}}
\newcommand{\lox}{\mathrm{lox}}
\newcommand{\tors}{\mathrm{tors}}
\newcommand{\free}{\mathrm{free}}

\newcommand{\ad}{\mathrm{Ad}}

\newcommand{\NN}{\mathbb N}

\newcommand{\CC}{\mathbb C}
\newcommand{\RR}{\mathbb R}
\newcommand{\ZZ}{\mathbb Z}
\newcommand{\HH}{\mathbb H}

\newcommand{\QQ}{\mathbb Q}

\title{Asymptotics of analytic torsion for hyperbolic three--manifolds}

\author{Jean Raimbault}

\address{Institut de Math\'ematiques de Toulouse ; UMR5219 Universit\'e de Toulouse ; CNRS--UPS IMT, F-31062 Toulouse Cedex 9, France}
\email{jraimbau@math.univ-toulouse.fr}


\newtheorem{theo}{Theorem}[section]
\newtheorem{lem}[theo]{Lemma}
\newtheorem{prop}[theo]{Proposition}

\newtheorem{theostar}{Theorem}

\newtheorem{propstar}[theostar]{Proposition}

\begin{document}

\numberwithin{equation}{section}

\begin{abstract}
We prove that for certain sequences of hyperbolic three--manifolds with cusps which converge to hyperbolic three--space in a weak (``Benjamini-Schramm'') sense and certain coefficient systems the regularised analytic torsion approximates the $L^2$-torsion of the universal cover. 

We also prove an asymptotic equality between the former and the Reidemeister torsion of the truncated manifolds. 
\end{abstract}

\maketitle

\setcounter{tocdepth}{1}
\tableofcontents

\setcounter{tocdepth}{2}

\section{Introduction}

\subsection{Integral homology of congruence manifolds}

In \cite{BV} N. Bergeron and A. Venkatesh have shown that for odd $m$, in sequences of compact arithmetic hyperbolic $m$-manifolds which converge to $\HH^m$ the homological torsion has an exponential growth for certain local systems. That is, there exists $\QQ$-representations of $\SO(m,1)$ on a space $V$ such that if $\Gamma$ is a uniform arithmetic lattice in this $\QQ$-form of $\SO(m,1)$, preserving a lattice $V_\ZZ$ in $V$ and $\Gamma_n$ a sequence of finite-index subgroups of $\Gamma$ such that the injectivity radius of the $M_n=\Gamma_n\backslash \HH^m$ goes to infinity we have that
\begin{equation} \label{limBV}
\liminf_{n\to\infty} \sum_{\substack{p=1,\ldots,m-1\\ p=\frac{m-1}2\pmod{2}}} \frac{\log|H_p(\Gamma_n,V_\ZZ)_\tors|}{\vol M_n} >0. 
\end{equation}
In \cite{7S} it is essentially proven that the limit \eqref{limBV} holds for any sequence of torsion-free congruence subgroups of a uniform arithmetic lattice (see \cite[6.1]{thesis} for a detailed argument). Moreover, when $m=3$ elementary arguments show that one can deduce from Bergeron and Venkatesh's proof an actual limit for the left-hand side, that is 
\begin{equation} \label{limBV_d=3}
\lim_{n\to +\infty} |H_1(\Gamma_n;V_\ZZ)|^{1/\vol M_n} = c
\end{equation}
where $c>1$ depends only on $V$. The present paper, originating from the author's Ph.D. thesis \cite{thesis}, aims at providing tools to prove an analogue of \eqref{limBV_d=3} for \emph{nonuniform} lattices in $\SO(3,1)\cong\SL_2(\CC)$. Weaker results (generalisations of \eqref{limBV}) were previously obtained  by J. Pfaff in \cite{Pfafftors} and by the author in \cite[Section 6.5]{thesis}. We refer to the introduction of \cite{moi2} for more details and further questions, and to \cite{BV},\cite{CV} and \cite{Scholze} for information on the number-theoretical significance of torsion homology of congruence subgroups. 


\subsection{Analytic torsion and Cheeger-M\"uller equality}

The main tools used in \cite{BV} are the Ray-Singer analytic torsion $T(M_n;V)$ and the Cheeger-M\"uller theorem. Bergeron and Venkatesh prove that the limit 
\begin{equation} \label{limanBV}
\lim_{n\to\infty} \frac{\log T(M_n;V)}{\vol M_n} = t^{(2)}(V) 
\end{equation}
holds, where the right-hand side $t^{(2)}(V)$ is the $L^2$-torsion associated to the representation $(\rho,V)$. In the case $m=3$, we have that $\SO(3,1)$ is isogenous to $G=\SL_2(\CC)$, and the real representations of the latter are given by its natural action on the spaces
\[
V(n_1, n_2) = \sym^{n_1}\left(\CC^2\right)\otimes\sym^{n_2}\left(\ovl{\CC^2}\right), \, n_1,n_2\in\NN 
\]
(where $\ovl{\CC^2}$ means that the action of $\SL_2(\CC)$ is by conjugate matrices). For $V = V(n_1, n_2)$ Bergeron and Venkatesh compute the numerical value of $t^{(2)}$ to be:
\begin{equation}
t^{(2)}(V)=\frac{-1}{48\pi} \left( (n_1+n_2+2)^3-|n_1-n_2|^3+3|n_1-n_2|(n_1+n_2+2)(n_1+n_2+2-|n_1-n_2|) \right).
\label{formuleL2}
\end{equation}

On the other hand, W. M\"uller's generalisation \cite{Muller} of the Cheeger--M\"uller Theorem (a more general result was proven independently by J.M.~Bismut and W.~Zhang in \cite{Bismut_Zhang}) yields that 
\[
T(M_n;V)=\prod_{p=0}^m|H_p(\Gamma_n,V_\ZZ)_\tors|^{(-1)^p}
\] 
from which \eqref{limBV} follows at once since the $L^2$-torsion $t^{(2)}(V)$ is positive for $m=3\pmod 4$ and negative for $m=1\pmod 4$; to deduce \eqref{limBV_d=3} when $m=3$ one needs to study independently the torsion in $H_0$ and $H_2$. One of the issues in \cite{7S} is then to prove that \eqref{limanBV} holds under weaker conditions than those of \cite{BV} and that these conditions are satisfied by sequences of congruence subgroups. Following the work of I. Benjamini and O. Schramm on graphs the notion of Benjamini-Schramm convergence of Riemannian manifolds is defined there (see \ref{BS} below) and it is then a relatively easy matter to show that the proofs of \cite{BV} extend to this setting. Note that the first step of the proof outlined above is purely differential-geometric and does not use the arithmeticity of the manifolds. 


\subsection{Approximation for regularised analytic torsion}

The first goal of the present paper is to define an analytic torsion for non-compact, finite-volume hyperbolic three--manifolds and to prove a generalisation of \eqref{limanBV} in this context. The definition of the regularised analytic torsion $T_R(M_n;V)$ is based on the Selberg trace formula; it is essentially the same torsion as that defined in \cite{MP} (but see \ref{diff_MP} for some comments on the differences). The definition depends on a choice of parametrisations (which we call `height functions' on $M_n$---see \ref{height_quotient}) for the cusps of $M_n$ as $T_j\times[1,+\infty[$ where the $T_j$ are flat tori. Let $M_n$ be a sequence of finite-volume hyperbolic three--manifolds; the conditions we need to prove approximation of the $L^2$-torsion are as follows:
\begin{itemize}
\item Geometric conditions:
   \begin{enumerate}
   \item \label{cond_BS} The sequence $(M_n)$ is Benjamini--Schramm convergent to $\HH^3$ (see \ref{BS}); 
   \item \label{cond_sys} We suppose that there is a $\delta>0$ such that $\sys(M_n)$ (the smallest length of a closed geodesic on $M_n$) is larger than $\delta$ for all $n$. 
   \item \label{cond_cusp} Some kind of regularity for the cusps: in this introduction we will take this to mean that the sequence be cusp-uniform (i.e. the cross-sections $T_j$ of the cusps of all lie in a fixed compact subset of the set of Euclidean tori up to similarity), but this can be relaxed a little (see \eqref{square!}in the statement of Theorem \ref{conv1}).  
   \end{enumerate}
\item Analytic assumptions:
   \begin{enumerate}[resume]
   \item \label{cond_gap} As in \cite{BV}, we need to use coefficients systems that induce a uniform spectral gap for all hyperbolic manifolds (such coefficients systems are called strongly acyclic; $V(n_1, n_2)$ is so exactly when $n_1\not= n_2$: see Proposition \ref{strongacyclicity} or \cite[Lemma 4.1]{BV}); 
   \item \label{cond_entr} In addition, to deal with the continuous spectrum we need to assume that the derivatives of the intertwining operators are well-behaved near the origin, namely that their trace be an $o(\vol M_n)$ uniformly in a neighbourhood of 0. 
   \end{enumerate}
\item A normalisation condition for the height functions (we emphasise that this is really not of the same nature as the other conditions and should be seen as specifying the range of height functions for which we can expect approximation results):
   \begin{enumerate}[resume]
   \item \label{cond_normheight} We suppose that $\sum_j \left| \log(\inj(T_j)) \right| = o(\vol M_n)$.  
   \end{enumerate}
\item We need also to choose lifts to $\SL_2(\CC)$ of the holonomies $\pi_1(M_n)\to\PSL_2(\CC)$; while our results are valid without assumptions on these (see \ref{non_unip_rem}), in this paper we will work under the following hypothesis 
   \begin{enumerate}[resume]
   \item \label{cond_unip} The lifts of all peripheral elements (i.e. elements in the image of maps $\pi_1(T_j)\to\pi_1(M)$) are unipotent (equivalently the image of $\pi_1(M)$ in $SL_2(\CC)$ does not contain an element with trace $-2$). 
   \end{enumerate}
\end{itemize}
Our first main result is Theorem \ref{Main1}, which can be stated as follows. 

\begin{theostar} \label{Main_intro1}
Let $V$ be a strongly acyclic representation of $G = \SL_2(\CC)$ and $\Gamma_n$ as sequence of torsion-free lattices in $G$. We suppose that the manifolds $M_n=\Gamma_n\bs\HH^3$ are endowed with height functions and satisfy \ref{cond_BS}, \ref{cond_sys}, \ref{cond_cusp}, \ref{cond_entr} and \ref{cond_normheight} and that $\Gamma_n$ satisfies \ref{cond_unip}. Then we have
\begin{equation}
\lim_{n\to\infty}\frac{\log T_R(M_n;V)}{\vol M_n} = t^{(2)}(V).
\end{equation}
\end{theostar}

Note that it is easily shown that for any given finite-volume hyperbolic three--orbifold there are sequences of finite covers which satisfy the assumptions \ref{cond_BS} and \ref{cond_cusp} above (see Proposition \ref{cuex}), but we will not check that \ref{cond_entr} holds for them in the present paper (it will be proven to hold for sequences of congruence covers of arithmetic orbifolds in \cite{moi2}). Conditions \ref{cond_entr}, \ref{cond_normheight} (unlike the others) depend on the choice of height functions on the $M_n$; however whether \ref{cond_entr} holds or not does not depend on this choice in the range of height functions such that \ref{cond_normheight} holds (see the remark after Theorem \ref{Main1}). Finally, if we consider a sequence of finite covers $M_n$ of a fixed orbifold $M$ then the natural height functions to use on the $M_n$ are the pull-back of those on $M$, and it is very easily seen that they satisfy \ref{cond_normheight} (see Lemma \ref{sumalpha}). 

Let us say a few more informal words about the necessity of these conditions: \ref{cond_BS} is necessary (there are sequences of covers where one can see that the torsion has an exponential growth with a different rate); \ref{cond_cusp} may or may not be (there are sequences of congruence covers which do not satisfy it, but we do not know whether approximation for the analytic torsion holds in these); \ref{cond_sys} is very likely necessary (one can make the torsion vary arbitrarily by doing Dehn surgeries on a given manifold). Condition \ref{cond_entr} was shown by J. Pfaff and W. M\"uller to always hold for sequences of covers (cf. \cite{MP2}, whose prepublication was posterior to the first submission of the present paper), but in general not much more is known; \ref{cond_gap} is likely not necessary for covers, but it is in general (see \cite{BD}). Of course \ref{cond_normheight} is necessary since for a given manifold the regularised analytic torsion can take arbitrarily large values if one does not put limitations on the height functions one wants to consider.


\subsection{An asymptotic Cheeger--M\"uller equality}

The next step in adapting Bergeron and Venkatesh's argument to the case of non-compact manifolds is to relate the regularised analytic torsion to a combinatorial, or Reidemeister torsion (the latter is named after K. Reidemeister who was one of the first to study this kind of invariants, for somewhat different purposes). In this paper we do not define such a torsion in an intrisic way for a non-compact hyperbolic manifold $M$ of finite volume (this is carried out in \cite{CV}, see also \cite{moi2}); we will instead use the truncated manifold $M^Y$, which are obtained by `cutting off the cusps' of $M$ using a parameter $Y$ (see \eqref{tron} for the definition). Thus $M^Y$ is a compact manifold with boundary, for which analytic and Reidemeister torsion are well-defined and the Cheeger--M\"uller equality is known---see \ref{sec_abs}. Our second main result is then the following (Theorem \ref{CMA}). 

\begin{theostar} \label{MainIntro2}
Suppose that $M_n$ and $V$ are as in the statement of the previous theorem and that the sequence $(M_n)$ satisfies the additional condition that
\[
h_n\ll \frac{\vol M_n}{\log(\vol M_n)^{20}}
\]
(where $h_n$ is the number of cusps of $M_n$), then there exists a sequence $Y^n\in[1,+\infty[^{h_n}$ such that we have
\begin{equation}
\lim_{n\to\infty}\frac{\log T_R(M_n;V) - \log\tau_\abs(M_n^{Y^n},V)}{\vol M_n} = 0.
\end{equation}
\end{theostar}

An explicit formula for $Y^n$ is given in the statement of Theorem \ref{Main2}. Note that the sequences constructed in Proposition \ref{cuex} satisfy also the stronger assumption in this theorem. 


\subsection{Betti numbers}

The behaviour of the characteristic 0 homology in BS-convergent sequences of non-compact hyperbolic manifolds is not dealt with in \cite{7S}. For three--manifolds we prove the following result. 

\begin{propstar} \label{betintro}
Let $M_n$ be a sequence of finite-volume hyperbolic three-manifolds and suppose that $M_n$ BS-converges to $\HH^3$. Then we have for $p=1,2$
\begin{equation*}
\frac{b_p(M_n)}{\vol(M_n)}\xrightarrow[n\to\infty]{} 0 
\end{equation*}
\end{propstar}

This limit is well-known for exhaustive sequences of covers as follows for example from M. Farber's generalisation \cite{Farber} of L\"uck's theorem \cite{Luck2} (applied to the manifolds truncated at 1). 

We will give two proofs of this: one which uses the techniques in this paper, and which consequently needs the assumption that the sequence $M_n$ satisfies the condition \eqref{square!}, and then a proof in all generality using Thurston's hyperbolic Dehn surgery and the results of \cite[Section 9]{7S}. The second proof does not generalise to higher dimensions but the first one does (after modifying \eqref{square!} adequately). We will perhaps return to this in the broader setting of $\QQ$-rank one lattices of semisimple real Lie groups in the future. 


\subsection{Outline of the proofs}

\subsubsection{Convergence of finite-volume manifolds, regularised traces and Betti numbers}

In \cite[Definition 1.1]{7S} the notion of Benjamini--Schramm convergence of locally symmetric spaces to their universal cover is introduced, and a good part of the paper studies the implications of this notion for compact manifolds. In this work we extend some of these results to nonuniform lattices in $\SL_2(\CC)$ (see Section \ref{BS}). Let us remind the reader that Benjamini--Schramm convergence (to the universal cover $\HH^3$) is an interpolation between the weaker pointed Gromov--Hausdorff convergence and the stronger condition that the global injectivity radius goes to infinity. It is conveniently summarised by saying that ``the injectivity radius goes to infinity at almost all points''; formally, for a sequence $M_n$ of finite--volume hyperbolic three--manifolds to be convergent to $\HH^3$ we require that for all $R>0$ the sequence $\vol\{x\in M_n:\: \inj_x M_n\le R\}$ be an $o(\vol M_n)$. 

The regularised trace $\otr_R(K)$ of an automorphic kernel $K$ on a finite-volume manifold $M$ is defined by taking either side of a very unrefined form of the trace formula for $K$, of which we give a mostly self-contained proof---minus the theory of Eisenstein series, which we review in \ref{Eisenstein}---in Section \ref{trace}. The study of the geometric side in Benjamini-Schramm convergent sequences is not very hard and results in Theorem \ref{conv1}; note however that we need an additional condition on the geometry of the cusps to prove the convergence of the unipotent part. We prove, using comparisons of traces with the truncated manifolds, that the Betti numbers in a BS-convergent sequence are sublinear in the volume in Proposition \ref{betprop} (we cannot deduce it directly from Theorem \ref{conv1} since we did not manage to control the non-discrete part of the spectral side of the trace formula in general). On the other hand, to study Betti numbers in dimension three one can bypass all this by using \cite[Theorem 1.8]{7S} and hyperbolic Dehn surgery.


\subsubsection{Analytic torsions}

Our definition of analytic torsion for cusped manifolds is the same as in \cite{Park} or \cite{MP} (we could have just quoted the results of the latter but we use a slightly different method to prove the asymptotic expansion of the heat kernel which is better suited to the rest of this paper). Let $M$ be a finite-volume manifold and $K_t^p$ its heat kernel on $p$-forms (we will suppose here that the coefficients are in a strongly acyclic bundle, but with more work one can see that the definition carries over to the general case, see \cite{Park},\cite{MP}). One defines the analytic torsion as in the compact case, by putting: 
\[
T_R(M) = \sum_{p=0}^3 p(-1)^p \frac d{ds}\left(\frac 1{\Gamma(s)}\int_0^{t_0}\otr_R(K_t^p)t^s\frac{dt}t\right)_{s=0} + \int_{t_0}^{+\infty} \otr_R(K_t^p)\frac{dt}t
\]
which does not depend on $t_0 > 0$. The justification of this definition uses meromorphic continuation and is highly nontrivial, see \ref{torsion} below or \cite{MP} for the  details needed to ensure the convergence of the integrals and their analytic continuation. In a sequence of manifolds we study the first summand using the geometric side of the trace formula and the second one using the spectral side, as in \cite[Section 4]{BV}. The spectral side is dealt with using the uniform spectral gap property established there; however the part coming from the continuous spectrum causes some additional difficulty which explains the conditionality of our approximation on the hypothesis \ref{cond_entr} on intertwining operators which we were not able to check for general sequences. The study of the geometric side is actually quite simple once the asymptotic expansion for $K_t^p$ at $t\to 0$ has been established (see Proposition \ref{dea}) using our unrefined trace formula. We remark that in \cite{thesis} we dealt with these problems in the more general context of finite-volume hyperbolic good orbifolds---the elliptic terms in the trace formula do not cause any real additional difficulty. 

We also show that under hypotheses (very) slightly more restrictive as for the approximation of analytic torsion there is an asymptotic equality between absolute analytic torsion for the truncated manifold $M^Y$ and regularised analytic torsion for the complete manifold, cf. Theorem \ref{Main2} below. As in the proof of the approximation result we separate into small and large times. We deal with the small-time part in Section \ref{CMA_small}, where we use estimates on the integral of automorphic kernels over the truncated manifolds and a result of W. L\"uck and T. Schick \cite{LS}; for this part we also need to extend the well-known Gaussian bound for the heat kernel (proven for example in \cite[Section 5]{RS}) to the case of the universal covers of truncated manifolds; we explain how to adapt the arguments from loc. cit. in Appendix \ref{heat}. The large times are taken care of in Section \ref{CMA_large}; the main point in the proof is to control the spectral gap for the truncated manifolds (Proposition \ref{ver}) and this is achieved using techniques inspired from F. Calegari and A. Venkatesh \cite[Chapter 6]{CV}.


\subsubsection{Asymptotic Cheeger-M\"uller theorem and homology growth}

In contrast with the compact case, for our coefficient systems there is usually a nontrivial homology in characteristic 0. Thus, to state and hopefully prove a Cheeger-M\"uller-type equality one needs to define a suitable Reidemeister torsion. This is done by F. Calegari and A. Venkatesh in \cite{CV}, in a manner similar to the regularisation for traces of integral operators. Thus a natural way to prove such an equality would be to apply the Cheeger-M\"uller equality for manifolds with boundary \cite{Bruening_Ma1},\cite{Luck} to the truncated manifolds and to compare both sides with their regularised analogue. 

Here we deal only with the first part of this program, we refer to \cite{moi2} for the applications of the results in the present paper to congruence subgroups and their homology growth. From the asymptotic equality of analytic torsions (Theorem \ref{Main2}) it is not hard to deduce an asymptotic equality with the absolute Reidemeister torsion of the truncated manifold using a recent generalisation by J. Br\"uning and X. Ma of the Cheeger-M\"uller theorem, see Theorem \ref{CMA}.


\subsection{Remarks}

\subsubsection{Non-unipotent holonomies}
\label{non_unip_rem}
In the case where condition (vii) on the holonomies of peripheral subgroups is not satisfied both Theorem \ref{Main_intro1} and \ref{MainIntro2} still hold. To prove this one must consider two cases depending on whether $n_2-n_1$ is even or odd. In the first case the representation $\SL_2(\CC)\to\SL(V(n_1,n_2))$ factors through $\PSL_2(\CC)$ and it makes no difference whether or not (vii) holds. When $n_1-n_2$ is odd the heat kernels become integrable in the `bad' cusps whose fundamental group has an holonomy containing elements of trace $-2$ (note that if all cusps are such, the heat kernel is in fact trace-class). The parabolic summand for the trace formula in Theorem \ref{Selberg} changes a bit (see \cite[3.5]{thesis}), but the estimates used all along the proofs in \ref{sec:small_time_approx} and \ref{CMA_small} can still be used. The proof of Proposition \ref{ver} still holds since in the bad cusps the eigenfunctions decay exponentially. 

\subsubsection{Related recent results}

In addition to the papers \cite{BV} and \cite{7S} from which this work originates there have been other papers dealing with similar problems. There has been a number of papers studying the asymptotic behaviour of analytic torsion of a compact manifold as the coefficient systems varies. This was done independently and concurrently, with different methods, on the one hand by W. M\"uller and J. Pfaff (starting with \cite{Muller2}) and on the other, in a more general setting, by J.M. Bismut, X. Ma and W. Zhang \cite{BMZ}. This has been extended to the noncompact setting (based on the work of M\"uller--Pfaff) by P. Menal-Ferrer and J. Porti \cite{MFP} and by W. M\"uller and J. Pfaff \cite{MP}. 

Cheeger--M\"uller type equalities for manifolds with cusps (and more general singularities) have attracted a lot of interest recently. Let use cite some papers which are close to our topic here: \cite{Pfaff}, \cite{Pfaff2}, \cite{Albin_Rochon_Sher}. 

\subsubsection{Analytic torsion here and in \cite{MP}}
\label{diff_MP}

Though we use the same definition of analytic torsion as W. M\"uller and J. Pfaff do in \cite{MP}, there is a slight difference in setting between their paper and ours, which we will explain here. In the present work, one starts from an hyperbolic manifold $M$ and assign it an arbitrary parametrisation of the cuspidal components of its thin part; if $M$ has finite hyperbolic volume we use these functions to derive a trace formula which is then used to define the regularised analytic torsion. In \cite{MP}, one starts from a lattice $\Gamma$ in $\SL_2(\CC)$, chooses representatives for the conjugacy classes of parabolics and then defines height functions on the quotient by choosing a point in $\HH^3$ (the fixed point of $\SU(2)$) and assigning to it height 1  for all the parabolics. Then \cite{MP} use existing forms of the trace formula to define the regularised analytic torsion, which is the same as the one we define here using these particular height functions. 

\subsubsection{About \cite{thesis}}

As noted above, in the Ph.D. thesis of the author some of the problems here were tackled in greater generality, rendering assumption (vii) unnecessary and also dealing with orbifolds. However, there are some very embarassing (to the author) and serious gaps in this manuscript (especially in a previous version of Proposition \ref{ver}), which nevertheless do not affect the validity of the results we quote (and which are filled in the present work). 


\subsection*{Acknowledgments} A first version of this paper was written while I benefited from a doctoral grant from the Universit\'e Pierre et Marie Curie (Paris 6). The present version was written while I was a post-doc at the Max-Planck Institut f\"ur Mathematik in Bonn. 

The reading of a preliminary version of \cite{CV} has been extremely profitable for the writing of this paper and I want to thank the authors for allowing me to read it. During the redaction I became more and more permeated by the point of view of Benjamini--Schramm convergence introduced in the joint work (with Mikl\'os Ab\'ert, Nicolas Bergeron, Ian Biringer, Tsachik Gelander, Nikolay Nikolov and Iddo Samet) \cite{7S}. I also benefited greatly from a week spent in Bonn with Werner M\"uller and Jonathan Pfaff, whose comments on previous versions of this paper were especially useful and thorough, and who pointed out a serious gap in a previous approach to Proposition \ref{terme3}. A pair of anonymous referees provided helpful suggestions for improving the presentation and spotted numerous mistakes. Last but not least I want to thank my Ph.D. advisor, Nicolas Bergeron, under whose supervision this work was conceived.


\section{Hyperbolic manifolds and Benjamini--Schramm convergence}

\label{hyp_mfd}
Let $G=\SL_2(\CC)$, so that $K=\SU(2)$ is a maximal compact subgroup and the Riemannian symmetric space $G/K$ is isometric to hyperbolic three--space $\HH^3$, which we will identify here with the Poincar\'e half-space $\CC\times]0,+\infty[$ endowed with the Riemannian metric given by $\frac{dzd\bar{z}+dy^2}{y^2}$ in coordinates $(z,y)$. 

\subsection{Height functions on $\HH^3$}

Define the following subgroups of $G$:
\begin{gather*}
P_\infty=\left\{\begin{pmatrix} a&b\\0&a^{-1}\end{pmatrix},\, a\in\CC^{\times},b\in\CC\right\} , \quad 
N_\infty=\left\{\begin{pmatrix} 1&b\\0&1\end{pmatrix},\,b\in\CC\right\}, \\
A_\infty = \left\{\begin{pmatrix} a&0\\0&a^{-1}\end{pmatrix},\, a\in\RR_+^{\times}\right\}, \quad 
M_\infty = \left\{\begin{pmatrix} e^{i\theta}&0\\0&e^{-i\theta}\end{pmatrix},\, \theta\in[0,2\pi]\right\}.
\end{gather*}
The proper parabolic subgroups of $G$ are the conjugates of $P_\infty$. Let $P=gP_\infty g^{-1}$ be such a subgroup and $N,A,M$ the conjugates of $N_\infty,A_\infty,M_\infty$ by $g$. We call any function that is conjugated by $g$ to the function $\begin{pmatrix} a&0\\0&a^{-1}\end{pmatrix} \mapsto a^2$ on $A_\infty$ a norm on $A$. We have the Langlands decomposition $P=NAM=MAN$ and the Iwasawa decomposition $G=NAK$. A height function on $\HH^3$ at $P$ is then defined to be any function of the form $gK\mapsto |a|$ where $g=nak\in NAK$ and $|.|$ is any norm on $A$. (as an illuminating example take $P=P_{\infty}$, then the height functions at $P$ are of the form $(z,y)\mapsto ty$ for $t\in\RR_+^{\times}$). 

The level sets of a height function at $P$ are called horospheres through $P$; they are isometric to the Euclidean plane $\CC$ and are acted upon simply transitively by the subgroup $N$. Let $y_P$ be a height function at $P$; we may identify $N$ with $\{y_P=1\}\cong\CC$ and we denote by $|n|$ the induced length function on $N$. If we normalise the Haar measure $dn$ on $N$ so that it is the pullback of the Lebesgue measure on $\CC$, then the volume form of $\HH^3$ is equal to $dndy_P/y_P^3$. For $x\in\HH^3$ the quotient $|n|/y_P(x)$ does not depend on the choice of $y_P$ and we have the following estimate for the translation length of unipotent elements. 

\begin{lem} \label{eval}
There exists a function $\ell:[0,+\infty)\rightarrow[0,+\infty)$ such that 
\begin{equation} \label{pointmajo}
d(x,nx) = \ell \left( \frac{|n|}{y_P(x)} \right) 
\end{equation}
for all parabolics $P=MAN$, $n\in N$ and $x\in\HH^3$. Moreover $\ell(r)\gg \log(1+r)$.
\end{lem}

\begin{proof}
We give a very awkward but very explicit proof. Obviously it suffices to prove the lemma for $P = P_\infty$; for $n = \begin{pmatrix} 1 & z \\ 0 & 1 \end{pmatrix}\in N_{\infty}$ we may take $|n| = |z|$. Let $x \in \HH^3$, $y = y_\infty(x)$. The formula \cite[Corollaire A.5.8]{BP} yields 
\begin{equation} \label{dist} 
d(x,nx) = 2\left(\log\left(1+\sqrt{\frac{|n|^2}{|n|^2+4y^2}}\right)-\log\left(1-\sqrt{\frac{|n|^2}{|n|^2+4y^2}}\right)\right)
\end{equation}
so that $d(x,nx) = \ell(|n|/y)$ where we put 
\[
\ell(r)=2\left(\log\left(1+(1+(r/2)^{-2})^{-\frac  12}\right)-\log\left(1-(1+(r/2)^{-2})^{-\frac 1 2}\right)\right)
\]
It remains to check that $\ell(r) \gg \log(1+r)$: the first summand is in $[0, \log(2)]$, and besides for $t \in [0, +\infty)$ one has $(1 + t^2)^{-1/2} \ge (1+t)^{-1}$ so that 
\[
-\log(1 - (1+t^{-2})^{-1/2}) \ge -\log(1 - (1+t^{-1})^{-1}) = \log(1 + t).
\]
from which the conclusion follows at once.
\end{proof}


\subsection{Height functions on hyperbolic three--manifolds} \label{height_quotient}

Let $\Gamma$ be a lattice in $G$ (i.e. $\Gamma$ is discrete and  $\Gamma \backslash G$ carries a finite right-$G$-invariant Borel measure). Given a parabolic subgroup $P$ we put $\Gamma_P = \Gamma\cap P$ and we say that $P$ is $\Gamma$-rational if $\Gamma_P$ contains a subgroup isomorphic to $\ZZ^2$ (equivalently $\Gamma\cap N$ is cocompact in $N$). Then $\Gamma$ is cocompact if and only if there are no $\Gamma$-rational parabolics (equivalently if $\Gamma$ contains no unipotent elements). In any case there are finitely many $\Gamma$-conjugacy classes of $\Gamma$-rational parabolics. We may thus choose representatives $P_1,\ldots,P_h$ for these classes and height functions $y_{P_1},\ldots,y_{P_h}$ at each one of then, and define a function $y_j$ on $\HH^3$ by 
\[
y_j(x) = \max_{\gamma\in\Gamma/\Gamma_{P_j}} y_{P_j}(\gamma^{-1} x)
\]
which we call a $\Gamma$-invariant height function (and which is, indeed, $\Gamma$-invariant). If $\Gamma$ is torsion-free let $M$ be the manifold $\Gamma\backslash\HH^3$ and for $Y\in(0,+\infty)^h$ put :
\begin{equation} \label{tron}
M^Y = \{x\in M :\: \forall j=1,\ldots,h \text{ we have } y_j(x)\le Y_j\}. 
\end{equation}
Then for $Y$ large enough (depending on the choice of the original height functions $y_{P_j}$) $M^Y$ is a compact manifold with boundary a union of flat tori $T_j,\, j=1,\ldots,h$. The ends $\{x\in M,\, y_j(x)\ge Y_j\}$ are isometric to the warped products $T_j\times(Y_j,+\infty)$ with the metrics $\frac{dx^2+dy_j^2}{y_j^2}$ where $dx^2$ is the euclidean metric on $T_j$. \emph{In this paper we will work under the following convention:} we always suppose that the height functions are normalised so that the maps $\Gamma_{P_j}\backslash \{ y_{P_j}\ge 1 \}\to M$ are embeddings (in particular, the horospheres of height one are disjoint). 

Finally, if $\Gamma'\subset\Gamma$ is a finite-index subgroup the $\Gamma$-invariant height functions are $\Gamma'$-invariant; when dealing with a sequence of finite covers of a given manifold (or orbifold) we will always suppose that the height functions on the covers come from those of the covered manifold.


\subsection{Euclidean lattices} \label{euclidien}

Let $\Lambda$ be a lattice in $\CC$ ; we denote by $\vol(\Lambda)$ its covolume (i.e. the volume of a fundamental parallelogram) and define
\[
\alpha_1(\Lambda) = \min\{|v| \: :\: v\in\Lambda, v\not=0 \}
\]
and for any $v_1\in\Lambda$ such that $|v_1|=\alpha_1(\Lambda)$
\[
\alpha_2(\Lambda) = \min\{|v| \: :\: v\in\Lambda, v\not\in \ZZ v_1 \}.
\]
Then the ratio $\alpha_2(\Lambda) / \alpha_1(\Lambda)$ depends only on $\Lambda$ up to similarity. We denote by $\mathcal{N}_\Lambda(r)$ the number of points in $\Lambda$ of absolute value less than $r$ and
\[
\mathcal{N}_{\Lambda}^*(r)=\mathcal{N}_{\Lambda}(r) - 1 = |\{ v \in \Lambda \setminus \{ 0 \} : |v| \le r \}|. 
\]
We will use $\mathcal{N}^*$ rather than $\mathcal{N}$ further on; moreover we get a cleaner bound in the lemma below. The following estimate for the counting function was proven by Gauss; we include a proof only for the reader's convenience and because we need a precise statement with regard to the constants.

\begin{lem} \label{ucount}
Define :
\[
E_\Lambda (r) := \mathcal{N}_{\Lambda}^*(r)-\frac{\pi r^2}{\vol(\Lambda)}.
\] 
For any lattice $\Lambda\subset\CC$ we have the estimate
\begin{equation} \label{count}
|E_\Lambda (r)| \ll \frac r {\alpha_1(\Lambda)} + \frac{\alpha_2(\Lambda)}{\alpha_1(\Lambda)}
\end{equation}
where the constant does not depend on $\Lambda$. 
\end{lem}

\begin{proof}
First we consider $r< \alpha_1(\Lambda)$ so that $E_\Lambda(r) = \pi r^2/\vol(\Lambda)$. By Minkowski's First Theorem, if $\pi r^2\ge 4\vol(\Lambda)$ then $\Lambda$ contains a nonzero vector of length $\le r$, which implies that the quotient $\vol(\Lambda)/\alpha_1(\Lambda)^2 \ge \pi/4$. Thus $E_\Lambda(r)\ll (r/\alpha_1(\Lambda))^2\le r/\alpha_1(\Lambda)$. 

Now suppose that $r\ge \alpha_1(\Lambda)$. We can choose a fundamental parallelogram $\Pi$ for $\Lambda$ whose diameter $d$ is $\asymp \alpha_2(\Lambda)$ and by Minkowski's second theorem we have $\vol(\Lambda) \asymp \alpha_2(\Lambda)\alpha_1(\Lambda)$. For $r\ge d$ let $z_1,\ldots,z_n$ be all the points in $\Lambda$ such that $|z_k|\le r$, then we have that $B(0,r-d)\subset \bigcup_k z_k+\Pi\subset B(0,r+d)$ so that $\pi(r-d)^2 \le \vol(\Pi)n \le \pi (r+d)^2$. It follows that:
\begin{align*}
\left|\mathcal{N}_\Lambda(r)-\frac{\pi r^2}{\vol(\Lambda)}\right| &\le \frac{d^2}{\vol(\Lambda)} + \frac{2rd}{\vol(\Lambda)} \ll \frac{\alpha_2(\Lambda)^2}{\vol(\Lambda)} + \frac r{\alpha_1(\Lambda)} \\
       & \ll \frac{\alpha_2(\Lambda)}{\alpha_1(\Lambda)} + \frac r{\alpha_1(\Lambda)}
\end{align*} 
which finishes the proof of \eqref{count}. 
\end{proof}

We say that a set $S$ of euclidean lattices is {\em uniform} if there exists a $C>0$ such that 
\begin{equation} \label{systole}
\forall \Lambda\in S, \: \vol(\Lambda)\le C\alpha_1(\Lambda)^2 
\end{equation}
By Mahler's criterion this is equivalent to asking that when we normalise the lattices in $S$ so that they are unimodular they form a relatively compact set in $\SL_2(\RR)/\SL_2(\ZZ)$. If $\Lambda$ belongs to a uniform set $S$ then the proof above yields that
\begin{equation} \label{?}
|E_\Lambda(r)| \ll \frac r{\alpha_1(\Lambda)}
\end{equation}
with a constant depending only on $S$. 


\subsection{Cusp-uniform sequences}
\label{cusp-unif}

If $P$ is a parabolic subgroup and $y_P$ is a height function at $P$ then we may identify the unipotent radical $N$ of $P$ with the horosphere $y_P=1$ and the conformal structure on $N$ thus obtained does not depend on the chosen $y_P$. Since the uniformity of a set only depends on the conformal structures of its elements we may define a cusp-uniform sequence as a sequence of lattices $\Gamma_n\subset G$ such that the set 
\[
\{ (\Gamma_n)_P, \, n\ge 1,\, P \text{ is a $\Gamma_n$-rational parabolic} \}
\]
is a uniform set of euclidean lattices. The following result gives a source of examples satisfying some the geometrical conditions of our main results. 

\begin{prop} \label{cuex}
Let $\Gamma\subset G$ be a lattice, then there exists a cusp-uniform sequence $\Gamma_n\subset\Gamma$ which exhausts $\Gamma$ and satisfies in addition that $h_n\ll (\vol M_n)^{1-\delta}$ for some $\delta>0$. 
\end{prop}

\begin{proof}
It is well-known that up to conjugation we may assume $\Gamma\subset\SL_2(F)$ for some number field $F$. Let $\so_F$ be the ring of integers of $F$; as $\Gamma$ is finitely generated there exists an $a\in\mathcal{O}_F$ such that $\Gamma\subset \SL_2(\mathcal{O}_F[a^{-1}])$. For an ideal $\fri\subset\mathcal{O}_F$ coprime to $a$ we may define $\Gamma(\fri)$ as the set of matrices in $\Gamma$ congruent to the identity modulo $\fri$. Then the sequence of $\Gamma(n)$ for $n\in\NN$ coprime to $a$ is clearly exhaustive and we claim that it is cusp-uniform. Indeed, if $P$ is a $\Gamma$-rational parabolic we have $\Gamma_P=1+\mathbb{Z}X_1+\mathbb{Z}X_2$ for some $X_1,X_2$ in the Lie algebra $\LSL_2(\mathcal{O}_F[a^{-1}])$. Let $\fri$ be the ideal in $\mathcal{O}_F[a^{-1}]$ generated by the entries of $X_1$ and $X_2$ and $m$ the unique positive rational integer such that $\fri\cap\ZZ=m\ZZ$. Put $\Lambda_n=n\Gamma_P$; then the $\Lambda_n$ are a uniform family of lattices in $N$ and we have $\Lambda_n\subset\Gamma(n)_P\subset m^{-1}\Lambda_n$, so that $\{\Gamma(n)_P, n\}$ is uniform as well. Since the subgroups $\Gamma(n)$ are normal in $\Gamma$ we need only consider a finite number of $P$ and the claim of cusp-uniformity follows. 

For all $\Gamma$-rational parabolic $P$ we have 
\[
[\Gamma_P : \Gamma(n)_P] \ge C[\Lambda_1:\Lambda_n] \ge Cn^2. 
\]
On the other hand, if $P_1,\ldots,P_{h_n}$ are representants for the conjugacy classes of $\Gamma(n)$-rational parabolics we have
\[
[\Gamma:\Gamma(n)] = \sum_{j=1}^{h_n} [\Gamma_{P_j}:\Gamma(n)_{P_j}] \ge Ch_nn^2
\]
We have finally $[\Gamma:\Gamma(n)]\le |\SL_2(\so_F/n)| \le n^{3[F:\QQ]}$ and it follows that $n^2 \ge [\Gamma:\Gamma(n)]^\delta$ for some $\delta>0$ (depending on $F$) so that we get $h_n\ll [\Gamma:\Gamma(n)]^{1-\delta}$. 
\end{proof}


\subsection{Some counting lemmas in hyperbolic space}

For this subsection we always denote by $\eps$ the Margulis constant for $\HH^3$. If $\Gamma$ is a finitely generated, discrete subgroup of $G$ we let $\sys(\Gamma)$ denote the systole of $\Gamma\bs\HH^3$, i.e. the smallest translation length of a loxodromic element in $\Gamma$. 

\subsubsection{Orbits}

The following lemma is well-known in the case of groups containing no unipotent isometries, but needs a slight modification to incorporate the general situation. 

\begin{lem}\label{gen_count}
There is an absolute constant $c>0$ such that the following holds: let $\Gamma$ be a torsion-free discrete subgroup in $G$. Let $x\in\HH^3$ and let $\Lambda$ be the subgroup of $\Gamma$ generated by the elements in $\Gamma$ which commute with a unipotent $\eta\in\Gamma$ such that $d(x,\eta x)\le\eps$ (thus $\Lambda$ is a free abelian group of rank $\le 2$). Then there is a $C>0$ depending only on $\sys(\Gamma)$ such that for all $r>0$ we have:
\begin{equation} \label{here_count}
|\{ \gamma\in\Gamma-\Lambda :\: d(x,\gamma x) \le r\}| \le Ce^{cr}
\end{equation}
\end{lem}

This implies in particular the following: for a discrete subgroup $\Gamma$ in $G$ we let $\Gamma_\lox$ be the set of loxodromic elements in $\Gamma$ and for any $x\in\HH^3$ and $r>0$ put
\begin{equation} \label{def_count_lox}
\mathcal{N}_{\Gamma}(x,r) = |\{\gamma\in\Gamma_\lox,\, d(x,\gamma x)\le r\}|.
\end{equation}
Then there is a constant $C$ depending only on the systole of $M$ such that:
\begin{equation} \label{loxcount}
\mathcal{N}_{\Gamma}(x,r) \le Ce^{cr}.
\end{equation}

\begin{proof}[Proof of Lemma \ref{gen_count}]
We define:
\[
\delta = \delta(x, r) = \frac 1 2 \min \bigl( d(\gamma x,\gamma' x) :\: \gamma,\gamma'\in\Gamma-\Lambda,\, d(x,\gamma x), d(x,\gamma' x) \le r \bigr).   
\]
The balls $B(\gamma x,\delta)$ for $\gamma\in\Gamma-\Lambda,\, d(x,\gamma x)\le r$ are pairwise disjoint. Moreover their union is contained in the ball $B(x, r + \delta)$. It follows that the right-hand side in \eqref{here_count} is smaller than $V(r+\delta)V(\delta)^{-1}$ where $V(R)$ denote the volume of a ball of radius $R$ in $\HH^3$. We have $r^d\le V(r) \le e^{c_0r}$ for an absolute $c_0$ and thus the lemma follows from the claim that for any $x\in\HH^3$ we have $\delta\ge Ce^{-r}$ for some $C>0$ depending only on $\sys(\Gamma)$. 

To prove this we may suppose that $\Gamma x$ lies in a noncompact component of the $\eps$-thin part $M_{\le\eps}$ of $M=\Gamma\bs\HH^3$ (otherwise $\delta\ge \min(\eps,\sys(\Gamma))$). We let $H$ be the horosphere preserved by $\Lambda$ (defined as in the statement) lifting the component of $\pl M_{\le \eps}$ closest to $\Gamma x$, and we claim that $\delta(x,r) = +\infty$ for all $r\le d(x,H)$ and $\delta(x,r)\ge e^{-d(x,H)}$ for $r\ge d(x,H)$, from which the original claim immediately follows. To prove the newest claim we first observe that any $\gamma\in\Gamma-\Lambda$ must move $x$ outside of the horoball bounded by $H$, hence the first part of the claim. Now the Euclidean displacement of an element of $\Lambda$ on the horosphere through $x$ is at least $e^{-d(x,H)}\eps$ (since its displacement on $H$ is at least $\eps$), hence $\inj_x(M)\ge C'\eps e^{-d(x,H)}$ for an absolute $C'$ by Lemma \ref{eval}. We have $\delta(x,r)\ge\inj_x(M)$ for any $r$ and the second part of the claim follows. 
\end{proof}


\subsubsection{Horospheres}

We let ${\mathcal B}$ be a collection of disjoint (closed) horoballs in $\HH^3$ and $\sh$ be the collections of horospheres $\{\pl B:\: b\in{\mathcal B}\}$. For a point $w\in \HH^3$ and a radius $R>0$ we denote by $N_\sh(w,R)$ the number of horoballs in ${\mathcal B}$ which are at a distance smaller than $R$ from $w$. 

\begin{lem}\label{counthoro2}
There are absolute constants $C,c$ such that for any $\sh$ as above and any $w\in \HH^3$ we have:
\begin{equation} \label{counthoro}
{\mathcal N}_\sh(w,R) := |\{H\in\sh :\: d(w,H) \le R\}| \le Ce^{cR}
\end{equation}
\end{lem}

\begin{proof}
Let $V(R)$ denote the volume of a ball of radius $R$ in $\HH^3$. We claim that:
\begin{equation} \label{hsj}
N_\sh(w,R)\cdot V(1) \le V(R+2). 
\end{equation}
Since $V(R)\le e^{cR}$ for some $c$ the lemma follows. 

The proof of \eqref{hsj} is staightforward: let $N=N_\sh(w,R)$ and $B_1,\ldots,B_N\in\sh$ be the horoballs meeting the ball of radius $R$ around $w$. For $i=1,\ldots,N$ take a $x_i\in\pl B_i$ such that $d(x_i,w)\le R$ and let $x_i'$ be the point at distance 1 from $x_i$ along the inwards normal to $\pl B_i$ at $x_i$; finally, et $U_i$ be the ball of radius $1$ around $x_i'$. Then the balls $U_i$ are disjoint (since $U_i\subset B_i$ and the $B_i$ themselves are disjoint) and contained in the ball of radius $R+2$ around $w$; it follows that we have 
\[
V(R+2) \ge \sum_i \vol (U_i) = N\cdot V(1)
\]
which finishes the proof of \eqref{hsj}
\end{proof}


\subsection{Benjamini--Schramm convergence for manifolds with cusps} \label{BS}

Let $M=\Gamma\backslash\HH^3$ be an hyperbolic three--manifold and let $x\in M$. Pick an arbitrary lift $\wdt x$ of $x$ to $\HH^3$ and define
\[
\ell_x = \min\{d(\wdt x,\gamma\wdt x), \: \gamma\in\Gamma,\gamma\not= 1_G\} = 2\inj_x(M).
\]
For $R>0$ we define the following subset of $M$:
\[
M_{\le R} = \{x\in M_n: \ell_x \le R/2\}
\]
Recall from \cite{7S} that a sequence $M_n$ is said to converge to $\HH^3$ in the Benjamini--Schramm topology (hereafter abreviated as $M_n$ BS-converges to $\HH^3$) if for any $R>0$ we have 
\begin{equation} \label{defBS}
\frac {\vol \left( (M_n)_{\le R} \right)} {\vol (M_n)} \xrightarrow[n\to\infty]{} 0.
\end{equation}
A source of examples is given by sequences where the injectivity radius goes to infinity; for example $M_n=\Gamma_n\backslash\HH^3$ where $\Gamma_n$ is an exhaustive sequence of torsion-free finite-index subgroups of a lattice $\Gamma$ (a sequence $\Gamma_n\subset\Gamma$ is said to exhaust $\Gamma$ if any $\gamma\in\Gamma$ belongs to at most a finite number of the $\Gamma_n$). Another is given by sequences of congruence lattices (see \cite{7S},\cite{moi2}). It follows from Proposition \ref{cuex} that every hyperbolic three--manifold has a sequence of finite covers that is BS-convergent to $\HH^3$ and cusp-uniform. 

In the sequel we will always consider a sequence $M_n=\Gamma_n\bs\HH^3$ of finite-volume hyperbolic three--manifolds. We will denote by $\Lambda_{n,j},\, j=1,\ldots,h_n$ the Euclidean lattices corresponding to the cusps of $M_n$, which are well-defined up to similarity. Recall that we have defined the counting function $\mathcal{N}_\Gamma$ in \eqref{def_count_lox}. 

\begin{lem} \label{hypBS}
The sequence $M_n$ is BS-convergent to $\HH^3$ if and only if 
\begin{equation} \label{hypBSloc}
\forall r>0, \: \int_{M_n}\mathcal{N}_{\Gamma_n}(x,r) dx = o(\vol M_n) 
\end{equation}
and 
\begin{equation} \label{sumcusp}
\sum_{j=1}^{h_n} \frac {\alpha_2(\Lambda_{n,j})} {\alpha_1(\Lambda_{n,j})} = o(\vol M_n).
\end{equation}
\end{lem}

\begin{proof}
We won't use the `if' statement in the remainder of this paper, and its proof is straighforward. Suppose now that the sequence $M_n$ is BS-convergent to $\HH^3$. If we suppose in addition that the systole of the $M_n$ is bounded away from $0$ then \eqref{hypBSloc} follows immediately from \eqref{loxcount}: for any $r>0$ we have 
\[
\int_{M_n} \mathcal{N}_{\Gamma_n}(x,r) dx  = \int_{(M_n)_{\le r}} \mathcal{N}_{\Gamma_n}(x,r) dx  \le Ce^{cr}\vol(M_n)_{\le r}
\]
where $C$ does not depend on $n$, and the right hand-side is an $o(\vol M_n)$ by the definition of BS-convergence. In general, we obtain from this resoning the conclusion that for any $\delta>0$ the part of the integral in \eqref{hypBSloc} on the $\delta$-thick part of $M_n$ is an $o(\vol M_n)$. The proof that \eqref{hypBSloc} holds in general then depends on a fine analysis of the orbits of points in $\HH^3$ mapping to the $\delta$-thin part of $M$ (for $\delta$ smaller than the Margulis constant) which is carried out in \cite[Section 7]{7S}. 

We finally establish \eqref{sumcusp} when $M_n$ is BS-convergent to $\HH^3$: let $\eps>0$ be the Margulis constant for $\HH^3$, and let $C_1,\ldots,C_h$ be the noncompact components of $(M_n)_{\le\eps}$. The boundaries of the $C_j$ are Euclidean tori $T_1,\ldots,T_j$ and we have $\eps\le c\alpha_1(T_j)$ for some absolute $c>0$ ; it follows that 
\[
\vol(T_j) \gg \alpha_1(T_j)\alpha_2(T_j) \gg \eps^2\frac{\alpha_2(\Lambda_j)}{\alpha_1(\Lambda_j)}
\]
where $\Lambda_j$ is the lattice in $\CC$ corresponding to $T_j$ (whose conformal class is well-defined). It follows that 
\[
\vol (M_n)_{\le\eps} \gg \sum_{j=1}^h \vol(T_j) \gg \eps^2 \sum_{j=1}^h \frac{\alpha_2(\Lambda_j)}{\alpha_1(\Lambda_j)}
\]
hence the right-hand side must be an $o(\vol M_n)$ which is precisely the content of \eqref{sumcusp}. 
\end{proof}

We record as a separate fact the following weaker consequence of \eqref{sumcusp}. 

\begin{lem} \label{nbcuspsBS}
Let $M_n$ be a sequence of finite-volume hyperbolic three--manifolds, $h_n$ the number of cusps of $M_n$. If $M_n$ BS-converges to $\HH^3$ then $h_n = o(\vol M_n)$. 
\end{lem}

When we assume cusp-uniformity we only need to look at the behaviour of closed geodesics; we have the following criterion for a sequence of cusp-uniform hyperbolic three--manifolds to BS-converge. The direct implication is contained in Lemma \ref{hypBS} above and the converse is proved in \cite[Proposition 4.7]{thesis}. 

\begin{lem}
Let $M_n$ be a cusp-uniform sequence of finite covers of a hyperbolic three--manifold $M$. Then $M_n$ BS-converges to $\HH^3$ if and only if condition \eqref{hypBSloc} holds.
\end{lem}


\section{Spectral analysis on manifolds with cusps}

\label{spec_dec}
\subsection{Local systems on hyperbolic manifolds}

\subsubsection{Definitions}
\label{systloc}

Let $\Gamma\subset G$ be a lattice and put $M=\Gamma\backslash\HH^3$. The flat real vector bundles (a.k.a. ``real local systems'') on $M$ are obtained as follows: if $\sigma$ is a representation of $\Gamma$ on a finite-dimensional real vector space $V$ we get a vector bundle $F_\sigma$ on $M$ whose total space is the quotient $\Gamma\backslash(\HH^3\times V)$. For $\gamma\in\Gamma$ and a $p$-form $f$ on $\HH^3$ with coefficients in $V$ we denote $\gamma^*f=\sigma(\gamma)^{-1}\circ f \circ\wedge^pT\gamma$. Then the $p$-forms on $M$ with coefficients in $F_\sigma$ correspond to $\Gamma$-equivariant sections of $\wedge^pT\HH^3\rightarrow V$  i.e. to those $f\in\Omega^p(\HH^3;V)$ such that $\gamma^*f=f$ for all $\gamma\in\Gamma$. 

Particularly interesting among all flat bundles are those whose holonomy comes from restricting a representation $\rho$ of $G$ on a real vector space $V$. The representation $\sigma=\rho|_\Gamma$ is never orthogonal but the bundle $F_\sigma$ has an alternative description which yields a natural euclidean product and which we will now describe. Up to scaling there is a unique inner product on $V$ which is preserved by $K$ and such that $\LP$ (the orthogonal for the Killing form of the Lie subalgebra $\frk \subset \frg$ of the group $K$) acts by self-adjoint maps (see \cite{Matsushima_Murakami}). We have a vector bundle $E_\rho$ on $M$ whose total space is $(\Gamma\backslash G\times V)/K$ so that it has a natural metric $|.|$ coming from the $K$-invariant metric on $V$. The square-integrable sections of $E_\rho$ correspond to the subspace:
\begin{equation*}
\{f:\Gamma\backslash G\rightarrow V,\: |f|\in L^2(\Gamma\backslash G),\:\forall g\in G,k\in K,\: f(gk)=\rho(k^{-1})f(g)\}.
v\end{equation*}
More generally, identifying the tangent space of $\HH^3$ at the fixed point of $K$ (which is an irreducible real $K$-representation) with $\LP$, the square-integrable $p$-forms correspond to:
\begin{equation*}
L^2\Omega^p(M;E_\rho)=\left( L^2(\Gamma\backslash G)\otimes V\otimes \wedge^p\LP^*\right)^K 
\end{equation*}
(where we use the habitual notation $H^K$ for the fixed subspace of $K$ in a vector space $H$). We have an isomorphism $E_\rho\to F_\sigma$ induced by the map $G\times V\rightarrow G\times V, \, (g,v)\mapsto (g,\rho(g) \cdot v)$. In the sequel we will denote by $L^2\Omega^p(M;V)$ the space of square-integrable $p$-forms on $M$ with coefficients in $E_\rho$. 

The Hodge Laplacians $\Delta^p[M]$ are essentially self-adjoint operators on the Hilbert spaces $L^2\Omega^p(M;V)$, see \cite{BW} or \cite[Section 3]{BV}.


\subsubsection{Strong acyclicity}
\label{repr}

The group $G = \SL_2(\CC)$ acts naturally on $\CC^2$. As a real Lie group it also has a representation on $\CC^2$ given by $g \mapsto \ovl g$ (where $\ovl\cdot$ denotes the complex conjugate matrix). We will use the notation $\overline{\CC^2}$ to indicate that we consider this conjugate action. For every pair of nonnegative integers we then have a representation of $G$ on the vector space $V(n_1, n_2)$ defined by:
\[
V(n_1,n_2) = \sym^{n_1}(\CC^2)\otimes\sym^{n_2}\overline{\CC^2}. 
\]
Standard representation theory tells us that these are all the irreducible finite-dimensional representations of $G$. 

The most important (for us) feature of the representations $V(n_1, n_2)$ is the following spectral gap property, which is proven in \cite[Lemma 4.1]{BV} and also follows from \cite[Proposition 6.12 in Chapter II]{BW}; Bergeron and Venkatesh term this ``strong acyclicity'' of the representation. 

\begin{prop}
Let $n_1\not=n_2$ and $V=V(n_1, n_2)$. There exists $\lambda_0>0$ such that for any lattice $\Gamma$ in $G$, $M=\Gamma\backslash\mathbb{H}^3$, $p=0,1,2,3$ and $\phi\in L^2\Omega^p(M;V)$ we have  
\begin{equation*}
\langle\Delta^p[M]\phi,\phi\rangle_{L^2\Omega^p(M;V)}\ge \lambda_0\|\phi\|_{ L^2\Omega^p(M;V)}.
\end{equation*}
\label{strongacyclicity}
\end{prop}


\subsubsection{Unitary representations}
Let $\sigma_m, y$ be defined by:
\begin{equation} \label{def_char}
\sigma_m\begin{pmatrix} e^{i\theta} & 0 \\ 0 & e^{-i\theta} \end{pmatrix} = e^{mi\theta}, \quad y\begin{pmatrix} t & 0 \\ 0 & t^{-1} \end{pmatrix} = t^2. 
\end{equation}
For $s\in \CC$ and $m \in \ZZ$ we denote by $\pi(s, m)$ the representation of $G$ induced by the character $\sigma_m \otimes y^{1 + \frac s 2}$ of $P_\infty = M_\infty A_\infty N_\infty$. This is the representation $\mathscr P^{m, s}$ defined in \cite[(2.11)]{Knapp}; it is unitary if and only if $s \in i\RR$.

\subsubsection{Laplace eigenvalues and differentials}
Let $e_1, e_{-1}$ be the canonical basis of $\CC^2$. For $l = -n, -n + 2, \ldots, n$ put $e_l = e_1^{\frac{n + l} 2} e_{-1}^{\frac{n - l} 2} \in V(n,0)$ and for $l = -n_1, \ldots, n_1, \, k = -n_2, \ldots, n_2$ put $e_{l,k} = e_l \otimes \ovl e_{-k} \in V(n_1, n_2)$. Thus we have 
\[
\rho(ma) \cdot e_{l,k} = \sigma_{l+k}(m) y(a)^{\frac{l-k}2} e_{l,k},\: m\in M_\infty,a\in A_\infty. 
\]
Define $V_{l,k} = \CC e_{l,k},\, V_m = \sum_{l+k=m} V_{l,k}$. Let $P = g_0 P_\infty g_0^{-1}$ be a parabolic subgroup, $y_P$ a height function at $P$, $s \in \CC, v \in V$ and define a section of $G\times V_\CC$ by the formula:
\[
\phi_{s,v}(g) = y_P(g)^{1 + \frac s 2} \rho(k^{-1}) \cdot v, \: g=nak. 
\]
If $v \in g_0 V_m$ then $\phi_{s,v}$ belongs to the space of $\pi(s, m)$. Its $K$-type is contained in $V_\CC$, and thus it yields a section of $E_\rho$ over $\HH^3$. 

By computation of the Casimir eigenvalues in the induced representation (see \cite[5.7]{BV} who cite \cite[Proposition 8.22 and Lemma 12.28]{Knapp}) the functions $\phi_{s,v},\, v\in g_0 V_m$, give rise to sections of $E_\rho$ which are eigenvectors of $\Delta^0[\HH^3]$ with eigenvalue 
\[
|s|^2 - m^2 + (n_1 + n_2 + 2)^2 + (n_1 - n_2)^2.
\] 
Note that this bounded away from zero for all $n_1,n_2,m$ and $s\in i\RR$ (since $m \in [-n_1 - n_2, n_1+n_2]$).

\medskip

Now let $W = V \otimes V(2,0)$. The $G$-equivariant bundle associated to $W_\CC$ is isomorphic to the bundle of $1$-forms with coefficients in $V_\CC$. Using the same construction as above we get an eigenform with coefficients in $E_\rho$ and eigenvalue 
\begin{equation} \label{Caseig}
-s^2 - (m + \eps)^2 + (n_1 + n_2 + 2)^2 + (n_1 - n_2)^2
\end{equation}
where $\eps=0,\pm2$ according to whether $v\in g_0 V_m \otimes V_\eps$; the eigenvalue is larger than $(n_1-n_2)^2$ for $s \in i\RR$, in particular bounded away from 0 when $n_1 \not= n_2$.  

Now let us compute the differentials for sections and $1$-forms. In both cases this has to be done in the $G$-equivariant model for $E_\rho$. Let $v \in V_{l,k}$, then the $G$-equivariant section corresponding to $\phi_{s,v}$ is $g\mapsto y_P(g)^{1 + \frac{s + l - k} 2} \rho(n) \cdot v$, $g=nak$ and thus: 
\begin{equation*}
d\phi_{s,v}(g) = \frac 1 2 (s + l - k + 2) y_P(g)^{\frac{s + l -k} 2}(\rho(n) \cdot v) \otimes dy_P + \ldots 
\end{equation*}
where $\ldots$ indicates terms which are orthogonal to $dy_P$. If $v \in V_{l,k} \otimes V_{-2,0}$ then the corresponding $G$-equivariant 1-form on $\HH^3$ is given by $y_P(g)^{\frac{s+l-k}2} (\rho(n) \cdot w)\otimes dz$ and we have:
\begin{equation}
d\phi_{s,v\otimes e_2} = \frac 1 2 (s + l - k) y_P(g)^{s+l-k-1} (\rho(n) \cdot v) \otimes dy_P \wedge dz + \ldots 
\label{calcdiffind1}
\end{equation}
where the $\ldots$ indicate terms in $dz\wedge d\ovl z$, and a similar computation holds for forms in $d\ovl z$. The forms in $dy_P$ are closed.


\subsection{Spectral decomposition}
\label{Eisenstein}

From now on we fix a $G$-representation $\rho$ on a vector space $V$. It is a well-known fact that one has the orthogonal sum 
\begin{equation} \label{spectral_decomposition}
L^2\Omega^p(M; V) = L_\disc^2\Omega^p(M; V) \oplus L_\cont^2\Omega^p(M; V)
\end{equation}
where $\Delta^p[M]$ has only discrete spectrum in $L_\disc^2\Omega^p(M;V)$ and completely continuous spectrum in $L_\cont^2\Omega^p(M;V)$. Here we briefly describe the proof of this result through the theory of Eisenstein series developed by Selberg, Langlands and others which actually yields a complete description of the continuous part. 


\subsubsection{Constant terms and cusp forms}

Let $P$ be any $\Gamma$-rational parabolic and $f\in L^2\Omega^p(\HH^3;V)$ a $\Gamma$-equivariant $p$-form. Its constant term at $P$ is defined to be the $p$-form given by
\begin{equation}
f_P(v)= \int_{\Gamma_P\backslash N} n^*f(v)\frac{dn}{\vol(\Gamma_P\backslash N)}.
\label{deftc}
\end{equation}
This descends a $p$-form on $\Gamma_P\backslash\HH^3$ (which depends only on the $\Gamma$-conjugacy class of $P$) which is actually $N$-equivariant. If $h :\Gamma\backslash G \to V\otimes\wedge^p\LP^*$ is the $K$-equivariant function corresponding to $f$ (see \ref{systloc}) then the one corresponding to $f_P$ is given by $g \mapsto 1/2 \int_{\Gamma_P\backslash N}h(ng)dn$. A $p$-form $f$ is said to be cuspidal when $f_P=0$ for all $\Gamma$-rational parabolics, and we denote by $L_\cusp^2\Omega^p(M;V)$ the space of all such forms. Theorem \ref{Selberg} below implies that we have $L_\cusp^2\subset L_\disc^2$.


\subsubsection{Eisenstein series}

If $P$ is a $\Gamma$-rational parabolic there is a map $E_P^p$ from the subspace of $N$-equivariant forms in $L^2\Omega^p(\HH^3;V)$ to $L^2\Omega^p(M;V)$ given by 
\begin{equation}
E_P^p(f) = \sum_{\gamma\in\Gamma/\Gamma_P}\gamma^*f.
\label{te}
\end{equation}
If $P,P'$ are two equivalent $\Gamma$-rational parabolics then the obvious map $\theta:L^2(N\backslash \HH^3;V)\to L^2(N'\backslash \HH^3;V)$ intertwines $E_P^p$ and $E_{P'}^p$, i.e. $E_P^p=E_{P'}^p\circ\theta$. We choose representatives $P_1,\ldots,P_h$ of the conjugacy classes of $\Gamma$-rational parabolics and put $E^p=\bigoplus_{j=1}^h E_{P_j}^p$. Then we have the following facts:
\begin{itemize}
\item[\textbullet] $\im(E^p)=L_\cusp^2\Omega^p(M;V)^\bot$; 
\item[\textbullet] there is a finite-dimensional subspace $L_\res^2$ inside 
$\im(E)$ such that we have the orthogonal sum $\im(E)=L_\cont^2\oplus L_\res^2$. 
\end{itemize}
When $V$ is strongly acyclic the subspace $L_\res^2$ is actually zero for all $p$; when $V$ is trivial it is of dimension one for $p=0$ or $p=3$ and zero for $p=1,2$. We will now describe how the map allows to describe the continous part $L_\cont^2\Omega^p(M;V)$: we begin by a general exposition and then specialise to sections and 1-forms with coefficients in a bundle $E_\rho$. 


\subsubsection{About references}

Our main reference for this subsection is G. Warner's disquisition \cite{Warner}; the theory we expose here is developed there in greater generality (for all real-rank-one locally symmetric spaces) with more details (though the author frequently refers to \cite{HC} for complete proofs). The exposition in this reference is not particularly user-friendly; for a more accessible one (only in the case of Fuchsian groups and functions, but all the main ideas are already present) we refer to H. Iwaniec's textbook \cite{Iwaniec}. The case of arithmetic 3-manifolds is also treated in detail in \cite[Chapitre 5]{thesis}; the book \cite{EGM} contains a complete treatment of functions on more general hyperbolic manifolds.


\subsubsection{Eisenstein series with coefficients in a $K$-equivariant bundle}
\label{eis_gen}
Let $\tau$ be a finite-dimensional representation of $K$ on a complex vector space $W$, with highest weight $q\in\NN$. The space $W$ decomposes as the orthogonal sum 
\begin{equation*}
W = \begin{cases} 
       \bigoplus_{k=-q}^q W_{2k}  & q=\frac{n_1+n_2}2, \:n_1-n_2 \text{ even}; \\
       \bigoplus_{k=-q}^q W_{2k+1}  & q=\frac{n_1+n_2-1}2, \:n_1-n_2 \text{ odd}    \end{cases}
\end{equation*}
where $W_l$ is the subspace on which $M$ acts by the character $\sigma_l$ defined in \eqref{def_char}. Let $E_\tau$ be the bundle on $\Gamma_P\bs\HH^3$ whose total space is given by $(W\times\Gamma_P\bs G)/K$; then the smooth sections of $E_\tau$ are identified with the space $\left(W\otimes C^\infty(\Gamma_P\bs G)\right)^K$. For $s \in \CC$ we identify the subset of such sections which are $N$-invariant on the right, proportional to $y_P^{1 + s}$ and in the image of $\left(W_l\otimes C^\infty(\Gamma_P\bs G)\right)^K$ with $W_l$: we denote this identification by $w\mapsto w_s$. Then for $w\in W_l$ the Eisenstein series $E_P(w_s)$ corresponds to the Eisenstein series denoted by $E(P:w:s/2:\cdot)$ in \cite{Warner}\footnote{Note that our parameter $s$ differs from that used in this reference by a factor of 2, but this does not affect any of the results we quote from there.}, and hence we have the following properties for it from loc. cit.:
\begin{itemize}
\item The series is convergent for $\Real(s)>0$, and admits a meromorphic extension to $\CC$ with no poles on the imaginary axis (\cite[page 9]{Warner}).

\item The constant terms of $E_{P_j}(w_s)$ are given by 
\begin{equation} \label{tc_general}
\delta_{i,j} y_{P_i}^{1 + \frac s 2} w + y_{P_i}^{1 - \frac s 2} \Phi_{i,j;l}(s)w
\end{equation}
where $\Phi_{i,j;l}$ is a meromorphic function with values in $\hom_\CC(W_l,W_{-l})$ (\cite[pages 7,13]{Warner} where $\Phi_{i,j;l}\oplus\Phi_{i,j;-l}$ corresponds to $c_{P_i|P_j}(w,s)$). 

\item Put $\Phi_l(s)=\bigoplus_{i,j}\Phi_{i,j;l}(s) \in \hom \left((W_l)^h,(W_{-l})^h \right)$. Then we have the functional equations $\Phi_l(-s)\Phi_l(s) = \id$, and $\Phi_l(iu)^* = \Phi_l(iu)^{-1}$ for $u \in \RR$ (\cite[page 8]{Warner}). 

\item The continuous part of $L^2(M;E_\tau)$ is spanned by the functions $\int_{-\infty}^{+\infty} \psi(u)E(w_{iu})du$ for $\psi\in L^2(\RR)$ and $w\in W^h$ (where $E(w_s) = \sum_{j=1}^h E((w_j)_s)$): \cite[page 32]{Warner}
\end{itemize}

For $Y = (Y_1, \ldots, Y_h) \in [1, +\infty[$ one defines the truncation operator at height $Y$ by:
\begin{equation*}
T^Y f(g) = f(g) - \sum_{j=1}^h 1_{[Y_j,+\infty)}(y_j(g)) f_{P_j}(g), \quad f\in C^\infty(M;E_\tau).
\end{equation*}
For $w\in (W_l)^h$ we have the `Maass--Selberg relations':
\begin{equation}
\begin{split}
\|T^Y E(w_s)\|_{L^2(M;E_\tau)}^2 &= 2\sum_{j=1}^h \log(Y_j)|w_j|_W^2 + \langle\Phi_l(s)^{-1}\Phi_l(s)'(w),w\rangle_{W^h} \\
                                 &\quad  +\sum_{j=1}^h \frac 1 s \bigl(Y_j^s\langle(\Phi_l(-s)w)_j,w_j\rangle_W - Y_j^{-s}\langle(\Phi_l(s)w)_j,w_j\rangle_W\bigr)
\end{split}
\label{MSG}
\end{equation} 
see \cite[page 83]{Warner}.


\subsubsection{Sections}
Let $v\in V^h$; we denote by $E(s,v)$ the section of the bundle $E_\rho$ over $\Gamma\bs\HH^3$ corresponding to $E(v_s)$ in the notation above. For $l=-q,\ldots,q$ we let $\Psi_l(s)=\bigoplus_{i} \Phi_{j,i:l}(s) \in\hom(V_l^h,V_l)$ so that the constant terms of $E(s,v)$ are given by 
\begin{equation*}
E(s,v)_{P_j} = y_j^{1 + s} v_j + y_j^{1 - s}(\Psi_l(s)v)_j
\end{equation*}
for $v\in V_l$. For $v\in (V_l)^h$ the sections $E(s,v)$ are eigenfunctions of the laplacian $\Delta^p[M]$ with eigenvalue $-s^2 - l^2 + \lambda_V$ where $\lambda_V$ is the Casimir eigenvalue of $V$, $\lambda_V = (n_1+n_2+2)^2 + (n_1-n_2)^2$ if $V = V(n_1, n_2)$ by \eqref{Caseig}. 

For $v\in V_l$ the Maass--Selberg relations \eqref{MSG} are written:
\begin{equation}
\begin{split}
\|T^Y E(s,v)\|_{L^2(M;V)}^2 &= 2\sum_{j=1}^h \log(Y_j)|v_j|_V^2 + \langle\Psi_l(s)^{-1}\Psi_l(s)'(v),v\rangle_{V^h} \\
                                 &\quad  +\sum_{j=1}^h \frac 1 s \bigl(Y_j^s\langle(\Psi_l(-s)v)_j,v_j\rangle_V - Y_j^{-s}\langle(\Psi_l(s)v)_j,v_j\rangle_V\bigr)
\end{split}
\label{MS0}
\end{equation}


\subsubsection{1-forms}

We denote by $\Omega_j^+(V)$ (resp. $\Omega_j^-(V)$) the space of 1-forms on $\Gamma_{P_j}\bs N_j$ with coefficients in the restriction of $E_\tau$ which are of the form $dz_j\otimes v$ (resp. $d\ovl z\otimes v$) for $v\in V$, and by $\Omega_j^+(V_l)$ the supspace of those for which $v\in g_jV_l$ where $g_j$ conjugates $P_j$ to the parabolic at infinity $P_\infty$ (and define $\Omega_j^-(V_l)$ similarly). We put $\Omega^\pm(V_l)=\bigoplus_j \Omega_j^\pm(V_l), \, \Omega^\pm(V)=\bigoplus_j \Omega_j^\pm(V)$. 

On the other hand, the 1-forms in coefficients in $E_\rho$ on $\Gamma_P\bs\HH^3$ correspond to the sections of the bundle $E_\tau$ where $\tau=\rho|_K\otimes\ad|_K$ (where $\ad$ is the adjoint representation of $G$, which is isomosphic to $V(2,0)$). The representation $\tau$ has two summands: one isomorphic to $\rho|_K$ which corresponds (in the correspondance set in \ref{eis_gen}) to the differential of sections, and its orthogonal which corresponds to co-closed 1-forms, whose constant terms are of the form $\omega+\ovl\omega$ for $\omega\in\Omega^+(V),\ovl\omega\in\Omega^-(V)$: we denote the latter by $W$, and by $W_l$ the subspace $\Omega^+(V_{l-2})\oplus\Omega^-(V_{l+2})$. Then for $\omega\in W_l$ the 1-form $E(s,\omega)$ corresponding to $E(\omega_s)$ is an eigenform of the laplacian with eigenvalue $-s^2-(l\mp 2)^2+\lambda_V$ (again by \eqref{Caseig}). The constant terms of $E(s,\omega)$ are given more precisely by
\[
E(s,\omega)_{P_j} = y^{1 + s} \omega_j + y^{1 - s} \left(\Phi_l(s)\omega\right)_j. 
\]

The Maass-Selberg relations are given by:
\begin{equation}
\|T^Y E(s,\omega)\|_{L^2\Omega^1(M;V)}^2 = 2\sum_{j=1}^h \log(Y_j)|\omega_j|_{\Omega^\pm(V)}^2+\langle\Phi^\pm(s)^{-1}\Phi^\pm(s)'(\omega),\omega\rangle_{\Omega^\pm(V)}.
\label{MS1}
\end{equation}


\subsubsection{2- and 3-forms}
The Hodge $*$ yields isometries $L^2\Omega^p(M;V)\rightarrow L^2\Omega^{3-p}(M;V)$, so that the spectral decomposition for $L^2\Omega^2,L^2\Omega^3$ spaces follows from that of $L^2\Omega^1$ and $L^2$ respectively.


\section{Selberg's trace formula and regularised traces}

\label{trace}
\subsection{Automorphic kernels}

As noted in \ref{systloc} the Laplacians $\Delta^p[\HH^3]$ on $\HH^3$ with coefficients in a flat bundle are essentially self-adjoint operators and the spectral theorem thus allows, for a function $\phi\in C^{\infty}([0,+\infty))$, to define an operator $\phi(\Delta^p[\HH^3])$ on $L^2\Omega^p(\HH^3;V)$. Moreover, if $\phi$ is sufficiently decreasing at infinity this operator is given by convolution with a kernel 
\[
k_{\phi,p}\in C^{\infty}(\HH^3\times\HH^3;(\wedge^pT\HH^3\otimes V)\otimes(\wedge^pT\HH^3\otimes V)^*),
\]
i.e. $k_{\phi,p}(x,y)\in\hom(\wedge^pT_x^*\HH^3\otimes V,\wedge^pT_y^*\HH^3\otimes V)$ and for a $p$-form $f\in L^2\Omega^p(\HH^3;V)$ one has
\[
\phi(\Delta^p[\HH^3])f(y)=\int_{\HH^3}k_{\phi,p}(x,y)f(x)dx.
\]
The kernels $k_{\phi,p}$ are invariant under the diagonal action of $G$ on $\HH^3\times\HH^3$, meaning that for $g\in G,\, x,y\in\HH^3$ we have 
\begin{equation}\label{Geq}
k_{\phi,p}(x,y)=(\wedge^p T_y g^{-1}\otimes\mathrm{Id}_V) \circ k_{\phi,p}(gx,gy)\circ (\wedge^p T_x g\otimes\mathrm{Id}_V).
\end{equation}
The Plancherel formula for $G$ allows to compute the $k_{\phi,p}$ and with a lot more work one can obtain the following lemma (essentially due to F. Sauvageot), an explanation of which can be found in \cite[Proposition 6.4]{7S} (by density of a subset $S$ we mean that any Radon measure is determined by its restriction to $S$). 

\begin{lem} \label{Plancherel}
The space $\mathcal{A}(\RR)$ of smooth functions $\phi$ on $\RR$ such that for any $\phi\in\mathcal{A}(\RR)$ we have $k_{\phi,p}(x,y)\ll e^{-Ad(x,y)}$ for all $A>0$ is dense in the space $\mathcal{S}(\RR)$ of Schwartz functions.
\end{lem}

From now on we will always suppose that $\phi\in\mathcal{A}(\RR)$. For $g\in G$ we put: 
\[
g^*k_{\phi,p}(x,y)=(\wedge^p T_y g^{-1}\otimes\rho(g)^{-1})\circ k_{\phi,p}(x,y) \in \hom(\wedge^pT_x^*\HH^3\otimes V,\wedge^pT_{g^{-1}y}^*\HH^3\otimes V). 
\]
By the above Lemma we have $|g^*k_{\phi,p}(x,y)|\ll e^{-Ad(x,y)}$ so that it follows from the well-known estimate
\[
|\{ \gamma\in\Gamma, d(x,\gamma y)\le r\}| \le Ce^{cr}
\]
(where $c$ is absolute and $C$ depends on $\Gamma,x$---see also Lemma \ref{gen_count}) that the following series converges uniformly on compact sets of $\HH^3\times\HH^3$: 
\[
K_{\phi,p}^\Gamma(x,y)=\sum_{\gamma\in\Gamma}\gamma^*k_{\phi,p}(x,\gamma y). 
\]
The kernel $K_{\phi,p}^\Gamma$ is $\Gamma$-equivariant in each variable and hence can be seen as a section of $(\wedge^pTM\otimes V)\otimes(\wedge^pTM\otimes V)^*$. On the other hand, since the operator $\Delta^p[M]$ (the Laplacian on $p$-forms on $M$ with coefficients in $E_\rho$) is essentially self-adjoint we can define the operator $\phi(\Delta^p[M])$ on $L^2\Omega^p(M;V)$. Then $K_{\phi,p}^\Gamma$ is a kernel for $\phi(\Delta^p[M])$, in other words for $f\in L^2\Omega^p(M;V)$ we have:
\begin{equation}
\phi(\Delta^p[M])f(y)=\int_{\Gamma\backslash\HH^3} K_{\phi,p}^\Gamma(x,y)f(x) dx.
\label{autodev}
\end{equation}


\subsubsection{Truncation}

In the sequel we will write $Kf$ for the convolution of a section $f$ with a kernel $K$. Let $P$ be a parabolic subgroup of $G$, we define the constant term at $P$ of $k_{\phi,p}$ to be the kernel given by
\[
(k_{\phi,p})_P(x,y) = \int_N n^* k_{\phi,p}(x,ny)dn.
\]
For a $\Gamma$-rational parabolic subgroup $P$ we define the constant term $(K_{\phi,p}^\Gamma)_P$ of $K_{\phi,p}^\Gamma$ at $P$ by 
\[
(K_{\phi,p}^\Gamma)_P(x,y) = 
          \frac 1{\vol(\Gamma_P\backslash N)} \sum_{\gamma\in\Gamma/\Gamma_P} \gamma^*(k_{\phi,p})_{\gamma P\gamma^{-1}}(x,\gamma y) 
        = \frac 1{\vol(\Gamma_P\backslash N)} \sum_{\gamma\in\Gamma/\Gamma_P} \int_N (\gamma n)^* k_{\phi,p}(x,\gamma ny)dn.
\]
For $f\in L^2\Omega^p(M;V)$ a routine calculation yields
\begin{equation} \label{Pterm}
(K_{\phi,p}^\Gamma)_P(f)=(K_{\phi,p}^\Gamma)(f_P)
\end{equation}
Recall that the truncated manifold $M^Y$ was defined in \eqref{tron}. One naturally defines the truncated kernel on $M$ by:
\[
T^Y K_{\phi,p}^\Gamma=
\begin{cases}K_{\phi,p}^\Gamma(x,y) - (K_{\phi,p}^\Gamma)_{P_j}(x,y) & y_j(y)\ge Y_j;\\
             K_{\phi,p}^\Gamma(x,y)                            & y\in M^Y
  \end{cases}
\]
and it follows from \eqref{Pterm} that
\begin{equation} \label{KYF=KFY}
T^Y K_{\phi,p}^\Gamma(f)=K_{\phi,p}^\Gamma(T^Y f).
\end{equation}


\subsection{Geometric side}

Let $h_{\phi,p}$ be the function on $[0,+\infty[$ defined by 
\begin{equation}
h_{\phi,p}(\ell) = \tr\left(n^*k_{\phi,p}(x,nx)\right) 
\label{defh}
\end{equation}
for any unipotent isometry $n\in G$ such that $\ell = d(x,nx)$. This definition is legitimate, i.e. the right-hand side depends only on $\ell$: indeed, if $n,n'$ are two unipotent elements of $G$ and $x,x'$ two points in $\HH^3$ such that $d(x,nx) = d(x',n'x')$ there exists a $g\in G$ such that $gx = x'$ and $gng^{-1} = n'$ (this follows immediately from Lemma \ref{eval} and the fact that the stabiliser of an horosphere in $\HH^3$ acts transitively on euclidean spheres---note that this is no longer true in symmetric spaces other than the real hyperbolic spaces) and hence 
\[
\tr(n^*k_{\phi,p}(x,nx)) = \tr((gng^{-1})^* k_{\phi,p}(gx,gnx)) = \tr((n')^* k_{\phi,p}(x',n'x')). 
\]
Let $\Gamma$ be a torsion-free lattice in $G$ and let $h$ be the number of cusps of the manifold $M=\Gamma\bs\HH^3$, which we suppose endowed with an arbitrary height function $y=\max_j y_j$. Let $\Lambda_1,\ldots,\Lambda_h$ be the Euclidean lattices associated to the cusps of $M$ and $y$; we associate to them the following quantity:
\[
\kappa_j = 2\int_{\alpha_1(\Lambda_j)}^{+\infty} E_{\Lambda_j}(\rho)\frac{d\rho}{\rho^3} + \frac{\pi(1 - 2\log\alpha_1(\Lambda_j))}{\vol(\Lambda_j)}; 
\]
note that only the second summand depends on the choice of $y$. We also define 
\[
\otr_\Gamma (k_{\phi,p}) = \vol(M)\cdot\tr \left( k_{\phi,p}(x_0,x_0) \right)
\]
where $x_0$ is any point of $\HH^3$. For the statement and proof of the following proposition we will suppose that $\Gamma\cap P\subset N$ for all parabolic subgroups $P$ of $G$ with unipotent radical $N$ (we remark that a modified version of the proposition is true in all generality, see \cite[3.5]{thesis}). 

\begin{prop}
Let $\phi\in{\mathcal A}(\RR)$, $p=0,\ldots,3$ and let $K_{\phi,p}^\Gamma$ be the associated automorphic kernel on $M$. Then for any $Y \in [1,+\infty[^h$ the integral
\[
\int_M T^Y K_{\phi,p}^\Gamma(x,x) dx
\]
is absolutely convergent, and as $\min_j(Y_j)$ tends to infinity we have the following asymptotic expansion:
\begin{align*}
\otr\left( T^YK_{\phi, p}^\Gamma\right) &= \left(2\pi h \int_0^{+\infty} r h_{\phi,p}(\ell(r)) dr \right) \sum_{j=1}^h \log Y_j \\
                                &\quad + \otr_\Gamma k_{\phi, p} + \int_M \sum_{\gamma\in\Gamma_\lox} \tr(\gamma^*k_{\phi, p}(x,\gamma x)) dx \\
                  &\quad +2\pi h \int_0^{+\infty} r \log(r) h_{\phi,p}(\ell(r))dr  + \sum_{j=1}^h \kappa_j \vol(\Lambda_j) \int_0^{+\infty} r h_{\phi,p}(\ell(r))dr + o(1).
\end{align*}
\label{geom_side}
\end{prop}

\begin{proof}
To make things more readable we will deal only with the case where $M$ has only one cusp (only notational alterations are necessary to deal with the general case). We fix a $\Gamma$-rational parabolic $P$ with unipotent radical $N$ and denote by $\Lambda$ the Euclidean lattice associated to $\Gamma_P=\Gamma\cap N$. We let $D$ denote a fundamental domain for $\Gamma$ in $\HH^3$ and $D^Y\subset D$ the preimage of $M^Y$: we suppose that the only ideal vertex of $D$ is the fixed point of $P$, so that for $Y$ large enough $D-D^Y$ is contained in the horoball of height $Y$ at $P$. By the definition of the function $h_{\phi,p}$ and Lemma \ref{eval} we have
\begin{align*}
&\int_{M-M^Y} |\tr T^YK_{\phi,p}^\Gamma(x,x)| dx = \int_{D-D^Y} |\tr(K_{\phi,p}^\Gamma-(K_{\phi,p}^\Gamma)_P)(x,x)| dx \\
      &\qquad\qquad \le  \int_{D-D^Y} \left| \sum_{\eta\in\Lambda-\{0\}} h_{\phi,p} \left(\ell\left(\frac{|\eta|}{y_P(x)} \right)\right) - \frac 1{\vol(\Lambda)} \int_N h_{\phi,p} \left(\ell\left(\frac{|n|}{y_P(x)} \right)\right) dn \right| dx \\
      &\qquad\qquad\quad + \int_{D-D^Y}  \sum_{\substack{\gamma\in\Gamma/\Gamma_P \\ \gamma \not= \Gamma_P}}\frac 1{\vol(\Lambda)} \int_N |k_{\phi,p}(x,\gamma nx)| dn  dx \\
      &\qquad\qquad\quad + \int_{D-D^Y} \sum_{\gamma\in\Gamma-\Gamma_P} |k_{\phi,p}(x,\gamma x)| dx  \\
      &\qquad\qquad\quad + \vol(D-D^Y) |\tr k_{\phi,p}(x_0,x_0)| \\
      &\qquad\qquad =:  I_1 + I_2 + I_3 + I_4. 
\end{align*}
Now we prove that the $I_k$ for $k=1,\ldots,4$ are finite and go to 0 as $Y\to+\infty$. The term $I_4$ is trivial to deal with. Let us deal with $I_3$: for any $a>0$, since $\phi\in\mathcal{A}(\RR)$ we have $|k_{\phi,p}(x,y)|\le e^{-ad(x,y)}$, and it follows that:
\begin{align*}
I_3 &\le \int_{D-D^Y} \sum_{\gamma\in\Gamma-\Gamma_P} e^{-ad(x,\gamma x)} dx \\
    &  = \int_{D-D^Y} \int_0^{+\infty} e^{-ar}dN_x(r) dx
\intertext{where $N_x(r)=|\{\gamma\in\Gamma-\Gamma_P:\: d(x,\gamma x)\le r\}|$; integrating by parts we get:}
I_3 & \le \int_{D-D^Y} \int_0^{+\infty} ae^{-ar}N_x(r) drdx \ll_{M,a} \vol(D-D^Y) \int_0^{+\infty} e^{(c-a)r}dr
\end{align*}
where the last estimate follows from Lemma \ref{gen_count}. Taking $a>c$ it follows that $I_3$ goes to 0 as $Y\to+\infty$. 

Now we deal with $I_2$; the main point is that for any (large enough) $a$ we have an estimate 
\begin{equation} \label{est_int_horo}
\int_H |k_{\phi,p}(x,y)| dy \le C_a e^{-ad(x,H)}
\end{equation}
for any horosphere $H$ such that $x$ doe not lie in the horoball that it bounds. Let us prove this: let $y_0$ be the projection of $x$ on $H$; we have  
\[
d(x,y) \ge \frac 1 2 (d(x,H) + d(y_0,y))
\]
for all $y\in H$ (indeed, since $d(y,x)\ge d(y_0,x)=d(x,H)$ it holds trivially if either $d(y,y_0)\ge d(x,H)=d(x,y_0)$ or $d(y,y_0)\le d(x,H)$) and since we supposed that $\phi \in \mathcal A(\RR)$ (see Lemma \ref{Plancherel} and the remark afterwards) for $A = 2a$ we get that 
\[
|k_{\phi,p}(x,y)| \ll_a e^{-ad(x,H)}e^{-ad(y,y_0)}
\]
for all $a>0$ and $y\in H$; since the integral $\int_H e^{-ad(y,y_0)}dy$ converges for $a$ large enough we obtain \eqref{est_int_horo}. It follows that
\begin{equation} \label{hororeste}
I_2 \le C_a \int_{D-D^Y} \sum_{\substack{\gamma\in\Gamma/\Gamma_P \\ \gamma \not= \Gamma_P}} e^{-a d(x,\gamma H)} dx
\end{equation}
and by Lemma \ref{counthoro2} the inner sum is finite and uniformly bounded for $x\in D-D^1$, hence $I_2$ is finite and goes to 0 as $Y\to+\infty$. 

Finally we deal with the first term, which is more subtle. The integrand is $N$-invariant and hence it equals 
\[
I_1 = \int_Y^{+\infty} \left| \sum_{\eta\in\Lambda-\{0\}} h_{\phi,p} \left(\ell\left(\frac{|\eta|}y \right)\right) - \frac 1{\vol(\Lambda)} \int_N h_{\phi,p} \left(\ell\left(\frac{|n|}y \right)\right) dn \right| \frac{dy}{y^3}. 
\]
Recall that $\mathcal{N}_\Lambda$ is the counting function for the Euclidean lattice $\Lambda$, $\mathcal{N}_\Lambda^*=\mathcal{N}_\Lambda-1$ and $E_\Lambda(r)=\mathcal{N}_\Lambda^*-\frac{\pi r^2}{\vol(\Lambda)}\ll\frac r{\alpha_1(\Lambda)}$. Now we compute:
\begin{align*}
& \int_Y^{+\infty} \left| \sum_{\eta\in\Lambda-\{0\}} h_{\phi,p} \left(\ell\left(\frac{|\eta|}y \right)\right) - \frac 1{\vol(\Lambda)} \int_N h_{\phi,p} \left(\ell\left(\frac{|n|}y \right)\right) dn \right| \frac{dy}{y^3}\\
&\phantom{\int_Y^{+\infty} (\sum_{\eta\in\Lambda_j-\{0\}} h_{\phi,p}(\ell(|\eta|/y)))} 
    = \int_Y^{+\infty} \left| \int_0^{+\infty} h_{\phi,p} \left(\ell \left(\frac r y \right)\right) \left( d\mathcal{N}_{\Lambda}^*(r) - \frac{2\pi r}{\vol(\Lambda)} dr\right) \right| \frac{dy}{y^3} \\
&\phantom{\int_Y^{+\infty} (\sum_{\eta\in\Lambda_j-\{0\}} h_{\phi,p}(\ell(|\eta|/y)))} 
    = \int_Y^{+\infty} \left|\int_0^{+\infty} \frac{dh_{\phi,p}(\ell(r/y))}{dr} \left(\mathcal{N}_{\Lambda}^*(r)-\frac{\pi r^2}{\vol(\Lambda)}\right) dr \right|\frac{dy}{y^3} \\
&\phantom{\int_Y^{+\infty} (\sum_{\eta\in\Lambda_j-\{0\}} h_{\phi,p}(\ell(|\eta|/y)))} 
    \le \int_Y^{+\infty} \int_0^{+\infty} \left| \frac{dh_{\phi,p}(\ell(r))}{dr} \right| E_{\Lambda}(ry) dr \frac{dy}{y^3}.
\end{align*}
We have $E_{\Lambda}(ry) \ll (r+1)y$ as $y\to\infty$, uniformly in $r$ (see Lemma \ref{ucount}) and it follows that the right-hand side (hence $I_1$) is finite and goes to 0 as $Y\to+\infty$. 

\begin{figure} 
\centering
\includegraphics[trim=0cm 3cm 0cm 3cm, width=.7\textwidth]{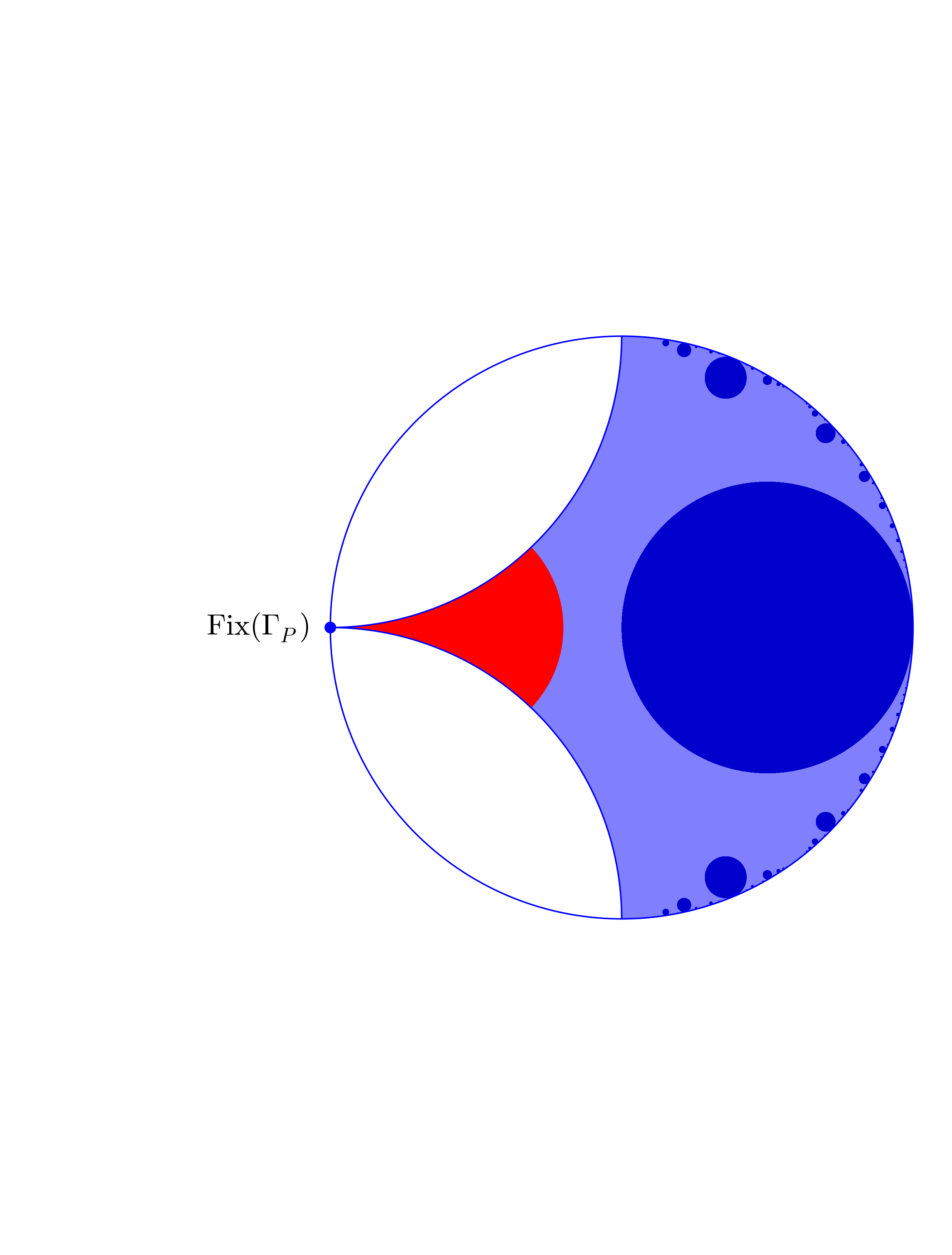}
\caption{\label{MYtilde} $B_\Lambda^Y$ is in blue, with $R$ in a darker shade.}
\end{figure}

\medskip

It remains to prove the stated asympotic expansion: what we did above shows that it suffices to prove that the integral $\int_{M^Y} \tr K_{\phi,p}^\Gamma(x,x) dx$ has such an expansion. Let $O$ be a fundamental parallelogram for $\Lambda$ in $N$ and $B_\Lambda$ the union of all geodesics from the fixed point of $N$ passing through $O$ (for example if $N = N_\infty$ is the upper triangular unipotent group which fixes $\infty$, identifying $N$ with $\CC$ we have $B_\Lambda = \{(z, t) : z \in O, t \in ]0, +\infty[ \}$). Define :
\begin{equation} \label{fundstrip}
B_\Lambda^Y = \{x\in B_\Lambda:\: y_P(x)\le Y\}. 
\end{equation}
Let $R$ be the union of the pieces of horoballs $\gamma (B_\Lambda - B_\Lambda^Y) \cap B_\Lambda$ for $\gamma \not \in \Gamma_P$ (see figure \ref{MYtilde}). Let $\wdt{M^Y}$ be the universal cover of $M^Y$, which is naturally identified with a subset of $\HH^3$. The strip $OA$ is a fundamental domain in $\HH^3$ for $\Lambda$, and it follows that $B_\Lambda^Y - R = B_\Lambda \cap \wdt{M^Y}$ is a fundamental domain in $\wdt{M^Y}$ for $\Lambda$. On the other hand, $\bigcup_{\gamma\in\Gamma/\Gamma_P} \gamma D^Y$ is also a fundamental domains in $\wdt{M^Y}$ for $\Lambda$ and it follows that we have the following expression for the sum over the unipotent elements: 
\begin{equation} \label{unfolding}
\begin{split}
\int_{M^Y} \sum_{\gamma\in\Gamma/\Gamma_P} \sum_{\eta\in\Lambda-\{0\}} h_{\phi,p}(d(x,\gamma\eta\gamma^{-1}x)) dx  &= \sum_{\gamma \in \Gamma/\Gamma_P} \int_{\gamma M^Y} \sum_{\eta\in\Lambda-\{0\}} h_{\phi,p}(d(x, \eta x)) dx \\ 
       &= \int_{B_\Lambda^Y} \sum_{\eta\in\Lambda-\{0\}} h_{\phi,p}(d(x, \eta x)) dx - \int_R \sum_{\eta\in\Lambda-\{0\}} h_{\phi,p}(d(x, \eta x)) dx. 
\end{split}
\end{equation}

We can bound $\int_R \sum_{\eta\in\Lambda-\{0\}} h_{\phi,p}(d(x, \eta x)) dx$ by using arguments similar to those used for $I_2$ above (see \eqref{hororeste}) and this shows that it is $o(1)$ as $Y \to +\infty$. The integral $\int_{M^Y} \tr K_{\phi,p}^\Gamma(x,x) dx$ can be decomposed as a sum over the elements of $\Gamma$ and using the conclusion of \eqref{unfolding} to modify the sum over unipotent elements we obtain: 
\begin{equation} \label{dec_sum} 
\begin{split}
\int_{M^Y} \tr K_{\phi,p}^\Gamma(x,x) dx &= \vol(M)\tr \left( k_{\phi,p}(x_0,x_0)\right) + \int_M\sum_{\gamma\in\Gamma_\lox} \tr \left(\gamma^*k_{\phi,p}(x,\gamma x) \right) dx \\
               &\quad + \int_{B_\Lambda^Y} \sum_{\eta\in\Lambda-\{0\}} h_{\phi,p} (d(x,\eta x)) dx + o(1). 
\end{split}
\end{equation}
Hence we need to get an asymptotic expansion when $Y\to\infty$ of:
\begin{equation} \label{dev_unip}
\int_{B_\Lambda^Y} \sum_{\eta\in\Lambda-\{0\}} h_{\phi,p} (d(x,\eta x)) dx = \int_{B_\Lambda^Y} \sum_{\eta\in\Lambda-\{0\}} h_{\phi,p} \left(\ell\left( \frac{|\eta|}{y_P(x)} \right)\right) dx. 
\end{equation}
The integrand is $N$-invariant so that the integral in \eqref{dev_unip} equals
\[
\vol(\Lambda)\int_0^Y \sum_{\eta\in\Lambda-\{0\}} h_{\phi,p}\left( \ell \left(\frac{|\eta|}y \right) \right) \frac{dy}{y^3}
\]
and by substituting $r = |\eta|/y$ in the right-hand side we obtain the following expression : 
\begin{equation} \label{qgeftqiazfyu}
\vol(\Lambda) \sum_{\eta\in\Lambda-\{0\}} \int_{|\eta|/Y}^{+\infty} h_{\phi,p}(\ell(r))\frac {rdr}{|\eta|^2} = \vol(\Lambda) \int_0^{+\infty} rh_{\phi,p}(\ell(r))\sum_{\substack{\eta\in\Lambda-\{0\}\\0<|\eta|\le rY}} \frac 1{|\eta|^2} dr.
\end{equation}
On the other hand, for any $R>0$ we get from integrating by parts (or Abel summation) that:
\[
\sum_{\substack{v\in\Lambda \\0<|v|\le R}}\frac 1{|v|^2} = \int_{\alpha_1(\Lambda)}^R \frac{d\mathcal{N}_\Lambda^*(\rho)}{\rho^2} = \frac{\mathcal{N}_\Lambda^*(R)}{R^2} + \int_{\alpha_1(\Lambda)}^R 2\frac{\mathcal{N}_\Lambda^*(\rho)}{\rho^3}d\rho
\]
and since we have
\[
\mathcal N^*(\rho) = \frac{\pi\rho^2}{\vol(\Lambda)} + E_\Lambda(\rho)
\]
we get that:
\begin{equation} \label{gadgzefiia}
\begin{split}
  \sum_{\substack{v\in\Lambda \\0<|v|\le R}}\frac 1{|v|^2} &= \frac{\pi}{\vol(\Lambda)} + \frac{E_\Lambda(R)}{R^2} + \frac{2\pi(\log(R)-\log\alpha_1(\Lambda))}{\vol(\Lambda)} + \int_{\alpha_1(\Lambda)}^R E_\Lambda(\rho)\frac{d\rho}{\rho^3} \\
  &= \frac{\pi}{\vol(\Lambda)} + \frac{E_\Lambda(R)}{R^2} + \frac{2\pi(\log(R)-\log\alpha_1(\Lambda))}{\vol(\Lambda)} + \int_{\alpha_1(\Lambda)}^{+\infty} E_\Lambda(\rho)\frac{d\rho}{\rho^3} - \int_{\max(R, \alpha_1(\Lambda))}^{+\infty} E_\Lambda(\rho)\frac{d\rho}{\rho^3}
\end{split}
\end{equation}
where the second line follows from the fact that the integral $\int_1^{+\infty} E_\Lambda(\rho)d\rho/\rho^3$ is absolutely convergent by Lemma \ref{ucount}. Putting 
\[
\kappa_\Lambda = \int_{\alpha_1(\Lambda)}^{+\infty} E_\Lambda(\rho)\frac{d\rho}{\rho^3} + \frac{\pi(1 - 2\log\alpha_1(\Lambda))}{\vol(\Lambda)}
\]
we can rewrite \eqref{gadgzefiia} as :
\begin{equation}
\sum_{\substack{v\in\Lambda \\0<|v|\le R}}\frac 1{|v|^2} = \frac{2\pi\log(R)}{\vol(\Lambda)} + \kappa_\Lambda - \int_{\max(R, \alpha_1(\Lambda))}^{+\infty} E_\Lambda(\rho) \frac{d\rho}{\rho^3} + \frac{E_\Lambda(R)}{R^2}. 
\end{equation}
Plugging this into \eqref{qgeftqiazfyu} we obtain: 
\begin{equation}
\begin{split}
\int_{B_\Lambda^Y} \sum_{\eta\in\Lambda-\{0\}} h_{\phi,p}(d(x,\eta x)) dx 
      &= \int_0^{+\infty} 2\pi\log(rY)r h_{\phi,p}(\ell(r)) dr + \kappa_\Lambda\vol(\Lambda) \int_0^{+\infty} h_{\phi,p}(\ell(r))dr \\
      &\quad - \vol(\Lambda) \int_0^{+\infty} rh_{\phi,p}(\ell(r)) \int_{\max(\alpha_1(\Lambda),rY)}^{+\infty} E_\Lambda(\rho)\frac{d\rho}{\rho^3} \, dr \\
      &\quad + \vol(\Lambda) \int_0^{+\infty} rh_{\phi,p}(\ell(r)) \frac{E_\Lambda(rY)}{(rY)^2} dr. 
\end{split}
\label{geomuni}
\end{equation}
The terms on the second and third lines are $O(Y^{-1})$, and plugging this expansion in \eqref{dec_sum} finishes the proof.
\end{proof}


\subsection{Spectral side}

The decomposition $L^2 = L_\disc^2 \oplus L_\eis^2$ from \eqref{spectral_decomposition} induces a splitting of the operators $T^YK_{\phi, p}^\Gamma$ into $T^Y(K_{\phi, p}^\Gamma)_\disc \oplus T^Y(K_{\phi, p}^\Gamma)_\eis$. It is well-known that the operators $(K_{\phi,p}^\Gamma)_\disc$ are trace-class (see e.g. \cite[Theorem 4.3]{Warner}). All these operators have integrable kernels and we have
\[
\otr(K_{\phi,p}^\Gamma)_\disc = \int_M (K_{\phi,p}^\Gamma)_\disc(x, x) dx.
\]
We will denote by $\otr(T^YK_{\phi, p}^\Gamma)$ the integral of the kernel $T^YK_{\phi, p}^\Gamma$ on $M$. We have computed it from the geometric expansion in Proposition \ref{geom_side}, now we will use the Maass--Selberg relations to compute it from the spectral decomposition. We note that our computation is essentially the same as that of the ``third parabolic term'' in \cite[Section 4]{Warner}---see especially p. 85 in loc. cit. 

\begin{prop} \label{spec_side}
For any $Y\in [1,+\infty[^h$ we have the following asymptotic expansions as $\min_j(Y_j)\to+\infty$ (we put $d_l = \dim(W_l)$): 
\begin{align}
\begin{split} \label{traceexpansion0} 
\otr(T^Y K_{\phi,0}^\Gamma) &= \sum_{j=1}^h \frac {\log Y_j}{\pi} \int_{-\infty}^{+\infty} \sum_{l=-2q}^{2q} d_l \phi\left(l^2 + u^2 + \lambda_V \right) du  \\ 
                         &\quad - \frac 1{2\pi} \int_{-\infty}^{+\infty} \sum_{l=-2q}^{2q} \phi\left(l^2 + u^2 + \lambda_V \right) \tr \left( \Psi_l(iu)^{-1}\frac{d\Psi_l(iu)}{du} \right) du  \\ 
                         &\quad + \otr (K_{\phi,0}^\Gamma)_\disc + \frac 1 4 \sum_{l=-2q}^{2q} \phi\left(l^2 + \lambda_V \right) \tr\Psi_l(0) + o(1)
\end{split}
\\
\begin{split} \label{traceexpansion1}
\otr(T^Y K_{\phi,1}^\Gamma)    &= \otr(T^Y (K_{\phi,0}^\Gamma)_\eis) + \sum_{j=1}^h \frac {2\log Y_j}{\pi} \int_{-\infty}^{+\infty}  \sum_{l=-2q-2}^{2q+2} d_l \phi\left((l\mp 2)^2-4+u^2+\lambda_V \right) du \\
                            &\quad +\otr (K_{\phi,1}^\Gamma)_\disc - \frac 1{2\pi} \int_{-\infty}^{+\infty} \sum_{l=-2q-2}^{2q+2} \phi\left((l\mp 2)^2 + u^2 + \lambda_V \right) \tr(\Phi_l(iu)^{-1}\frac{d\Phi_l(iu)}{du}) du + o(1) 
\end{split}
\end{align}
(here $\otr(T^Y (K_{\phi,0}^\Gamma)_\eis)$ is the trace of the restriction of $T^Y (K_{\phi,0}^\Gamma$ to the subspace $L^2_\eis$ spanned by Eisenstein series and it is given by \eqref{traceexpansion0} minus the term $\otr (K_{\phi,0}^\Gamma)_\disc$). 
\end{prop}

\begin{proof}
We can compute the operation of automorphic kernels on the continuous part of $L^2(M;V)$ in the following way. Let $\psi\in L^2(\RR)$ and $v\in W_l$, we have:
\[
K_{\phi,p}^\Gamma E(\psi,\omega) = \int_{-\infty}^{+\infty}\phi \left((l\mp 2)^2 + u^2 + \lambda_V \right) \psi(u)E(iu,\omega)du
\]
Put $d=\dim V$, choose an orthonormal basis $v_k,\: k=1,\ldots,dh$ for $V$ where all $v_k\in (W_{l_k})^h$. From the preceding identity and \eqref{KYF=KFY} it follows that
\[
 \otr(T^Y (K_{\phi,0}^\Gamma)_\eis) = \int_{-\infty}^{+\infty} \sum_{k=1}^{dh} \phi \left(l_k^2 + u^2 + \lambda_V \right) \|T^Y E(iu, v_k)\|^2 du. 
\]
Now expanding $ \|T^Y E(iu,v_k)\|^2$ using the Maass-Selberg reletions \eqref{MS0} yields:
\[
\begin{split}
\otr(T^Y K_{\phi,0}^\Gamma) &= \otr (K_{\phi,0}^\Gamma)_\disc \\
                         &\quad + \sum_{j=1}^h \frac {\log Y_j}{\pi} \int_{-\infty}^{+\infty} \sum_{l=-2q}^{2q} d_l \phi\left(l^2 + u^2 + \lambda_V \right) du\\
                            &\quad - \frac 1{2\pi} \int_{-\infty}^{+\infty} \sum_{l=-2q}^{2q} \phi\left(l^2 + u^2 + \lambda_V \right) \tr \left( \Psi_l(iu)^{-1}\frac{d\Psi_l(iu)}{du} \right) du\\
                            &\quad + \frac 1{2\pi} \sum_{j=1}^h \sum_{l=-2q}^{2q} \phi\left(l^2 + u^2 + \lambda_V \right)\int_{-\infty}^{+\infty}\frac{Y_j^{iu}\tr\Psi_l(-iu)-Y_j^{-iu}\tr\Psi_l(iu)}{iu} du + o(1)
\end{split}
\]
To deduce \eqref{traceexpansion0} we must deal with the last line: but a classical computation (cf. \cite[Proposition 5.3 in Chapter 6]{EGM}) shows that for any function $\xi\in\mathcal{S}(\mathbb{R})$ one has
\[
\lim_{Y\rightarrow\infty}\left(\int_{-\infty}^{+\infty} \xi(u)\frac {Y^{2iu}\tr\Psi_l(-iu)-Y^{-2iu}\tr\Psi_l(iu)}{2iu}du\right) = \frac 1 4 \xi(0)\tr\Psi_l(0) 
\]
and hence we are finished. The proof of \eqref{traceexpansion1} is exactly similar, using \eqref{MS1} in addition. 
\end{proof}


\subsection{Trace formula}

The output of the work done in the previous two subsections is the following result, an avatar of the Selberg Trace Formula. We do not push further the analysis of the loxodromic summands on the geometric side since we will not need it. 

\begin{theo} \label{Selberg}
For any $\phi\in{\mathcal A}(\RR)$ the operators $(K_{\phi,p}^\Gamma)_\disc$ are trace-class and we have the equality for $p=0$:
\begin{align*}
&\otr (K_{\phi,0})_\disc    \\
&\quad - \frac 1{2\pi} \int_{-\infty}^{+\infty} \sum_{l=-2q}^{2q} \phi\left(l^2 + u^2 + \lambda_V \right) \tr(\Psi_l(iu)^{-1}\frac{d\Psi_l(iu)}{du}) du &   \\
&\quad + \frac 1 4 \sum_{l=-2q}^{2q} \phi\left(l^2 + \lambda_V \right) \tr\Psi_l(0) \\
&\qquad\qquad =\otr_\Gamma k_{\phi,0} + \int_M \sum_{\gamma\in\Gamma_\lox} \tr(\gamma^*k_{\phi,0}(x,\gamma x)) dx \\
&\qquad\qquad\quad + 2\pi h \int_0^{+\infty} r\log(r)h_{\phi,0}(\ell(r))dr  + \sum_{j=1}^h \kappa_j \vol(\Lambda_j)\int_0^{+\infty}rh_{\phi,0}(\ell(r))dr. 
\end{align*}
A similar equality holds for $p=1$, replacing the right-hand side above by the appropriate spectral terms according to \eqref{traceexpansion1}. 
\end{theo}

\begin{proof}
Let $B',B$ denote respectively the right-hand side and the left-hand side of the equality in the statement; the equality between $B$ and $B'$ follows from the fact that we have written the expansion as $\min_j(Y_j)\to+\infty$ of $\otr (T^YK_{\phi,p}^\Gamma)$ as either $A\log Y + B + o(1)$ (Proposition \ref{geom_side}) and $A'\log Y + B' + o(1)$ (Proposition \ref{spec_side}). 
\end{proof}


\subsection{Asymptotics of regularised traces}

Let $M=\Gamma\bs\HH^3$ be a finite-volume hyperbolic three--manifold. For a function $\phi\in{\mathcal A}(\RR)$ we define $\otr_R\left(\phi(\Delta^p[M])\right)$, which we will also denote by $\otr_R K_{\phi,p}^\Gamma$, to be either side of the equality in Theorem \ref{Selberg}. The convenient form in which we wrote the trace formula allows the following result to be proven very easily. 

\begin{theo} \label{conv1}
Let $\Gamma_n$ be a sequence of torsion-free lattices in $G$ which contain no element with trace $-2$ and such that $M_n=\Gamma_n\bs\HH^3$ is BS-convergent to $\HH^3$. Suppose that the height functions on the $M_n$ are chosen such that 
\begin{equation} \label{norm_height}
\sum_{j=1}^{h_n} \left|\log\alpha_1(\Lambda_{j,n})\right| = o(\vol M_n) 
\end{equation}
(where the notation is as in Lemma \ref{hypBS}). Suppose also that the following condition holds:
\begin{equation} \label{square!}
\sum_{j=1}^{h_n} \frac{\alpha_2(\Lambda_{n,j})^2}{\alpha_1(\Lambda_{n,j})^2} = o(\vol M_n). 
\end{equation}
Then we have the limit
\begin{equation} \label{convtr1}
\lim_{n\to\infty} \frac {\otr_R (K_{\phi, p}^{\Gamma_n})}{\vol(M_n)} = \otr^{(2)}(k_{\phi, p}). 
\end{equation}
\end{theo}

\subsubsection{Remarks}
\begin{enumerate}
\item For cusp-uniform sequences, the condition \eqref{square!} reduces to $h_n=o(\vol M_n)$, which always holds for BS-convergent sequences by Lemma \ref{nbcuspsBS}. 

\item The hypothesis \eqref{norm_height} on the height functions is satisfied if we take a sequence of covers of some given orbifold $M$ and the pull-back of the height functions on $M$ (see Lemma \ref{sumalpha} below). 

\item If $y_n,y_n'$ are two height functions on $M_n$ which both satisfy \eqref{norm_height} then we have $\log(y_n/y_n')=o(\vol M_n)$ (indeed, high enough in the $j$th cusp the function $y_n/y_n'$ is constant and equals $\alpha_1(\Lambda_{n,j}) / \alpha_1'(\Lambda_{n,j})$). 
\end{enumerate}

\subsubsection{Proof of Theorem \ref{conv1}}
Let $h_n$ be the number of cusps of $M_n$; we choose representatives $P_1,\ldots,P_{h_n}$ of the $\Gamma_n$-classes of $\Gamma$-rational parabolic subgroups and s before denote by $\Lambda_{n,j}$ the Euclidean lattice $(\Gamma_n)_{P_j}$ inside $N_j$ where $N_j$ is the unipotent radical of $P_j$, identified with the horosphere $\{y_{P_j}=1\}$. 

For $p=0,1,2,3$ we have $3d\ge\dim V\otimes\wedge^p\LP$ so that 
\begin{equation} \label{kernel_majo}
\tr(\gamma^*k_{\phi, p}(x,\gamma x))\le 3d |\rho(\gamma^{-1})|_V |k_{\phi, p}(x,\gamma x)|. 
\end{equation}
For $x = gK,y = g'K\in\HH^3$ we put $H(d(x,y)) = 3d|\rho(g^{-1}g')| \cdot |k_{\phi,p}(x,y)|$, then we have $H(r)\ll e^{-ar}$ for all $a>0$ as $r\to\infty$. We first want to estimate:
\[
H_n = \int_{M_n}\sum_{\gamma\in\Gamma_\lox} H(d(x,\gamma x)) dx.
\]
which is done in the following lemma.

\begin{lem} \label{limhyp}
If $M_n$ BS-converges to $\HH^3$ then $H_n = o(\vol M_n)$. 
\end{lem}

\begin{proof}
For $a$ large enough (so that all the integrals below are absolutely convergent) we have
\begin{align*}
H_n &\ll \int_{M_n}\int_0^{+\infty} e^{-ar}d\mathcal{N}_{\Gamma_n}(x,r) dx \\
    & = a \int_{M_n} \int_0^{+\infty} e^{-ar}\mathcal{N}_{\Gamma_n}(x,r) dr dx = a\int_0^{+\infty} e^{-ar}\int_{M_n}\mathcal{N}_{\Gamma_n}(x,r)dx dr. 
\end{align*}
If we add the hypothesis that $\sys(\Gamma_n)\ge\delta>0$ for all $n$ then the lemma is a consequence of \eqref{loxcount} (which imply that the sequence of functions $r\mapsto e^{-ar}\int_{M_n}\mathcal{N}_{\Gamma_n}(x,r)/\vol(M_n) dx$ is dominated), Lemma \ref{hypBS} and Lebesgue's theorem. In general one needs to study in addition the integral on the Margulis tubes near small closed geodesics; this is carried out in the proof of Theorem 7.14 in \cite{7S}. 
\end{proof}

We have:
\begin{equation} \label{cusp_facile}
C_n := 2\pi h_n \int_0^{+\infty}r\log(r)h(\ell(r))dr = o(\vol M_n) 
\end{equation}
by Lemma \ref{nbcuspsBS}. To conclude we need also the following asymptotic estimate: 
\begin{equation} \label{cusp_dur}
U_n := \sum_{j=1}^{h_n} \kappa_{n, j}\vol(\Lambda_{n,j})\int_0^{+\infty}rh(\ell(r))dr = o(\vol M_n)
\end{equation}
where we denote
\[
\kappa_{n, j} := \kappa_{\Lambda_{n, j}} = \int_{\alpha_1(\Lambda_{n,j})}^{+\infty} E_{\Lambda_{n,j}}(\rho)\frac{d\rho}{\rho^3} + \frac{\pi(1-2\log\alpha_1(\Lambda_{n,j})}{\vol(\Lambda_{n,j})} .
\]
We get from Lemma \ref{ucount} the following estimate: 
\begin{align*}
\int_{\alpha_1(\Lambda_{n,j})}^{+\infty} E_{\Lambda_{n,j}}(\rho)\frac{d\rho}{\rho^3} 
            &\ll \frac 1{\alpha_1(\Lambda_{n,j})} \int_{\alpha_1(\Lambda_{n,j})}^{+\infty} (\rho + \alpha_2(\Lambda_{n,j})) \frac{d\rho}{\rho^3} \\
            &= \frac 1{\alpha_1(\Lambda_{n,j})^2} + \frac{\alpha_2(\Lambda_{n,j})}{\alpha_1(\lambda_{n,j})^3} \\
            &\ll \frac 1{\vol(\Lambda_{n,j})} \left( \frac{\alpha_2(\Lambda_{n,j})}{\alpha_1(\Lambda_{n,j})} + \frac{\alpha_2(\Lambda_{n,j})^2}{\alpha_1(\Lambda_{n,j})^2} \right)
\end{align*}
with a constant that does not depend on $n$ or $j$, and it follows that 
\[
\kappa_{n, j}\vol(\Lambda_{n,j}) \ll \frac{\alpha_2(\Lambda_{n,j})^2}{\alpha_1(\Lambda_{n,j})^2}+\log\alpha_1(\Lambda_n^j). 
\]
so that by the hypothesis of the theorem, 
\[
U_n\ll \sum_{j=1}^{h_n} \log\alpha_1(\Lambda_n^j) + o(\vol M_n)  
\] 
and the right-hand side is an $o(\vol M_n)$ according to the assumption on the height functions.

Now the summands in $\otr_R K_{\phi, p}^{\Gamma_n}$ given by the right-hand side of the trace formula in Theorem \ref{Selberg} are, with the exception of $\otr_\Gamma(k_{\phi, p}) \vol(M_n)$, majorised by $U_n + C_n + H_n$ according to \eqref{kernel_majo}. So it follows from \eqref{cusp_facile}, \eqref{cusp_dur} and Lemma \ref{limhyp} that
\[
|\otr_R (K_{\phi, p}^{\Gamma_n}) - \otr_\Gamma(k_{\phi, p}) \vol(M_n)| = o(\vol M_n)
\]
which proves the theorem.  


\subsubsection{Height functions in coverings}

\begin{lem} \label{sumalpha}
Suppose that $M_n$ is a sequence of finite covers of a finite--volume orbifold $M$ and that the height functions are pulled back from those (chosen arbitrarily) on $M$. Then $\sum_{j=1}^{h_n} \alpha_1(\Lambda_{n,j}) = o(\vol M_n)$.
\end{lem}

\begin{proof}
We show that for all $C>0$ we have
\begin{equation} \label{limsup}
\limsup_n\frac{\sum_{j=1}^{h_n} \alpha_1(\Lambda_{n,j})}{\vol M_n} \le C^{-1}.
\end{equation}
We order the $P_j$ so that $\alpha_1(\Lambda_{n,j})$ is increasing with $j$ and denote by $h_n^C$ the largest index such that $\alpha_1(\Lambda_{n,j})<C$ for all $j\le h_n^C$. Then:
\begin{align*}
\sum_{j=1}^{h_n} \alpha_1(\Lambda_{n,j}) & \ll Ch_n^C + \sum_{j=h_n^C+1}^{h_n} (\alpha_1(\Lambda_{n,j}))^{-1} [\Lambda_j:\Lambda_{n,j}] \ll Ch_n + (\alpha_1(\Lambda_{n,h_n^C+1})^{-1} \sum_{j=h_n^C+1}^{h_n} [\Lambda_j:\Lambda_{n,j}] \\
              & \le C h_n + C^{-1}h_1[\Gamma:\Gamma_n]
\end{align*}
where $h_1$ is the number of cusps of $M$. The conclusion \eqref{limsup} then follows at once from Lemma \ref{nbcuspsBS}. 
\end{proof}


\section{Analytic torsion and approximation}

\label{torsion}
From now on we fix a strongly acyclic representation $\rho,V$ of $G$ and all forms are taken with coefficients in $E_\rho$. We will define the regularised torsions $T_R(M;V)$ and the $L^2$-torsion $t^{(2)}(V)$ in Section \ref{sec_deftors} below, and prove the following result. 

\begin{theo} \label{Main1}
Let $M_n:=\Gamma_n\backslash\HH^3$ be a sequence of finite--volume hyperbolic three--manifolds (together with height functions) satisfying the assumptions of Theorem \ref{conv1}. Suppose in addition that the systole of the $M_n$ is bounded away from 0, and that there exists $\eps>0$ and a sequence $a_n=o(\vol M_n)$ such that for all $n\ge 1$, $l = -2q, \ldots, 2q$ and $u\in[-\eps,\eps]$ we have
\[
\tr\left( \Psi_l(iu)^{-1}\frac{d\Psi_l(iu)}{du} \right) \le a_n, \: \tr\left( \Phi_l(iu)^{-1}\frac{d\Phi_l(iu)}{du} \right) \le a_n.
\]
Then we have 
\begin{equation} \label{limtors2}
\lim_{n\to\infty}\frac{T_R(M_n;V)}{\vol(M_n)} = t^{(2)}(V).
\end{equation}
\end{theo}

Note that the condition on the intertwining operators holds (or not) independantly of the choice of the height functions satisfying condition \eqref{norm_height} (see Remark 3 after the Theorem \ref{conv1}). 


\subsection{Heat kernels}

For $\phi(u)=e^{-tu}$ the kernel $k_{\phi,p}$ (resp. $K_{\phi,p}^{\Gamma}$) is called the heat kernel of $\HH^3$ (resp. of $M$). We will use the bounds for the heat kernel given by the following result (see for example \cite[Lemma 3.8]{BV}).

\begin{prop} \label{cly}
Let $\rho$ be a finite-dimensional representation of $\SL_2(\CC)$ and $E_\rho$ the associated $\SL_2(\CC)$-equivariant Hermitian bundle on $\HH^3$ (see \ref{systloc}). Let $t_0>0$; there exists a constant $C$ depending only on $t_0$ such that for all $x,y\in X$ and $t\in ]0,t_0[$ we have
\[
|e^{-t\Delta^p[\HH^3]}(x,y)|\le Ct^{-d/2}e^{-\frac{d(x,y)^2}{5t}}.
\]
\end{prop}

We will also make use of the following fact about the heat kernel (see \cite[Theorem 2.30]{Berline_Getzler_Vergne}.

\begin{prop} \label{deart}
There exists $\alpha_k^p\in C^\infty(G,\End(\wedge^p\LP\otimes V))$ such that for all $x\in\HH^3$ we have the asymptotic expansion at $t\to 0$
\[
g^*e^{-t\Delta^p[\HH^3]}(x,gx) = t^{-\frac 3 2} \sum_{k = 0}^{m+1} \alpha_k^p(g) e^{-\frac{d(x,gx)^2}{4t}} t^k + O(t^{m + \frac 1 2}).
\]
Moreover the term  $\alpha_0^p(g)$ equals $g^*\tau_{gx}(x)$ where $\tau_y(x)$ denotes parallel transport from $x$ to $y$ along the unique geodesic arc between them. 
\end{prop}


\subsection{Asymptotic expansion of the heat kernel at $t\to 0$}

We will need the following result to define the regularised analytic torsion. Note that \cite[Proposition 6.9]{MP} prove a more precise result where all coefficients $b_k^p$ are shown to vanish for odd $k$. 

\begin{prop} \label{dea}
For all $p=0,1,2,3$ and $m\ge 1$ there are coefficients $a_0^p,\ldots,a_m^p,b_0^p,\ldots,b_m^p$ and a function $H^p$ such that 
\begin{equation}
\otr_R(e^{-t\Delta^p[M]}) = t^{-\frac 3 2} \sum_{k=0}^{m+2} a_k^p t^{\frac k 2} + b_0^p t^{-\frac 1 2}\log(t) + t^{-\frac 1 2} \sum_{k=2}^m b_k^pt^{\frac k 2}\log(t) + H^p(t)
\end{equation}
and $H^p(t)\ll t^{\frac{m+1} 2}$ as $t\rightarrow 0$.
\end{prop}

\begin{proof}
We fix $p$ and put $h_t^p(\ell)=\tr(n^*e^{-t\Delta^p[M]}(x,nx))$ for a unipotent element $n\in G$ and a $x\in\HH^3$ such that $d(x,nx)=\ell$ (cf. \eqref{defh}). We choose a fundamental domain $D$ for $\Gamma$ and define 
\begin{align*}
  S_1(t)  &= 2\pi h \int_0^{+\infty} r\log(r)h_t^p(\ell(r))dr, \\
  S_2(t)  &= \sum_{j=1}^h \kappa_j\vol(\Lambda_j) \int_0^{+\infty} rh_t^p(\ell(r)) \, dr \\
  S_3(t)  &= \int_D \sum_{\gamma\in\Gamma_\lox} \tr(\gamma^*e^{-t\Delta^p[M]}(x,\gamma x))
\end{align*}
so that by Proposition \ref{geom_side} we have:
\begin{equation} \label{geomsidet}
\otr_R e^{-t\Delta^p[M]} = \otr_{\Gamma}e^{-t\Delta^p[\HH^3]}+S_1(t)+S_2(t)+S_3(t). 
\end{equation} 
Putting $g=1_G$ in Proposition \ref{deart} and integrating over $D$ we get an expansion
\begin{equation} \label{L2exp}
\otr_{\Gamma}e^{-t\Delta^p[\HH^3]} = \vol(D) \sum_{k=-3}^{m} f_k^p t^{\frac k 2} + O(t^{\frac{m+1}2})
\end{equation}
where the $f_k^p$ are absolute coefficients, which takes care of the first summand.

Now we deal with $S_3$; if we put $\ell_0=\sys(\Gamma)$ we get:
\begin{equation} \label{majhypt}
\begin{split}
\sum_{\gamma\in\Gamma_\lox} \tr(\gamma^*e^{-t\Delta^p[M]}(x,\gamma x)) 
     &\le \sum_{\gamma\in\Gamma_\lox} C t^{-\frac 3 2}e^{-\frac{d(x,\gamma x)^2}{5t}} = C \int_{\ell_0}^{+\infty} t^{-\frac 3 2} e^{-\frac{\ell^2}{5t}} d\mathcal{N}_\Gamma(x,\ell) \\
     &= C \int_{\ell_0}^{+\infty} t^{-\frac 5 2}\ell e^{-\frac{\ell^2}{5t}} \mathcal{N}_\Gamma(x,\ell) d\ell \\
     &\ll  t^{-\frac 5 2}e^{-\frac{\ell_0^2}{5t}}
\end{split}
\end{equation} 
so that $S_3(t)$ is actually an $o(t^{m + \frac 1 2})$ for any $m>0$. 

To deal with $S_1$ and $S_2$ we will use the following expansion at $t\to 0$ (which follows immediately from Proposition \ref{deart})
\begin{equation} \label{dah}  
h_t^p(\ell) = t^{-\frac 3 2} \sum_{k=0}^m b_k^p(\ell) e^{-\frac{\ell^2}{4t}} t^k + O(t^{m + \frac 12}) 
\end{equation}
together with an elementary lemma in real analysis. 

\begin{lem}
Let $\omega$ be smooth in a neighbourhood of $[0,1]$. For every integer $m>1$ there are constants $c_l,c_l'$ for $\, l=1,\ldots,m+1$ (depending on $\omega$) such that 
\begin{equation} \label{da}
\int_0^1 r\log(r)\omega(r) e^{-\ell(r)^2/4t} dr-\sum_{l=2}^m t^{l/2}(c_l+c_l' \log t) \le c_{m+1}t^{(m+1)/2} 
\end{equation}
\end{lem}

\begin{proof}
Since $r\mapsto\ell(r)$ is a smooth diffeomorphism of $[0,+\infty[$ and the function $\ell\mapsto r\log(r)/\ell(r)\log(\ell(r))$ is a smooth function near 0, by the change of variable from $r$ to $\ell(r)$ we are reduced to showing that for a smooth function $\omega_0$ on $[0,1]$ there an expansion of the right form as $t\to 0$ of:
\[
\int_0^1 \ell\log(\ell) \omega_0(\ell) e^{-\frac{\ell^2}t} d\ell = \frac {t\log t}2 \int_0^{t^{-1/2}} \ell\omega_0(t^{\frac 1 2}\ell) e^{-\ell^2} d\ell + t \int_0^{t^{-1/2}} \ell\log(\ell)\omega_0(t^{\frac 1 2}\ell) e^{-\ell^2} d\ell
\] 
This is an immediate consequence of Taylor's formula applied to $\omega_0$ at 0 and of the following easy estimate:
\[
\int_0^{t^{-\frac 1 2}} \ell^ke^{-\ell^2}d\ell = \int_0^{+\infty}\ell^ke^{-\ell^2}d\ell+O(t^{-\frac k 2} e^{-\frac 1 t}).
\]
\end{proof}

We get an expansion similar to \eqref{da} (but without the $\log t$ factor) for $\int_0^1 r\omega(r) e^{-\frac{\ell(r)^2}{4t}} dr$ using the same argument. We finally get that for all $m\ge 1$ there are coefficients $c_k^p,d_k^p,e_k^p$ such that we have at $t\to 0$
\begin{equation} \label{prol23}
\begin{split}
\int_0^{+\infty} r\log(r)h_t^p(\ell(r))dr &= \sum_{k=-3}^m c_k^p t^{\frac k 2} + \sum_{k=-1}^m d_k^p t^{\frac k 2}\log t + O(t^{\frac{m+1}2}) ,\\
\int_0^{+\infty} rh_t^p(\ell(r))dr &= \sum_{k=-3}^m e_k^p t^{\frac k 2} + O(t^{\frac{m+1}2})
\end{split}
\end{equation}
and from this, \eqref{geomsidet}, \eqref{L2exp} and \eqref{majhypt}, we get the following expansion for the regularised trace : 
\begin{equation} \label{predea}
\otr_R(e^{-t\Delta^p[M]}) = t^{-\frac 3 2} \sum_{k=0}^{m+2} a_k^p t^{\frac k 2} + t^{-\frac 1 2} \sum_{k=0}^m b_k^pt^{\frac k 2}\log(t) + H^p(t). 
\end{equation}

It remains to prove that the coefficient $b_1^p$ in \eqref{predea} is zero. Looking at the proof of \eqref{da} we see that it comes from the degree-1 monomial in the Taylor expansion when $\ell \to 0$ of the (variable) coefficient $b_0^p(\sqrt t \ell)$ from \eqref{dah}. According to Proposition \ref{deart} we have that $b_0^p(s)$ is equal to $\tr(n_s^*\tau_{n_sx}(x))$. The map $s \mapsto n_s^*$ is the inverse of the parallel transport along the horocycle in $\HH^3$ associated to the unipotent one-parameter group $s \mapsto n_s$. To show the vanishing of $b_1^p$ we thus have to prove that parallel transport is the same up to $O(s^2)$ when done along a geodesic or horospheric arc with the same endpoints. 

Here is an explanation for why this is true. As $s \to 0$ the tangent vectors between the horosphere and the geodesic are $O(s)$-close (in any smooth metric near $x$). Taking a smooth trivialisation of $\Omega^p(M ; V)$ near $x$ the parallel transports are thus solutions of differential equations of the form $\overset{\circ} x(u) = A(u)x(u)$ (for the horosphere) and $\overset{\circ} y(u) = (A(u) + \eps_s(u))y(u)$ (for the geodesic) where $|\eps_s(u)| \le s$ for $0 \le u\le s$. We get using Taylor's formula that for $x(0) = y(0)$ we have :
\[
x(t) - y(t) = (\overset{\circ}x(0) - \overset{\circ}y(0))u + O(u^2) = u\eps_s(0)\cdot x(0) + O(u^2) \ll_{u \le s} s^2
\]
which finishes the proof. 
\end{proof}


\subsection{Definition of analytic torsions} \label{sec_deftors}

\subsubsection{Regularised torsion}

We fix a nonuniform torsion-free lattice $\Gamma$ in $G$. As usual we also denote by $\Gamma$ Euler's Gamma function defined for $\Real(s)>0$ by the formula $\Gamma(s)=\int_0^{+\infty} e^{-t}t^s\frac{dt}t$ and meromorphically continued to $\CC$. It has a simple pole at each $s=-n$ for $n\in\mathbb{N}$ and no zeroes, so that $1/\Gamma$ is holomorphic on $\CC$. We want to check that 
\begin{equation} \label{defzeta}
\zeta_p(s) := \frac 1{\Gamma(s)} \int_0^{+\infty} \otr_R(e^{-t\Delta^p[M]}) t^s \frac{dt}t 
\end{equation}
defines a holomorphic function on the half-plane $\Real(s)>3/2$ which can be continued to a meromorphic function on $\CC$ which is holomorphic at 0. The large-time convergence of the integral is ensured by the spectral gap for the Laplacian (when there's no spectral gap the integral converges only in a half-plane $\Real(s)<c<0$ and has also to be analytically continued, see \cite{Park} or \cite{MP}) as we now explain: the spectral expansion \eqref{traceexpansion1} applied to the heat kernel yields, for example for $p=1$, the following estimate as $t\to\infty$: 
\[
\otr_R e^{-t\Delta^1[M]} = \otr (e^{-t\Delta^1[M]})_\disc + e^{-t(n_1-n_2)^2} T 
\]
where $T$ is bounded as $t\to+\infty$, whence it follows that $\otr_R e^{-t\Delta[M]} \ll e^{-\lambda_0t}$ where $\lambda_0>0$ is a lower bound for the whole spectrum of all $\Delta^p[M]$ (see Proposition \ref{strongacyclicity}) as $t\to\infty$. Thus we get that for any $t_0>0$ the integral $\int_{t_0}^{+\infty}\otr_R(e^{-t\Delta^p[M]})t^s\frac{dt}t$ converges for all $s\in\CC$. An easy computation moreover yields that
\begin{equation}
\frac d{ds}\left( \frac 1 {\Gamma(s)}\int_{t_0}^{+\infty}\otr_R(e^{-t\Delta^p[M]})t^s\frac{dt}t\right)_{s=0} = \int_{t_0}^{+\infty}\otr_R(e^{-t\Delta^p[M]})\frac{dt}t.
\end{equation}

To deal with the small-time part we use the following classical lemma.

\begin{lem} \label{Gamma}
Let $\phi\in C^0(0,+\infty)$ such that there are integers $m,m'\ge 0$, coefficients $a_k, \, k=-m',\ldots,m$ and a continuous function $H$ so that 
\[
\phi(t) \underset{t\rightarrow 0}= \sum_{k=-m'}^m \left( a_k t^{k/2} + b_k t^{k/2}\log(t) \right) + H(t)
\]
with $b_0 = 0$ and $H(t)\ll t^{\frac{m+1}{2}}$ near 0. Then for all $t_0>0$ the integral $\frac 1{\Gamma(s)}\int_0^{t_0}\phi(t)t^s\frac{dt}t$ converges on the half-plane $\Real(s)>m'/2$ and the holomorphic function it defines may be meromorphically continued to a function on $\Real(s)>1/2$ which is regular at 0. 
\end{lem}

\begin{proof}
For $\alpha\in\CC$ the integral $\int_0^{t_0}t^{\alpha+s-1}dt$ converges absolutely on $\Real(s)>\alpha$ and defines a meromorphic function on $\CC$ with a single simple pole at $s=-\alpha$, and since $1/\Gamma$ has a zero at 0 and the integral $\int_0^{t_0}H(t)t^sdt/t$ converges for $\Real(s)>-m/2$ we get the first part. The formula for the derivative at 0 follows from a staightforward computation. 

The proof for the terms $\int_0^{t_0} t^{\alpha+s-1} \log(t) dt$ is the same except that we get a double pole at $s = \alpha$, thus we need to assume $b_0 = 0$ for the continuation to be holomorphic at $0$ (see also \cite[Lemma 9.35]{Berline_Getzler_Vergne}). 
\end{proof}

It follows from Proposition \ref{dea} and Lemma \ref{Gamma} that we may define the regularised determinant of the Laplacian by 
\begin{equation} \label{defdet}
\det{}_R\: \Delta^p[M]:= \exp(\zeta_p'(0))
\end{equation}
and the analytic torsion by 
\begin{equation}
T_R(M)=\left(\prod_{p=0}^3 \det{}_R(\Delta^p[M])^{(-1)^p p}\right)^{\frac 1 2} = (\det{}_R(\Delta^0[M])^{-3}\det{}_R(\Delta^1[M]))^{\frac 1 2} 
\end{equation}


\subsubsection{$L^2$-torsion}

The natural candidate to be the limit of finite torsions is the $L^2$-torsion, cf. \cite[Question 13.73]{LuckB}. The following definition does not depend on $t_0>0$:
\begin{equation} \label{defL2}
\begin{split}
\log T^{(2)}(M;V) &=  \frac 1 2  \sum_{p=1}^3 p(-1)^p \biggl(\frac{d}{ds}\biggl(\frac{1}{\Gamma(s)}\int_0^{t_0}\otr_\Gamma(e^{-t\Delta^p[\HH^3]})t^{s-1}dt\biggr)_{s=0} \\
                 &\phantom{\frac 1 2  \sum_{p=1}^3 p(-1)^p \biggl(} +\int_{t_0}^{+\infty}\otr_\Gamma(e^{-t\Delta^p[\HH^3]})\frac{dt}{t}\biggr).
\end{split}
\end{equation}
The convergence of the fist integral follows from the asymptotic expansion \eqref{L2exp}; the large-time convergence is obvious here because the Laplacian on $\HH^3$ with coefficients in $V$ has a spectral gap. We see that $\log T^{(2)}$ is a multiple of $\vol(D)$ and we denote by $t^{(2)}(V)$ the constant $\log(T^{(2)}(M,V))/\vol(M)$. This has been computed in all generality in \cite{BV} to yield \eqref{formuleL2}.


\subsection{Proof of Theorem \ref{Main1}}

\subsubsection{Plan of proof}

We naturally study small and large times separately. We want first to prove that for any $t_0>0$ the following limit holds:
\begin{equation} \label{pt} 
\frac 1{\vol(M_n)} \frac d{ds}\left(\int_0^{t_0} \otr_R \left( e^{-t\Delta^p[M_n]})-\otr_{\Gamma_n}(e^{-t\Delta^p[\HH^3]}) \right) t^s\frac{dt}t\right)_{s=0} \xrightarrow[n\to\infty]{} 0
\end{equation}
The proof of this is more involved than that of the pointwise convergence of the traces  since we have to control the asymptotics as $t\to 0$ of the heat kernels of $M_n$ as $n\to\infty$. We carry it out in \ref{sec:small_time_approx} by going to go over the steps of the proof of Theorem \ref{conv1} with extra care for the dependency in $t$ of the majorations. 

We also have to deal with the convergence of the large-time integral as $n$ varies. What we need are the following limits, which we will prove right away in \ref{gtps} below.
\begin{gather}
\lim_{t_0\to+\infty} \left( \limsup_{n\to\infty} \int_{t_0}^{+\infty} \frac{\otr_R(e^{-t\Delta^p[M_n]})}{\vol(M_n)}\frac {dt}t \right) = 0, 
\label{gtr}\\
\lim_{t_0\rightarrow\infty}\left(  \int_{t_0}^{+\infty}\otr_{\Gamma}(e^{-t\Delta^p[\HH^3]})\frac {dt}t \right) = 0.
\label{gtl2}
\end{gather}

Assuming all these limits we can now conclude the proof of Theorem \ref{Main1}: the limit \eqref{pt} above yields for all $t_0>0$
\[
\limsup_{n\to\infty}\frac{\log T_R(M_n)-\log T^{(2)}(M_n)}{\vol(M_n)} 
   \le \limsup_{n\to\infty}\left( \int_{t_0}^{+\infty} \frac{\otr_R(e^{-t\Delta^p[M_n]})}{\vol(M_n)} \frac{dt}t \right) + \int_{t_0}^{+\infty}\otr_{\Gamma}(e^{-t\Delta^p[\HH^3]})\frac {dt}t
\]
and by \eqref{gtr},\eqref{gtl2} we get that the right-hand side goes to 0 as $t_0\to+\infty$, so that the limit superior on the left must be 0.


\subsubsection{Spectral gap and large times} \label{gtps}

Obviously \eqref{gtl2} follows from the convergence of the integral. Now we prove \eqref{gtr} using the uniform spectral gap. Let us first deal with the continuous part. For $u\in[\eps,+\infty)$ the Maass-Selberg relations \eqref{MS0} for $Y=1$ yield 
\[
- \langle \Psi_l(iu)^{-1}\frac{d\Psi_l(iu)}{du} \cdot v, v\rangle_{V_\CC^h} 
      = |T^1E(iu,v)|_{L^2(M_n;V_\CC)}^2 +\frac 1{iu} (\langle \Psi_l(iu) \cdot v, v\rangle_{V_\CC^h} -\langle \Psi_l(-iu) \cdot v, v\rangle_{V_\CC^h}). 
\]
As $\Psi_l(iu)$ is unitary the right-hand side is bounded below by $-2\eps^{-1}$ and it follows that $2\eps^{-1}1_{V_\CC^h}-\Psi_l(iu)^{-1}\frac{d\Psi_l(iu)}{du}$ is positive when $|u|\ge\eps$; in particular,
\[
\xi(u) := \tr\left( 2\eps^{-1} 1_{V_\CC^{h_n}}  - \Psi_l(iu)^{-1}\frac{d\Psi_l(iu)}{du}\right) \ge 0
\]
and since for $t\ge 1$ we have $e^{-tu^2}\le e^{-u^2}$ we get:
\begin{equation} \label{computation_above}
\begin{split}
&  -\int_{|u|\ge\eps} e^{-u^2t}\tr\left(\Psi_l(iu)^{-1}\frac{d\Psi_l(iu)}{du}\right) du \\
&\qquad\qquad = -2\eps^{-1}h_n\dim(V)\int_{|u|\ge \eps}e^{-u^2t} du + \int_{|u|\ge\eps} \xi(u)e^{-u^2t}du \\
&\qquad\qquad \le -2\eps^{-1}h_n\dim(V)\int_{|u|\ge \eps}e^{-u^2t} du + \int_{|u|\ge\eps} \xi(u)e^{-u^2}du \\
&\qquad\qquad = \left(2\eps^{-1} \dim(V) \int_{|u|\ge \eps} e^{-u^2} du\right) h_n - \int_{|u\ge \eps} e^{-u^2}\tr\left(\Psi_l(iu)^{-1}\frac{d\Psi_l(iu)}{du}\right) du. 
\end{split}
\end{equation}
We put 
\[
C = \left(2\eps^{-1} \dim(V) \int_{|u|\ge \eps}e^{-u^2}du\right),
\]
recall from \eqref{Caseig} that $l^2 + \lambda_V \ge (n_1 - n_2)^2$ for all $l$ under consideration here, so that $e^{-(l^2 + \lambda_V)t} \le e^{-(n_1 - n_2)^2(t - 1)}e^{-(l^2 + \lambda_V)}$; from the last line of \eqref{computation_above} we get:
\begin{multline} \label{large_u}
-\int_{|u|\ge\eps} e^{-(u^2 + l^2 + \lambda_V)t} \tr\left(\Psi_l(iu)^{-1}\frac{d\Psi_l(iu)}{du}\right) du \\
       \le Ch_ne^{-(n_1-n_2)^2t} - e^{-(n_1-n_2)^2(t-1)} \int_{|u\ge \eps} e^{-(u^2 + l^2 + \lambda_V)}\tr\left(\Psi_l(iu)^{-1}\frac{d\Psi_l(iu)}{du}\right) du . 
\end{multline}
Since $a_n\ge \tr\Psi_l(iu)^{-1}\frac{d\Psi_l(iu)}{du}$ we also have
$$
\tr\left( \frac{a_n}{h_n\dim V} 1_{V_\CC^{h_n}}  - \Psi_l(iu)^{-1}\frac{d\Psi_l(iu)}{du}\right) \ge 0
$$
and we obtain in the same way
\begin{equation} \label{small_u}
-\int_{|u|\le\eps} e^{-u^2t}\tr\left(\Psi_l(iu)^{-1}\frac{d\Psi_l(iu)}{du}\right) du \le Ca_n - \int_{|u\le \eps} e^{-u^2}\tr\left(\Psi_l(iu)^{-1}\frac{d\Psi_l(iu)}{du}\right) du 
\end{equation}
Since $\tr\Psi_l(0)\le d_lh_n$ for all $l$, we also have 
\begin{equation} \label{psi(0)}
\sum_l e^{-t(l^2 + \lambda_V)}\tr\Psi_l(0) \le \dim(V)h_n + e^{-(t-1)(n_1-n_2)^2} \sum_l e^{-(l^2 + \lambda_V)}\tr\Psi_l(0). 
\end{equation}
Now let $\lambda_0$ be the lower bound on the spectra of all $(\Delta^p[M_n])_\disc$ given by Proposition \ref{strongacyclicity}; we may suppose $\lambda_0\le(n_1-n_2)^2$. We have $\otr(e^{-t\Delta^p[M_n]})_\disc \le e^{-\lambda_0(t-1)}\otr(e^{-\Delta^p[M_n]})_\disc$ and since $\otr_R e^{-t\Delta^p[M_n]}$ is the sum of this with the terms on the right-hand side of \eqref{large_u},\eqref{small_u} and \eqref{psi(0)} we finally obtain:
\[
\otr_R e^{-t\Delta^p[M_n]} \le e^{-\lambda_0(t-1)}\otr_R e^{-\Delta[M_n]} + O(e^{-t(n_1-n_2)^2}(a_n + h_n))
\]
from which follows:
\[
\sup_n\left(\frac 1{\vol(M_n)} \int_{t_0}^{+\infty} \otr_R(e^{-t\Delta^p[M_n]})\frac{dt}t\right) \le \int_{t_0}^{+\infty} e^{-\lambda_0 (t-1)}\frac{dt}t \cdot \sup_n\frac{\otr_R e^{-\Delta^p[M_n]}}{\vol(M_n)} + \frac{O(a_n + h_n)}{\vol(M_n)}. 
\]
By Theorem \ref{conv1}, Lemma \ref{nbcuspsBS} and the hypothesis on $a_n$ we have the right-hand side above is bounded in $n$ and goes uniformly to 0 as $t_0\to\infty$. 


\subsubsection{Small times} \label{sec:small_time_approx}

To deal with the small-time part we analyse each of the terms in \eqref{geomsidet}. Recall that for $j=1,\ldots,h_n$ we have put $\Lambda_{n, j}=(\Gamma_n)_{P_j}$ and 
\[
\kappa_{n, j} = \int_{\alpha_j^n}^{+\infty} E_{\Lambda_{n, j}}(\rho)\frac{d\rho}{\rho^3} - \frac{\pi(1+2\log\alpha(\Lambda_{n,j}))}{\vol(\Lambda_{n, j})}.
\]
Then we get that
\begin{align*}
\otr_Re^{-t\Delta^p[M_n]}-\otr_{\Gamma_n} e^{-t\Delta^p[\HH^3]} 
   &= \int_{D_n} \sum_{\gamma\in(\Gamma_n)_\lox}\tr\gamma^*e^{-t\Delta^p[\HH^3]}(x,\gamma x) dx \\ 
 &\quad + 2\pi h_n \int_0^{+\infty} r\log(r)h_t^p(\ell(r)) dr + \sum_{j=1}^{h_n}\kappa_{n,j} \vol(\Lambda_{n,j}) \int_0^{+\infty}rh_t^p(\ell(r)) dr \\
 &=: T_1+T_2.
\end{align*}

By the estimates from Theorem \ref{cly} we have 
\begin{align*}
\frac d{ds}\left(\int_0^{t_0} T_1 t^s\frac{dt}t\right)_{s=0} 
              &\ll \int_0^{t_0} \int_{D_n} \sum_{\gamma\in(\Gamma_n)_\lox}e^{-\frac{d(x,\gamma x)^2}{C_2t}}dx t^{-\frac 5 2}\frac{dt}t \\
              &= \int_{M_n}\sum_{\gamma\in(\Gamma_n)_\lox} \left( \int_0^{t_0} e^{-\frac{d(x,\gamma x)^2}{C_2t}} t^{-\frac 5 2}\frac{dt}t \right) dx
\end{align*}
and the right-hand side is an $o(\vol M_n)$ by Lemma \ref{limhyp}. 

Now we deal with $T_2$; put:
\begin{equation} \label{xitheta}
\begin{split}
\Xi(s)    &= \int_0^{t_0}\int_0^{+\infty} r\log(r)h_t^p(\ell(r)) dr t^s\frac{dt}t , \\
\Theta(s) &= \int_0^{t_0} \int_0^{+\infty}rh_t^p(\ell(r)) dr t^s \frac{dt}t.
\end{split}
\end{equation}
It follows from \eqref{prol23} and Lemma \ref{Gamma} that $\Xi,\Theta$ extend to meromorphic functions on $\CC$ which are holomorphic at 0 and we get
\[
\frac d{ds} \left(\int_0^{t_0} T_2 t^s\frac{dt}t\right)_{s=0} = 2\pi h_n\frac{d\Xi}{ds}(0) + \left(\sum_{j=1}^{h_n}\kappa_{n,j}\vol(\Lambda_{n,j})\right) \frac{d\Theta}{ds}(0).
\]
On the other hand we have seen that $\sum_{j=1}^{h_n}\kappa_{n,j}\vol(\Lambda_{n,j}) = o(\vol M_n)$ in the proof of Theorem \ref{conv1}, and $h_n=o(\vol M_n)$ by Lemma \ref{nbcuspsBS}, so that the right-hand side itself is $o(\vol M_n)$, which concludes the proof.


\section{The asymptotic Cheeger--M\"uller equality: statement}

\label{CMA_stat}
In this section we recall the definition of absolute analytic torsion for manifolds with boundary and we give the scheme of proof for the following theorem. The actual work is done in Sections \ref{CMA_small} and \ref{CMA_large} below. 

\begin{theo} \label{Main2}
Let $V$ be a strongly acyclic representation of $G$ and $M_n=\Gamma_n\backslash\HH^3$ a sequence of finite-volume hyperbolic three--manifolds satisfying the conditions of Theorem \ref{Main1}, with \eqref{square!} replaced by the stronger condition that we have
\begin{equation} \label{square2}
\sum_{j=1}^{h_n} \left(\frac{\alpha_1(\Lambda_{j,n})}{\alpha_2(\Lambda_{j,n})}\right)^2 \ll \frac{\vol(M_n)}{(\log(\vol M_n))^{20}}. 
\end{equation}
Then for 
\[
Y_j^n = \left( \frac{\vol(M_n)}{\sum_{j=1}^{h_n} \frac{\alpha_2(\Lambda_{n,j})^2}{\alpha_1(\Lambda_{n,j})^2}} \right)^{\frac 1{10}} \cdot \alpha_1(\Lambda_{n,j}).
\]
the following limit holds.
\begin{equation}
\frac{\log T_R(M_n;V) - \log T_\abs(M_n^{Y^n};V)}{\vol(M_n)} \xrightarrow[n\to\infty]{} 0.
\end{equation}
\end{theo}

Once we accept this result we can deduce an asymptotic equality between regularised torsion and a \emph{combinatorial} absolute torsion that we will define in \eqref{deftau} below. The other important ingredient for this is the generalisation by Br\"uning and Ma \cite{Bruening_Ma1} of the Cheeger--M\"uller equality to the case of flat bundles on manifolds with boundary.

\begin{theo} \label{CMA}
Notations as above, we have the limit
\begin{equation}
\frac{\log T_R(M_n;V) - \log \tau_\abs(M_n^{Y^n};V)}{\vol(M_n)} \xrightarrow[n\to\infty]{} 0.
\end{equation}
\end{theo}


\subsection{Absolute torsions} \label{sec_abs}

\subsubsection{Analytic torsion}

Let $X$ be a compact Riemannian manifold with boundary and $V$ a real flat vector bundle on $X$ with a Euclidean metric. Then the space $\Omega^p(X;V)$ of smooth $p$-forms on $X$ with coefficients in $V$ is  operated upon by the Hodge Laplacian $\Delta^p[X]$. The restriction of $\Delta^p[X]$ to the forms satisfying absolute conditions on the boundary (i.e. the boundary restrictions of  $*f$ and $*df$ are zero, where $*$ is the Hodge star) admits an essentially autoadjoint extension $\Delta_\abs^p$ to the space $L^2\Omega^p(X;V)$ of square-integrable $p$-forms. We thus may form the associated heat kernel $e^{-t\Delta_\abs^p[X]}$ which is the convolution by a smooth kernel $e^{-t\Delta_\abs^p[X]}(.,.)$, is trace-class and has an asymptotic expansion $\otr(e^{-t\Delta_\abs^p[X]})=a_3t^{-\frac 3 2}+...+a_0+O(t^{\frac 1 2})$ as $t\to 0$ (cf. \cite[Theorem 1.11.4]{Gilkey}). On the other hand the spectrum of $\Delta_\abs^p[X]$ is discrete and thus we have an estimate 
\[
\otr(e^{-t\Delta_\abs^p[X]})-\dim\ker(\Delta_\abs^p[X]) \ll e^{-\lambda_1t}
\]
where $\lambda_1$ is its smallest positive eigenvalue. Thus the zeta function
\[
\zeta_{p,\abs}(s):=\frac 1{\Gamma(s)}\int_0^{+\infty} (\otr(e^{-t\Delta_\abs^p[X]})-\dim\ker(\Delta_\abs^p[X])) t^s\frac{dt}t
\]
is well-defined for $\Real(s)>3/2$ and may be extended to a meromorphic function on $\CC$ which is holomorphic at 0. One then defines $\det(\Delta_\abs^p[X])=\exp(\zeta_{p,\abs}'(0))$ and 
\[
T_\abs(X;V)=\left(\prod_{p=1}^{\dim X} \det(\Delta_\abs^p[X])^{(-1)^p p}\right)^{\frac 1 2}.
\]


\subsubsection{Reidemeister torsion}

We will use a definition of Reidemeister torsion derived from that given in \cite{Bruening_Ma1}. In this reference the authors define two norms on the determinant line 
\[
D=\otimes_{p=0}^{\dim X}\wedge^{b_p}H^p(X;V)
\] 
(where $b_p=\dim H^p(X;V)$): one, that we will denote by $\|\cdot\|_{L^2}$, which is induced by the $L^2$-norms obtained by identifying $H^p(X;V)=\ker\Delta_\abs^p[X]$ and another, $\|\cdot\|_{\rm comb}$, obtained from a smooth triangulation of $X$. The Reidemeister torsion is then the positive real number defined by 
\[
\tau_\abs(X;V) = \frac{\|\cdot\|_{\rm comb}}{\|\cdot\|_{L^2}}
\]
which does not depend on the triangulation. In \cite[(0.11)]{Bruening_Ma1} the authors compute the difference $\log\tau_\abs-\log T_\abs$ in terms of the geometry of the boundary; we will only use the special case of their result (cf. (0.14 in loc. cit.) which states that 
\begin{equation} \label{CMbd}
\log\tau_\abs(X;V) - \log T_\abs(X;V) = \frac{\log(2)}2 \dim(V) \chi(\pl X)
\end{equation}
when the metric is a product near the boundary. 

Now suppose that $\pi_1(X)$ preserves a lattice $V_\ZZ$ in $V$; we can then define the integral homology $H_*(X;V_\ZZ)$ and we have a decomposition $H_p(X;V_\ZZ)=H_p(X;V_\ZZ)_\free\oplus H_p(X;V_\ZZ)_\tors$. The free part $H_p(X;V_\ZZ)_\free$ is a lattice in $\ker(\Delta_\abs^p[X])$ so we can define its covolume $\vol\left(H_p(X;V_\ZZ)_\free\right)$. We then have, more or less tautologically (see \cite[Section 1]{Cheeger}):
\begin{equation} \label{deftau}
\tau_\abs(X;V) = \prod_{p=0}^{\dim X}\left(\frac{|H_p(X;V_\ZZ)_\tors|}{\vol\left(H_p(X;V_\ZZ)_\free\right)}\right)^{(-1)^p}
\end{equation}
by evaluating the norms on a basis of $D$ coming from bases of the free $\ZZ$-module $H^p(X;V_\ZZ)_\free$. 


\subsection{Comparing analytic torsions} \label{scheme}

We give here the proof of Theorem \ref{Main2} assuming the content of Sections \ref{CMA_small} and \ref{CMA_large} (note that the condition \eqref{square2} is not used until Section \ref{CMA_large}). We have 
\begin{align*}
&\log T_R(M_n;V) - \log T_\abs(M_n^{Y^n};V) \\
&\hspace{3cm} = \sum_{p=1}^3 p(-1)^p \frac d{ds}\left( \frac 1{\Gamma(s)} \int_0^{t_0}(\otr_R e^{-t\Delta^p[M_n]} - \otr e^{-t\Delta_\abs^p[M_n^{Y^n}]}) t^s\frac{dt}t\right)_{s=0} \\
                &\hspace{3cm} \phantom{=\sum_{p=1}^3 p(-1)^p} + \int_{t_0}^{+\infty} \otr_R e^{-t\Delta^p[M_n]} \frac{dt}t - \int_{t_0}^{+\infty} \otr e^{-t\Delta_\abs^p[M_n^{Y^n}]} \frac{dt}t \\
                &\hspace{3cm} \phantom{=\sum_{p=1}^3 p(-1)^p} + \frac d{ds}\left( \frac 1{\Gamma(s)} \int_0^{t_0} b_p(M_n;V) t^s\frac{dt}t\right)_{s=0}. 
\end{align*}
The first line is an $o(\vol M_n)$ for any $t_0>0$ according to Proposition \ref{stc}, the limit superior of the second one goes to 0 as $t_0\to+\infty$ according to Proposition \ref{terme3} and \eqref{gtr}. The third lines equals $h_n$ times a constant and thus it is also negligible before $\vol(M_n)$. Thus the right-hand side is an $o(\vol M_n)$. 


\subsection{Applying Br\"uning and Ma's result} \label{apply}

We now give the proof of Theorem \ref{CMA}. For a finite-volume hyperbolic manifold $M$ let $g_0$ be the hyperbolic metric on $M^Y$ and $g_1$ a Riemannian metric on $M^Y$ which equals $g_0$ on $M^{Y/3}$ and is a product on a neighbourhood of the boundary, for example we can take (in coordinates $(z,y)$ in a cusp):
\[
g_1(z,y) = \left(\psi(\log(Y/y))Y^{-2}+(1-\psi(\log(Y/y)))y^{-2}\right)(dz^2 + dy^2)
\]
where $\psi$ is a smooth function which is zero on  $[1,+\infty)$ and constant equal to 1 near zero. we put $g_u=ug_1+(1-u)g_0$ which is a smooth family of Riemannian metrics on $M^Y$. The following result is well-known (see also \cite[Section 4]{Bruening_Ma2} which gives an exact formula for the error term). 

\begin{lem}
There exists smooth functions $c_p(u)$ depending only on $\psi$ such that 
we have, for all $M$ and $Y\in[1,+\infty)^h$:
\begin{equation*}
\frac{d}{du}\left(\log T_\abs(M^Y,g_u)-\log\tau_\abs(M^Y,g_u)\right)=\vol(\partial M^Y)\sum_{p=0}^3p(-1)^p c_p(u).
\end{equation*}
\end{lem}

\begin{proof}
This follows at once from \cite[Theorem 2.22]{Muller} (see also \cite[Theorem 3.27]{Cheeger}) since the isometry class of the germ of $g_u$ on the boundary $\partial M_n^Y$ does not depend on $n$ or $Y$.
\end{proof}

On the other hand, by \cite[Theorem 0.1]{Bruening_Ma1} (see \eqref{CMbd}) we get that 
$T_\abs(M^Y,g_1)=\tau_\abs(M^Y,g_1)$ so that
\begin{equation}
\frac{\log T_\abs(M^Y)-\log\tau_\abs(M^Y)}{\vol(M)}
             = \frac{\vol(\partial M^Y)}{\vol(M)} \int_0^1\sum_{p=1}^3(-1)^p p c_p(u)du. 
\label{equal}
\end{equation}
Now we apply this to the heights $Y^n$; we have: 
\begin{align*}
\vol(\partial M_n^{Y^n}) &\le \sum_{j=1}^{h_n} (Y_j^n)^{-2} \alpha_1(\Lambda_{n,j})\alpha_2(\Lambda_{n,j}) \\
            &= \sum_{j=1}^{h_n} \frac{\alpha_1(\Lambda_{n,j})}{(Y_j^n)^2} \cdot \frac{\alpha_2(\Lambda_{n,j})}{\alpha_1(\Lambda_{n,j})} \le \max_{j=1, \ldots, h_n}\left(\frac{\alpha_1(\Lambda_{n,j})}{(Y_j^n)^2} \right) \cdot \sum_{j=1}^{h_n} \frac{\alpha_2(\Lambda_{n,j})}{\alpha_1(\Lambda_{n,j})}
\end{align*}
and the right-hand side is an $o(\vol M_n)$ since on the one hand we have $\sum_{j=1}^{h_n} \frac{\alpha_2(\Lambda_{n,j})}{\alpha_1(\Lambda_{n,j})} = o(\vol M_n)$ by the hypothesis that \eqref{square!} holds and on the other for the sequence $Y^n$ from Proposition \ref{mainmajo} we have $\max_{j=1, \ldots, h_n} (\alpha_1(\Lambda_{n,j}) / (Y_j^n)^2) \to 0$ as $n \to +\infty$. Thus it follows from \eqref{equal} that 
\begin{equation*}
\frac{\log T_\abs(M_n^{Y^n};V)-\log\tau_\abs(M_n^{Y^n};V)}{\vol(M_n)}
      \xrightarrow[n\to\infty]{}0
\end{equation*}
which finishes the proof of Theorem \ref{CMA}


\section{The asymptotic Cheeger--M\"uller equality: the small-time part}

\label{CMA_small}
The aim of this section is to prove the following result, which is an immediate consequence of Propositions \ref{terme2} and \ref{comp} below. 

\begin{prop}
For the sequence $Y^n$ defined in Theorem \ref{Main2} we have for any $t_0>0$ the limit
$$
\lim_{n\to\infty} \frac{\frac d{ds}\left(\frac 1{\Gamma(s)} \int_0^{t_0}(\otr_R e^{-t\Delta^p[M_n]} - \otr e^{-t\Delta_\abs^p[M_n^{Y^n}]}) t^s\frac{dt}t\right)_{s=0}}{\vol(M_n)} = 0.
$$
\label{stc}
\end{prop}

\subsection{Heat kernels on truncated hyperbolic manifolds}

Let $M=\Gamma\bs\HH^3$ be a complete hyperbolic manifold with cusps $y_1,\ldots,y_h$ a set of $\Gamma$-invariant height functions. Then for $Y\in[1,+\infty[^h$ the set
\[
\widetilde{M^Y}=\{ x\in\HH^3,\, \forall j=1,\ldots,h:\, y_j(x)\le Y_j\} 
\]
is the universal cover of $M^Y$.  The following generalisation of Proposition \ref{cly} to this context will be proved in Appendix \ref{heat} (as in Section \ref{hyp_mfd} we use $\eps$ to denote the Margulis constant of $\HH^3$). 

\begin{prop} \label{cly_bd}
Let $M$ be a hyperbolic manifold with $h$ cusps and $Y\in[1, +\infty[^h$ such that for all peripheral subgroups $\Lambda$ of $\pi_1(M)$ there is a vector in $\Lambda$ which has a displacement less than $\eps/10$ on the relevant horosphere at height $Y$. Then for all $t_0>0$ there is a $C>0$ such that for all $t\in]0,t_0]$ we have that
\[
\left|e^{-t\Delta_\abs^p[\wdt{M^Y}]}(x,y) \right| \le Ct^{-\frac 3 2} e^{-\frac{d(x,y)^2}{5t}}. 
\]
\end{prop}

This implies that the series in the following expansion for the heat kernel converges uniformly on compact sets of $\wdt{M^Y}$: 
\begin{equation} \label{dab}
e^{-t\Delta_\abs^p[M^Y]}(x,y) = \sum_{\gamma\in\Gamma} \gamma^* e^{-t\Delta_\abs^p[\wdt{M^Y}]}(x,\gamma y).
\end{equation}


\subsection{A useful estimate} 

The following proposition is the starting point for the proof of Theorem \ref{Main2}; note that the fact that $\min_{j=1,\ldots,h_n}(Y_j^n / \alpha_1(\Lambda_{n,j}))\xrightarrow[n\to\infty]{} + \infty$ implies that for $n$ large enough the heights $Y^n$ satisfy the assumption of Proposition \ref{cly_bd}. 

\begin{prop} \label{mainmajo}
Let $M_n$ be as in the statement of Theorem \ref{Main2} and put:
\[
S(n,t,Y) = \int_{M_n^Y}\sum_{\gamma\in\Gamma_n,\, \gamma\not= 1} e^{-\frac{d(x,\gamma x)^2}{5t}} dx.
\]
Then for all $t_0>0$ the function $\Omega$ defined by $\Omega(s) = \int_0^{t_0} S(n,t,Y) t^{s-\frac 3 2} \frac{dt}t$ is holomorphic on $\CC$ and there is a sequence $Y^n\in[1,+\infty[^{h_n}$ such that 
\[
\min_{j=1,\ldots,h_n}(Y_j^n / \alpha_1(\Lambda_{n,j}))\xrightarrow[n\to\infty]{} +\infty
\] 
and we have:
\[
\frac d{ds}\left(\frac 1{\Gamma(s)} \Omega(s) \right)_{s=0} 
     = \int_0^{t_0}t^{-\frac 3 2} \frac{S(n,t,Y^n)}{\vol(M_n)} \frac{dt}t \xrightarrow[n\to\infty]{}0.
\]
More precisely, we can take:
\[
Y_j^n = \left( \frac{\vol(M_n)}{\sum_{j=1}^{h_n} \frac{\alpha_2(\Lambda_{n,j})^2}{\alpha_1(\Lambda_{n,j})^2}} \right)^{\frac 1{10}} \cdot \alpha_1(\Lambda_{n,j}).
\]
\end{prop}

\begin{proof}
Let $Y \in [1, +\infty[^{h_n}$. Recall that $B_{n,j} = B_{\Lambda_{n,j}}$ et $B_{n,j}^{Y_j} = B_{\Lambda_{n,j}}^{Y_j}$ were defined in \eqref{fundstrip}. Let $R_{n,j}$ be the union of pieces of horoballs $\gamma (B_{n,j} - B_{n,j}^{Y_j}) \cap B_{n,j}$ for $\gamma \not \in \Gamma_{P_j}$. Separating unipotent and loxodromic elements we obtain as in \eqref{unfolding} the equality :
\begin{align*}
S(n,t,Y) &= \sum_{j=1}^n \int_{B_{n,j}^{Y_j}} \sum_{\eta\in\Lambda_{n,j}-\{0\}} e^{-\frac{d(x,\eta x)^2}{5t}} dx \\
         &\quad - \sum_{j=1}^{h_n} \int_{R_{n,j}} \sum_{\eta\in\Lambda_{n,j}-\{0\}} e^{-\frac{d(x, \eta x)^2}{5t}} dx
         + \int_{M-M^Y} \sum_{\gamma\in(\Gamma_n)_\lox} e^{-\frac{d(x,\gamma x)^2}{5t}} dx
\end{align*}
and we put : 
\begin{gather*}
T_1 = \int_{M-M^Y} \sum_{\gamma\in(\Gamma_n)\lox} e^{-\frac{d(x,\gamma x)^2}{5t}} dx, \quad T_2 = \sum_{j=1}^{h_n} \int_{R_{n,j}} \sum_{\eta\in\Lambda_{n,j}-\{0\}} e^{-\frac{d(x,\gamma x)^2}{5t}} dx, \\
T_3 = \sum_{j=1}^n \int_{B_{n,j}^{Y_j}} \sum_{\eta\in\Lambda_{n,j}-\{0\}} e^{-\frac{d(x,\eta x)^2}{5t}} dx. 
\end{gather*}

The term $T_1$ is dealt with exactly as the similar term $T_1$ in \ref{sec:small_time_approx}. The term $T_2$ is dealt with similarly to $I_2$ in \eqref{hororeste} in the proof of Proposition \ref{geom_side} (taking into account the integral over $t$). 

We deal with the more delicate term $T_3$ cusp by cusp: put 
\[
S_j = \int_{B_{n,j}^{Y_j}} \sum_{\eta\in\Lambda_{n,j}-\{0\}} e^{-\frac{d(x,\eta x)^2}{5t}} dx, 
\]
then we have: 
\begin{align*}
S_j     &= \vol(\Lambda_{n,j}) \int_0^{Y_j} \int_{\alpha_1(\Lambda_{n,j})}^{+\infty} e^{-\frac{\ell(r/y)^2}{5t}} d\mathcal{N}_{n,j}^*(r) \frac{dy}{y^3} \\
        &= \frac {\vol(\Lambda_{n,j})}{5t} \int_0^{Y_j} \int_{\alpha_1(\Lambda_{n,j})/y}^{+\infty} \frac{d\ell}{dr} e^{-\frac{\ell(r)^2}{5t}} \mathcal{N}_{n,j}^*(ry) dr\frac{dy}{y^3}.
\end{align*}

By Lemma \ref{ucount} we have $\mathcal{N}_{n,j}^*(ry)\ll \frac{(ry)^2}{\vol(\Lambda_{n,j})} + \frac{\alpha_2(\Lambda_{n,j})}{\alpha_1(\Lambda_{n,j})}$ where the constant does not depend on $n,j$. It follows that 
\begin{equation} \label{MAJO}
\begin{split}
S_j
    &\ll \int_0^{Y_j} \int_{\alpha_1(\Lambda_{n,j})/y}^{+\infty} r^2 e^{-\frac{\ell(r)^2}{5t}} dr \frac{dy}y t^{-\frac 5 2} \\
    &\quad + \vol(\Lambda_{n,j}) \int_0^{Y_j} \int_{\alpha_1(\Lambda_{n,j})/y}^{+\infty} \frac{\alpha_2(\Lambda_{n,j})}{\alpha_1(\Lambda_{n,j})} e^{-\frac{\ell(r)^2}{5t}} dr \frac{dy}{y^3} t^{-\frac 5 2}
\end{split}
\end{equation}
We deal with the first line now. We split the integration between $r\ge 2$ and $\alpha_1(\Lambda_{n,j})/y\le r\le 2$. When $r\ge 2$ it follows from Lemma \ref{eval} that $\ell(r)\gg\log r$ and we obtain:
\begin{equation} \label{MAJO1}
\begin{split}
\int_0^{t_0}\int_0^{Y_j} \int_2^{+\infty} r^2 e^{-\frac{\ell(r)^2}{5t}} dr \frac{dy}y t^{-\frac 5 2} \frac{dt}t    
   &\ll \int_0^{t_0} e^{-\frac{(\log 2)^2}{5t}}\log Y_j \, t^{-\frac 5 2} \frac{dt}t\\
   &\ll \log Y_j.
\end{split}
\end{equation}
On the other hand, when $r\in[0,2]$ we have $\ell(r)>cr$ for some $c>0$ and thus:
\begin{equation} \label{MAJO2}
\begin{split}
\int_0^{t_0} \int_0^{Y_j} \int_{\alpha_1(\Lambda_{n,j})/y}^2 r^2 e^{-\frac{\ell(r)^2}{5t}} dr \frac{dy}y t^{-\frac 5 2} \frac{dt}t    
   &\ll \int_0^{Y_j}\int_0^{t_0} e^{-\frac{(\alpha_1(\Lambda_{n,j})/y)^2}{Ct}} t^{-\frac 5 2}\frac{dt}t \frac{dy}y \\
   &\ll \int_0^{Y_j} \int_0^{+\infty} e^{-\frac 1{Ct}}t^{-\frac 5 2}\frac{dt}t \left(\frac y{\alpha_1(\Lambda_{n,j})}\right)^5 \frac{dy} y \\
   &\ll \left(\frac{Y_j}{\alpha_1(\Lambda_{n,j})}\right)^5.
\end{split}
\end{equation}

For the second line of \eqref{MAJO} we have the majoration 
\begin{align*}
\vol(\Lambda_{n,j}) \int_0^{Y_j} \int_{\alpha_1(\Lambda_{n,j})/y}^{+\infty} \frac{\alpha_2(\Lambda_{n,j})}{\alpha_1(\Lambda_{n,j})} e^{-\frac{\ell(r)^2}{5t}} dr \frac{dy}{y^3}  
        & \ll \alpha_2(\Lambda_{n,j})^2 \int_0^{Y_j} \int_{\alpha_1(\Lambda_{n,j})/y}^{+\infty} e^{-\frac{\ell(r)^2}{5t}} dr \frac{dy}{y^3} \\
        &\ll \frac{\alpha_2(\Lambda_{n,j})^2}{Y_j^2} e^{-\frac{(\alpha_1(\Lambda_{n,j})/Y_j)^2}{Ct}}.
\end{align*}
and it follows as in \eqref{MAJO2} above that 
\begin{equation} \label{MAJO3}
\int_0^{t_0} \vol(\Lambda_{n,j}) \int_0^{Y_j} \int_{\alpha_1(\Lambda_{n,j})/y}^{+\infty} \frac{\alpha_2(\Lambda_{n,j})}{\alpha_1(\Lambda_{n,j})} e^{-\frac{\ell(r)^2}{5t}} dr \frac{dy}{y^3} t^{-\frac 5 2} \frac{dt} t\ll \frac{\alpha_2(\Lambda_{n,j})^2}{\alpha_1(\Lambda_{n,j})^2} \times \left( \frac{Y_j}{\alpha_1(\Lambda_{n,j})}\right)^3.
\end{equation}
Putting \eqref{MAJO},\eqref{MAJO1},\eqref{MAJO2} and \eqref{MAJO3} together and summing over the cusps we obtain:
\begin{equation} \label{MAJO_U}
\int_0^{t_0} \sum_{j=1}^{h_n} S_j \frac {dt}t \ll \sum_{j=1}^{h_n} \left( \log Y_j + \left(\frac{Y_j}{\alpha_1(\Lambda_{n,j})}\right)^5 + \frac{\alpha_2(\Lambda_{n,j})^2}{\alpha_1(\Lambda_{n,j})^2} \times \left( \frac{Y_j}{\alpha_1(\Lambda_{n,j})} \right)^3 \right)
\end{equation}
We now define 
\[
a_n = \left( \frac{\vol(M_n)}{\sum_{j=1}^{h_n} \frac{\alpha_2(\Lambda_{n,j})^2}{\alpha_1(\Lambda_{n,j})^2}} \right)^{\frac 1{10}}
\]
so that $a_n$ tends to infinity, and put 
\begin{equation} \label{defY}
Y_j^n = \alpha_1(\Lambda_{n,j})\times a_n
\end{equation}
so that $\min_j(\alpha_1(\Lambda_{n,j}) / Y_j^n) = a_n^{-1} \xrightarrow[n\to\infty]{}0$. We let $S_j^n$ be $S_j$ for $Y = Y_j^n$. From \eqref{MAJO_U} above we finally get that: 
\begin{align*}
\int_0^{t_0} \frac {\sum_j S_j^n}{\vol(M_n)}  \frac{dt}t 
          &\ll  \frac 1{\vol(M_n)} \sum_{j=1}^{h_n} \left( \log a_n + \log(\alpha_1(\Lambda_{n,j})) + a_n^5 + a_n^3 \cdot \frac{\alpha_2(\Lambda_{n,j})^2}{\alpha_1(\Lambda_{n,j})^2} \right) \\
          &\le \frac{\log(\vol M_n)}{\vol(M_n)} + \frac{\sum_{j=1}^{h_n}\log\alpha_1(\Lambda_{n,j})} {\vol(M_n)} + \left(\frac {h_n}{\vol(M_n)} \right)^{\frac 1 2} + \left( \frac{\sum_{j=1}^{h_n} \frac{\alpha_2(\Lambda_{n,j})^2}{\alpha_1(\Lambda_{n,j})^2}} {\vol(M_n)} \right)^{\frac 7{10}} 
\end{align*}
The second summand is $o(\vol M_n)$ because of the condition \eqref{norm_height} on the height functions, the third by Lemma \ref{nbcuspsBS} and the last one by condition \eqref{square!}. 
\end{proof}


\subsection{Comparisons} 

In this subsection we will abbreviate:
$$
\otr_Y K = \int_{M^Y} \tr K(x,x) dx
$$
for a continuous kernel $K$ on a complete hyperbolic manifold $M$. 

\begin{prop} \label{terme2}
Let $t_0>0$, $p=1,2,3$ and $Y^n$ the sequence from Proposition \ref{mainmajo}. We have
\begin{equation} \label{t2}
\frac 1{\vol(M_n)} \frac{d}{ds}\left(\frac 1{\Gamma(s)}\int_0^{t_0}(\otr_{Y^n}(e^{-t\Delta^p[M_n]})-\otr(e^{-t\Delta_\abs^p[M_n^{Y^n}]}))t^s \frac{dt}t \right)_{s=0} \xrightarrow[n\to\infty]{} 0.
\end{equation}
\end{prop}

\begin{proof}
From \eqref{dab} it follows that
\begin{align*}
\otr_{Y^n}(e^{-t\Delta^p[M_n]})-\otr(e^{-t\Delta_\abs^p[M_n^{Y^n}]})
   &= \int_{D_n^{Y_n}} \tr(e^{-t\Delta[\HH^3]}(x,x)-e^{-t\Delta_\abs^p[\widetilde{M_n^{Y^n}}]}(x,x))dx\\
   &\quad + \int_{M_n^{Y^n}} \sum_{\gamma\in\Gamma_n-\{1\}}\tr(e^{-t\Delta^p[\HH^3]}(x,\gamma x))dx\\
   &\quad + \int_{M_n^{Y^n}} \sum_{\gamma\in\Gamma_n-\{1\}}\tr(e^{-t\Delta_\abs^p[\widetilde{M_n^{Y^n}}]}(x,\gamma x)) dx\\    
   &=: E_1+E_2+E_3.
\end{align*}

In the case all $M_n$ are covers of a given orbifold $M$ and all $Y_j^n,j=1,\ldots,h_n$ are equal the manifolds $\widetilde{M_n^{Y^n}}$ are equal to $\widetilde{M^{Y^n}}$, so that the first summand $\frac{d}{ds}(\Gamma(s)^{-1}\int_0^{t_0} E_1 t^{s-1}dt)_{s=0}$ is equal to $\vol(M_n) \cdot (\log T^{(2)}(M;V) - \log T_\abs^{(2)}(\widetilde{M^{Y^n}};V))$. 

We will now use the method of L\"uck and Schick in \cite{LS} to study this. Note that in loc. cit. these authors deal only with trivial coefficients. On the other hand, once the estimates in their Theorem 2.26 are established the proof works in all cases. The proof given in loc. cit for this result likely adapts to unimodular coefficients ; but since we need the results only for manifolds of the form $\wdt{M^Y}$, in this case their result can be deduced from \eqref{Luck_Schick_here} in the Appendix. It is proven in \cite[Section 2]{LS} that 
\begin{equation} \label{compL2}
\log T^{(2)}(M;V)-\log T_\abs^{(2)}(\widetilde{M^Y};V)\xrightarrow[\text{all }Y_j\to\infty]{} 0.  
\end{equation}
In this case we finally get
\begin{equation} \label{E1}
\frac{d}{ds}\left(\frac{1}{\Gamma(s)}\int_0^{t_0} E_1 t^s \frac{dt}t\right)_{s=0}\xrightarrow[n\to\infty]{}0.
\end{equation}

We can also adapt the argument of L\"uck and Schick to our general situation, as we now explain. Their key result is Proposition 2.37, and to prove it they separate the left-hand side of \eqref{E1} into seven summands $s_1,\ldots,s_7$. Let $Y_n=\min_j Y_j^n$; for $i\not=4$ they get by straightforward arguments the bounds (caveat: their parametrisation of the cusps is by arclength in the $y$ direction, so their statements are different in form):
\begin{gather*}
|s_1| \ll \vol(M_n)e^{-(\log Y_n)^2}, \quad |s_2| \ll \vol\left(M_n-M_n^{\sqrt{Y^n}}\right), \quad |s_3| \ll \vol\left(M_n-M_n^{Y^n/2}\right) \\
s_5=0,\quad |s_6| \ll \vol\left(M_n-M_n^{Y^n}\right), \quad |s_7| \ll \vol\left(\pl M_n^{Y^n}\right)
\end{gather*}
and as $Y_n\to+\infty$ and $\min_j(Y_j^n/\alpha_1(\Lambda_{n,j}))\to+\infty$ all the terms of the right-hand sides above are $o(\vol M_n)$. The argument for $s_4$ is more involved: they subdivide it into $s_{41},s_{42},s_{43}$ and they prove that 
$$
|s_{41}|,|s_{42}| \ll \vol\left(M_n-M_n^{Y^n/2}\right), |s_{43}| \ll e^{-2\log Y_n}\vol\left(\pl M_n^1\right) 
$$
and the terms on the right in both majorations are $o(\vol M_n)$, which finishes the proof of \eqref{E1}. 

To finish the proof we observe that Propositions \ref{cly} and \ref{cly_bd} yield the bounds 
\begin{equation*}
\tr(e^{-t\Delta^p[\HH^3]}(x,\gamma x)) ,\, \tr(e^{-t\Delta_\abs^p[\widetilde{M^Y}]}(x,\gamma x))\ll t^{-3/2}e^{-d(x,\gamma x)^2/5t} 
\end{equation*}
where the constant does not depend on (large enough, see the remark before Proposition \ref{mainmajo}) $Y$, so that we have in the notation of Proposition \ref{mainmajo} the inequality $E_2,E_3 \ll S(n,t,Y^n)$ and we get 
\begin{equation*}
\frac{d}{ds}\left(\frac{1}{\Gamma(s)}\int_0^{t_0} E_i t^{s-1}dt\right)_{s=0}\xrightarrow[n\to\infty]{}0.
\end{equation*}
for $i=2,3$, which finishes the proof of the proposition.
\end{proof}

\begin{prop}
Let $Y^n$ be the sequence from Proposition \ref{mainmajo}, then we have
\begin{equation*}
\frac 1 {\vol(M_n)} \frac d{ds}\left( \frac 1{\Gamma(s)} \int_0^{t_0} (\otr_R e^{-t\Delta^p[M_n]} - \otr_{Y^n} e^{-t\Delta^p[M_n]}) t^s\frac{dt}t \right)_{s=0} \xrightarrow[n\to\infty]{} 0.
\end{equation*}
\label{comp}
\end{prop}

\begin{proof}
From the explicit expression of the $o(1)$ terms in \eqref{geomuni} and \eqref{unfolding} we get for any $\phi\in\mathcal{A}(\RR)$ and $Y\in[1,+\infty)^{h_n}$:
\begin{equation}
\begin{split}
\otr_Y \phi(\Delta^p[M_n]) - \otr_R \phi(\Delta^p[M_n]) &= \sum_{j=1}^{h_n} 2\pi\log Y_j \int_0^{+\infty}rh_\phi^p(\ell(r)) dr \\
           &\quad + \sum_{j=1}^{h_n} \vol(\Lambda_{n,j}) \int_0^{+\infty} rh_{\phi,p}(\ell(r)) \int_{\max(\alpha_1(\Lambda_{n,j}),rY_j)}^{+\infty} E_{\Lambda_{n,j}}(\rho)\frac{d\rho}{\rho^3} dr \\
           &\quad + \sum_{j=1}^{h_n} \vol(\Lambda_{n,j}) \int_0^{+\infty} \frac{E_{\Lambda_{n,j}}(rY_j)}{(rY_j)^2} h_{\phi,p}(\ell(r)) dr \\
           &\quad + \sum_{j=1}^{h_n} \int_{B_{\Lambda_j}-B_{\Lambda_j}^{Y_j}} \sum_{\gamma\in\Gamma-\Lambda_j} \tr(\gamma^* e^{-t\Delta^p[\HH^3]}(x,\gamma x)) dx
\end{split}
\end{equation}
We want to study the right-hand side with $Y=Y^n$. The last line can be dealt with as the term $T_1$ in \ref{sec:small_time_approx}, using Lemma \ref{gen_count} instead of \eqref{loxcount}---we will not repeat the proof here. We put
\begin{align*}
T_1 &= \sum_{j=1}^{h_n} 2\pi\log Y_j^n \int_0^{+\infty}rh_t^p(\ell(r)) dr \\
T_2 &= \sum_{j=1}^{h_n} \vol(\Lambda_{n,j}) \int_0^{+\infty} \frac{E_{\Lambda_{n,j}}(rY_j^n)}{(rY_j^n)^2} h_t^p(\ell(r)) dr. \\
T_3 &= \sum_{j=1}^{h_n} \vol(\Lambda_{n,j}) \int_0^{+\infty} rh_t^p(\ell(r)) \int_{\max(\alpha_1(\Lambda_{n,j}),rY_j^n)}^{+\infty} E_{\Lambda_{n,j}}(\rho)\frac{d\rho}{\rho^3} dr 
\end{align*}
We deal first with $T_1$; recall that $\Theta$ was defined in \eqref{xitheta} so that
\[
\frac d{ds}\left(\frac 1{\Gamma(s)}\int_0^{t_0} T_1 t^s \frac{dt}t \right)_{s=0} = \sum_{j=1}^{h_n} 2\pi\log Y_j^n \frac{d\Theta}{ds}(0)
\]
and by the definition \eqref{defY} of the sequence  $Y^n$ the right-hand side is an $O(\log(\vol M_n)+\sum_{j=1}^{h_n} \log \alpha_1(\Lambda_{n,j}))$ which is itself an $o(\vol M_n)$ by the assumption \eqref{norm_height}. 

\medskip

Next we deal with $T_2$. We have 
\begin{equation} \label{spT2}
\begin{split}
\frac d{ds} \left(\int_0^{t_0} T_2 t^s\frac{dt}t \right)_{s=0} 
          &= \frac d{ds} \left( \int_0^{t_0} \sum_{j=1}^{h_n} \pi \int_0^{\alpha_1(\Lambda_{n,j})/Y_j^n} h_t^p(\ell(r))dr \, t^s\frac{dt}t \right)_{s=0} \\
          &\quad + \int_0^{t_0} \sum_{j=1}^{h_n} \int_{\alpha_1(\Lambda_{n,j})/Y_j^n}^{+\infty} \vol(\Lambda_{n,j})\frac{E_{\Lambda_{n,j}}(rY_j^n)}{(rY_j^n)^2} h_t^p(\ell(r))dr \frac{dt}t.
\end{split}
\end{equation}
As in the proof of Proposition \ref{mainmajo} we get that the second line is bounded by
\begin{equation} \label{line2}
\sum_{j=1}^{h_n} \left( 1+\frac{\alpha_2(\Lambda_{n,j})^2}{\alpha_1(\Lambda_{n,j})^2}\right) \times\left(\frac{Y_j^n}{\alpha_1(\Lambda_{n,j})}\right)^3
\end{equation}
which we sam there to be an $o(\vol M_n)$. The first line needs analytic continuation; recall the asymptotic expansion 
\eqref{dah} : 
\[
h_t^p(\ell) = \sum_{k=0}^3 b_k^p(\ell) e^{-\frac{\ell^2}{5t}} t^{\frac {-k} 2} + O(t^{\frac 1 2}). 
\]
The term associated to the $O(t^{\frac 1 2})$ is easily seen to be bounded by $h_n$ : indeed, $s \mapsto \int_0^{t_0}t^{s-1/2}dt$ is regular at $s = 0$ and we get that
\begin{equation} \label{t^12}
\frac d{ds} \left( \int_0^{t_0} \sum_{j=1}^{h_n} \int_0^{\alpha_1(\Lambda_{n,j})/Y_j^n} t^{\frac 1 2} t^s\frac{dt}t \right)_{s=0} = \sum_{j=1}^{h_n} \frac{\alpha_1(\Lambda_{n, j})}{Y_j^n} \frac d{ds}\left( \int_0^{t_0}t^{s-\frac 1 2}dt \right) = o(h_n) 
\end{equation}
since $\alpha_1(\Lambda_{n, j}) / Y_j^n = o(1)$. 

For $k = 0, \ldots, 3$ we have: 
\begin{equation} \label{sep_vs_nss}
\begin{split}
\int_0^{\alpha_1(\Lambda_{n,j})/Y_j^n} b_k^p(\ell(r)) e^{-\frac{\ell(r)^2}{5t}} dr 
     &= \int_0^{\alpha_1(\Lambda_{n,j})/Y_j^n}  b_k^p(\ell(r)) e^{-\frac{\ell(r)^2}{5t}} dr \\
     &= \int_0^1  b_k^p(\ell(r)) e^{-\frac{\ell(r)^2}{5t}} dr -\int_{\alpha_1(\Lambda_{n,j})/Y_j^n}^1  b_k^p(\ell(r)) e^{-\frac{\ell(r)^2}{5t}} dr.
\end{split}
\end{equation}
We have: 
\begin{equation} \label{est_nss}
\int_{\alpha_1(\Lambda_{n,j})/Y_j^n}^1  b_k^p(\ell(r)) e^{-\frac{\ell(r)^2}{5t}} dr \ll e^{-\frac{(\alpha_1(\Lambda_{n,j})/Y_j^n)^2}{Ct}}, 
\end{equation}
and moreover \eqref{prol23} yields the expansion
\begin{equation} \label{dev_vs}
\int_0^1  b_k^p(\ell(r)) e^{-\frac{\ell(r)^2}{5t}} dr = \sum_{l=2}^m c_l t^{l/2}
\end{equation}
with coefficients $c_l$ not depending on $n,j$. It follows from \eqref{sep_vs_nss}, \eqref{est_nss} and \eqref{dev_vs} that: 
\begin{equation} \label{estimate above}
\int_0^{\alpha_1(\Lambda_{n,j})/Y_j^n} \sum_{l=0}^3 b_l^p(\ell(r)) e^{-\frac{\ell(r)^2}{5t}} dr =  \pi \sum_{l=2}^m c_l t^{\frac{l-3}2} + O(e^{-\frac{(\alpha_1(\Lambda_{n,j})/Y_j^n)^2}{Ct}})
\end{equation}
Gathering \eqref{spT2},\eqref{line2}, \eqref{t^12} and \eqref{estimate above} we get: 
\begin{equation} \label{presque fini}
\frac d{ds}\left(\frac 1{\Gamma(s)} \int_0^{t_0} T_2 t^s\frac{dt}t\right)_{s=0} \ll h_n + o(\vol M_n) + \sum_{k=0}^3 \sum_{j=1}^{h_n} \int_0^{t_0} e^{-\frac{(\alpha_1(\Lambda_{n,j})/Y_j^n)^2}{Ct}} t^{\frac {-k} 2} \frac{dt}t. 
\end{equation}
The third summand on the right is dealt with as in the proof of Proposition 
\ref{mainmajo}: for $k=1,2,3$ we have
\[
\int_0^{t_0} e^{-\frac{(\alpha_1(\Lambda_{n,j})/Y_j^n)^2}{Ct}} t^{-\frac k 2} \frac{dt}t \le \left(\frac{Y_j^n}{\alpha_1(\Lambda_{n,j})}\right)^k \int_0^{+\infty} e^{-1/Ct} t^{-\frac k2}\frac{dt}t
\]
and for $k=0$
\[
\int_0^{t_0} e^{-\frac{(\alpha_1(\Lambda_{n,j})/Y_j^n)^2}{Ct}} \frac{dt}t = \int_0^{\frac{Y_j^n}{\alpha_1(\Lambda_{n,j})}t_0} e^{-1/Ct} \frac{dt}t \ll \log(Y_j^n/\alpha_1(\Lambda_{n,j}))
\]
so that we finally obtain
\begin{equation} \label{last_term}
\frac d{ds}\left(\frac 1{\Gamma(s)} \int_0^{t_0} T_2 t^s\frac{dt}t\right)_{s=0} \ll \sum_{j=1}^{h_n} \left(\frac{Y_j^n}{\alpha_1(\Lambda_{n,j})}\right)^3  
\end{equation}
which is an $o(\vol M_n)$ as $n\to\infty$ for $Y=Y^n$. We can finally conclude from \eqref{presque fini} and \eqref{last_term} that 
\[
\frac d{ds}\left(\frac 1{\Gamma(s)} \int_0^{t_0} T_2 t^s\frac{dt}t\right)_{s=0} = o(\vol M_n). 
\]

\medskip

The summand $T_3$ is dealt with in a similar manner. For any $n\ge 1$ and $j=1,\ldots,h_n$ we put
\[
\kappa_j'=\vol(\Lambda_{n,j}) \int_{\max(\alpha_1(\Lambda_{n,j}),rY_j^n)}^{+\infty} E_{\Lambda_{n,j}}(\rho)\frac{d\rho}{\rho^3}, 
\]
We have seen in the proof of Theorem \ref{conv1} that $\kappa_j' \ll \left(\frac{\alpha_2(\Lambda_{n,j})}{\alpha_1(\Lambda_{n,j})}\right)^2$ (uniformly in $r$) and we get 
\begin{align*}
& \vol(\Lambda_{n,j}) \int_0^{+\infty} rh_t^p(\ell(r)) \int_{\max(\alpha_1(\Lambda_{n,j}),rY_j^n)}^{+\infty} E_{\Lambda_{n,j}}(\rho)\frac{d\rho}{\rho^3} dr \\ 
&\qquad = \kappa_j' \int_0^{\alpha_1(\Lambda_{n,j})/Y_j^n} rh_t^p(\ell(r))dr + \kappa_j' \int_{\alpha_1(\Lambda_{n,j})/Y_j^n}^{+\infty} rh_t^p(\ell(r)) dr \\
&\qquad = \kappa_j'\int_0^1 rh_t^p(\ell(r)) dr + O \left(t^{-\frac 3 2} e^{-\frac{(\alpha_1(\Lambda_{n,j})/Y_j^n)^2}{Ct}} \left(\frac{\alpha_2(\Lambda_{n,j})}{\alpha_1(\Lambda_{n,j})}\right)^2 \right)
\end{align*}
where to obtain the last line from the second we used the same arguments as to deal with the term \eqref{sep_vs_nss} for the first summand, and we applied arguments from the proof of Proposition \ref{mainmajo} to the second.

Thus, we have 
\[
\frac d{ds}\left(\frac 1{\Gamma(s)} \int_0^{t_0} T_3 t^s\frac{dt}t\right)_{s=0} \ll \sum_{j=1}^{h_n} \left(\frac{\alpha_2(\Lambda_{n,j})}{\alpha_1(\Lambda_{n,j})}\right)^2 + \sum_{j=1}^{h_n} \left(\frac{\alpha_2(\Lambda_{n,j})}{\alpha_1(\Lambda_{n,j})}\right)^2\int_0^{t_0} e^{-\frac{(\alpha_1(\Lambda_{n,j})/Y_j^n)^2}{Ct}} t^{-\frac 3 2} \frac{dt}t
\]
and we have already proved that the right-hand side is an $o(\vol M_n)$, which concludes the proof.  
\end{proof}


\section{The asymptotic Cheeger--M\"uller equality: the large-time part}

\label{CMA_large}
We will give here a proof of the following result. 

\begin{prop} \label{terme3} 
For the sequence $Y^n$ from proposition \ref{terme2} we have that 
\[
\sup_{n}\left( \frac 1 {\vol(M_n)} \int_{t_0}^{+\infty} \left( \otr(e^{-t\Delta_\abs^p[M_n^{Y^n}]}) - \dim\ker\Delta_\abs^p[M_n^{Y^n}]\right) \frac{dt}t\right)
\]
is finite and goes to 0 as $t_0\to\infty$.
\end{prop}

The main point in the proof of Proposition \ref{terme3} is that for the sequence $Y^n$ of Proposition \ref{mainmajo} there is a uniform spectral gap for the manifolds $M_n^{Y^n}$: the proof of the following statement will take up most of this section.  

\begin{prop} \label{ver} 
There is a $\lambda_1>0$ depending only on $V$ such that for $n$ large enough and any $p$-form $f\in\Omega_\abs^p(M_n^{Y^n};V)$ which is orthogonal to harmonic forms we have 
\[
\frac{\langle\Delta^pf,f\rangle_{L^2(M_n^{Y^n})}}{\|f\|_{L^2(M_n^{Y^n})}} \ge \lambda_1.
\] 
\end{prop}

\begin{proof}[Proof of Proposition \ref{terme3}]
We proceed as in \ref{gtps} above, but have to check that both properties of the heat kernel used there still hold for $M_n^{Y^n}$. Namely we need to check that 
\begin{enumerate}
\item \label{gapp} There is a $\lambda_1>0$ such that for any $n$ and any eigenvalue 
$\lambda>0$ of $\Delta_\abs^p[M_n^{Y^n}]$ we have $\lambda\ge\lambda_1$. 
\item \label{bddd} The sequence $\otr(e^{-\Delta_\abs^p[M_n^{Y^n}]})$ is bounded.
\end{enumerate}
The point \ref{gapp} is a direct consequence of Proposition \ref{ver} below, and we deduce \ref{bddd} from the following more precise result: for any given  $t>0$ we have in fact the limit
\begin{equation}
\lim_{n\to\infty} \frac {\otr (e^{-t\Delta_\abs^p[M_n^{Y^n}]})}{\vol(M_n)} = \otr_{\Gamma}(e^{-t\Delta^p[\HH^3]})
\label{convtr2}
\end{equation}
Indeed, we see that
\begin{equation*}
\frac {|\otr(e^{-t\Delta_\abs^p[M_n^{Y^n}]})-\otr_\Gamma(e^{-t\Delta^p[\HH^3]})|}{\vol(M_n)}
     \ll \frac{|\otr_{\Gamma_n} e^{-t\Delta_\abs^p[\widetilde{M^{Y^n}}]} - \otr_{\Gamma_n} e^{-t\Delta^p[\HH^3]}|}{\vol(M_n)} + \frac{S(n,t,Y^n)}{\vol(M_n)}.
\end{equation*}
The proof of Proposition \ref{mainmajo} yields that $S(n,t,Y^n)=o(\vol M_n)$ and that of Proposition \ref{terme2} that the first summand also goes to 0 as $n\to\infty$.
\end{proof}


\subsection{Preliminaries to the proof of Proposition \ref{ver}}

\subsubsection{Comparisons of eigenfunctions with the constant term}

We will make intensive use of the following inequality: there are constants $C,c$ such that for any finite-volume manifold $M$ with $h$ cusps and height functions at each of them and for all $Y\in[1,+\infty[^h$, if $f\in\Omega_\abs^p(M^Y;V)$ is an eigenform of eigenvalue $\lambda$ and $Y_j/\alpha_1(\Lambda_j)\ge C\sqrt{\lambda}$ for all $j$ then 
\begin{equation} \label{compc} 
|f(x)-f_{P_j}(x)| \le |f|_{L^2\Omega^p(M^Y;V)}e^{-c\frac {y_j(x)}{\alpha_1(\Lambda_{n,j})}} \text{ for all }x\in M-M^Y.
\end{equation}
This is a refined version of \cite[(6.2.1.3)]{CV}, and it follows from the proof of the latter: the only difference in our statement is in the explicit constant $c/\alpha_1(\Lambda_j)$ in the exponential which replaces the $b_0$ in {\it loc. cit.}. This constant $b_0$ comes from the estimate of Fourier expansions and equals the systole of the dual lattice of $\Lambda_j$, which is easily seen to be equal to $\alpha_1(\Lambda_j)^{-1}$. 


\subsubsection{Spectral gap of submanifolds}

Let us set up notation for the next result: we will denote by $X,g$ a complete Riemannian manifold, by $N$ an open subset of $X$ with smooth boundary and by $E$ a flat bundle on $X$. We suppose that the 1-neighbourhood $W$ of $\pl N$ in $X$ is a collar neighbourhood which we parametrise as $W=[-1,1]\times \pl N$; this is satisfied for $X=M$ a finite--volume hyperbolic manifold and $N=M^Y,\,Y\ge 3$. 

\begin{lem} \label{ns}
Suppose that the spectrum of $\Delta^p[X]$ is bounded below by some $\lambda_0>0$, and let $f\in \Omega_\abs^p(N;E)$ be a co-closed form such that 
$$
\frac{\|f\|_{L^2(W^-)}^2}{\|f\|_{L^2(N)}^2} \le \frac 1{10} \min(1, \lambda_0), \quad W^- = \pl N\times[-1,0] 
$$
Then the Rayleigh quotient 
$$
\frac{\langle \Delta_\abs^p[N] f,f\rangle_{L^2(N)}}{\|f\|_{L^2(N)}^2}  = \frac{\|df\|_{L^2(N)}^2}{\|f\|_{L^2(N)}^2}
$$ 
is bounded below by $\lambda_0/4$. 
\end{lem}

\begin{proof}
Let $h$ be a smooth function with value $0$ on $(-\infty,0]$ and $1$ on $[1,+\infty)$, and $0\le h'\le 2$. Define a smooth $p$-form $\wdt f$ on $X$ by 
\begin{equation*}
\wdt f(x) =
   \begin{cases}
    f(x) & x\in N-W^-; \\
    h(d(x,\pl N))f(x) & x\in W^-;\\
    0 & x\in X-N. 
   \end{cases}
\end{equation*}
Then, putting $y(x)=d(x,\pl N)$, we get that $d\wdt f=h'(y)dy\wedge f + h(y)df$, whence it follows that
$$
\| d\wdt f\|_{L^2(X)}^2 \le 2\|df\|_{L^2(N)}^2 + 4\|f\|_{L^2(W^-)}^2. 
$$
On the other hand, we have that 
$$
\|\wdt f\|_{L^2(X)}^2 \ge \|f\|_{L^2(N-W^-)}^2 = \|f\|_{L^2(N)}^2 - \|f\|_{L^2(W^-)}^2 \ge \frac 9{10} \|f\|_{L^2(N)}
$$
and it follows that
\begin{align*}
\frac{\|df\|_{L^2(N)}^2}{\|f\|_{L^2(N)}^2} &\ge \frac{\|d\wdt f\|_{L^2(X)}^2}{2\|f\|_{L^2(N)}^2} - 2 \frac{\|f\|_{L^2(W^-)}^2}{\|f\|_{L^2(N)}^2} \\
               &\ge \frac 9{20} \frac{\|d\wdt f\|_{L^2(X)}^2}{\|\wdt f\|_{L^2(X)}^2} - \frac{2\lambda_0}{10} \ge \frac 5{20} \lambda_0 = \lambda_0/4. 
\end{align*}
\end{proof}


\subsubsection{Eigenvalues don't jump}

We will need the following weak continuity result for the spectrum which follows from \cite[Lemma 9.9 and Proposition 9.10(2)]{Berline_Getzler_Vergne} but can also easily be proven using the min-max principle or more powerful continuity properties of spectra for families of operators. 

\begin{lem} \label{hole}
Suppose that $N$ is a compact smooth manifold and $g_u,u\in[0,1]$ a smooth family of Riemannian metrics on $N$. Let $\Delta^p[g_u]$ be the Hodge Laplace operator on forms with values in a flat Hermitian bundle over $N$ (and absolute boundary conditions) and suppose that there are $0<a<b$ such that:
\begin{itemize} 
\item for all $u\in[0,1]$ there is no eigenvalue of $\Delta^p[g_u]$ in $[a,b]$; 
\item for some $u_0\in[0,1]$ there is no eigenvalue of $\Delta^p[g_{u_0}]$ in $]0,b]$. 
\end{itemize}
Then there is no eigenvalue of $\Delta^p[g_u]$ in $]0,b]$ for any $u\in[0,1]$. 
\end{lem}


\subsection{Proof of Proposition \ref{ver}}

\subsubsection{Outline}

Recall that $\lambda_0>0$ denotes a lower bound for the spectrum of $\Delta^p[M_n]$; in the sequel we will suppose that $n_1>n_2$ (the symmetric case can be dealt with with similar arguments). We will prove that there is a $0<\lambda_1\le\lambda_0/4$ such that the two following claims hold:
\begin{enumerate}
\item\label{large_eigen} There are $\lambda_1>\eps>0$ such that for $n$ large enough there is no eigenvalue of $\Delta^p[M_n^{\Ups}]$ in $]\eps,\lambda_1[$ for all $\Ups\in[1,+\infty[^{h_n}$ such that $\forall j,\, \Ups_j\ge Y_j^n$; 
\item\label{small_eigen} For any $n$ and $\Ups$ large enough there is no eigenvalue of $\Delta_\abs^p[M_n^{\Ups}]$ in $]0,\lambda_1[$. 
\end{enumerate}
The proposition then follows by application of lemma \ref{hole}. 

Here is a quick outline of the proof of both points before embarking on the formal demonstration: the idea in both cases is that if we have an absolute eigenform which violates the claim, then either (for \ref{large_eigen}) it will also violate Lemma \ref{ns} or (for \ref{small_eigen}) we can modify it by an harmonic form to construct a function which violates the same lemma. In both cases we compute norms of constant terms in the cusps and use \eqref{compc} to compare them to the norm of our eigenfunctions. Let us remark once more that this proof is very much inspired from \cite[Section 6.9]{CV}. 


\subsubsection{Proof of \ref{large_eigen}}

We will work in what follows with an hyperbolic manifold $M$ with $h$ cusps and $Y\in [1,+\infty[^h$, and apply our computations to $M_n$ and $Y^n$ only at the end.  

Suppose that $f$ is a $p$-eigenform with coefficients in $V$ and eigenvalue in $]0,\lambda_0/4[$. For notational ease we will suppose that there are $s \in \RR$, integers $l, k$ and $\omega,\omega'\in\Omega^+(W_{l,k})$, $\ovl\omega,\ovl\omega'\in\Omega^-(W_{-l,-k})$ such that 
\begin{equation}
f_{P_j} = y_j^{1 + \frac s 2}\omega_j + y_j^{1 - \frac s 2}\omega_j' + y_j^{1 - \frac s 2}\ovl\omega_j + y_j^{1 + \frac s 2}\ovl\omega_j' 
\label{thc}
\end{equation} 
(here is an outline as to how to adapt the arguments below to the case where the constant terms of $f$ are not purely of the form \eqref{thc}: then they are a linear combination of such, and since the components are pointwise orthogonal the computations of $L^2$-norms below carry over to this case; the reader will see that this is sufficient for the whole proof to work with a very few cosmetic alterations). Moreover we may (by symmetry) take $s>0$, and since the Laplace eigenvalues are bounded away from 0 on the imaginary line we may actually suppose that $s$ is bounded away from 0 (by a constant depending only on $V$). Now the idea is that because of absolute boundary conditions, the dominant term in \eqref{thc} far in the cusp will be $y_j^{1-s/2}\omega_j' + y_j^{1-s/2}\ovl\omega_j$ unless the eigenvalue is very small, and this is concentrated away from the boundary of $M^Y$, contradicting Lemma \ref{ns}. 

\medskip

Absolute boundary conditions have to be satisfied by all $f_{P_j}$ as well as by $f$; we can make them explicit by taking the differential of \eqref{thc} using \eqref{calcdiffind1} and we get that 
$$
(s+l-k)Y_j^{s/2}\omega_j - (s-l+k)Y_j^{-s/2}\omega_j' = 0 = -(s+l-k)Y_j^{-s/2}\ovl\omega_j + (s-l+k)Y_j^{s/2}\ovl\omega_j' 
$$
since both the $(1,0)$-part (on the left above) and the $(0,1)$-part (on the right) of the contraction of $df_P$ with the normal vector $\pl/\pl y_j$ have to be zero. We can rewrite this as 
\begin{equation} \label{tc_bd_cond}
\omega_j'  = \frac{s+l-k}{s-l+k}Y_j^s\omega_j ,\quad \ovl\omega_j = \frac{s-l+k}{s+l-k}Y_j^s\ovl\omega_j'. 
\end{equation}

Now let $a>0$ and $Y'\le Y/2$ be such that $\alpha_1(\Lambda_j)/Y_j'\le a$ for all $j$; we have
\begin{equation} \label{min_tc_prelim}
\|f\|_{L^2(M^Y)}^2 \ge \|f\|_{L^2(M^Y-M^{Y'})}^2  \ge \frac 1 2 \|f_P\|_{L^2(M^Y-M^{Y'})}^2 - \|f-f_P\|_{L^2(M^Y-M^{Y'})}^2 ;  
\end{equation}
by \eqref{compc} we have:
\begin{align*}
\| f-f_P\|_{L^2(M^Y-M^{Y/3})} &\le \int_{M^Y-M^{Y'}} e^{-c\max_j(y_j(x))/\alpha_1(\Lambda_j)} dx\cdot \|f\|_{L^2(M^Y)} \\
    &\le e^{-c\max_j(Y_j'/\alpha_1(\Lambda_j))}\vol(M^Y-M^{Y'})\cdot \|f\|_{L^2(M^Y)} 
\end{align*}
and it follows from \eqref{min_tc_prelim} that:
\begin{equation} \label{mino_tot}
\|f\|_{L^2(M^Y)}^2  \ge (1/2 - e^{-c\min_j(Y_j'/\alpha_1(\Lambda_j))}\vol(M))^{-1} \cdot \|f_P\|_{L^2(M^Y-M^{Y'})}^2
\end{equation}

We now give a lower bound for the norm of the constant term on $M^Y-M^{Y'}$. This is computed as follows, using \eqref{tc_bd_cond} to rewrite \eqref{thc}: 
\begin{equation} \label{comp_constnorm}
\begin{split}
&\|f_P\|_{L^2(M^Y- M^{Y'})}^2 \\
           &\quad=  \sum_{j=1}^h \int_{M^Y-M^{Y'}} \left(y_j^{1 + \frac s 2} - \frac{s+l-k}{s-l+k} Y_j^{s}y_j^{1 - \frac s 2}\right)^2|\omega_j|^2 + \int_{M^Y-M^{Y'}} \left( \frac{s-l+k}{s+l-k} y_j^{1 + \frac s 2} +  Y_j^{s}y_j^{1 - \frac s 2}\right)^2|\ovl\omega_j'|^2 \\
           &\quad= \sum_{j=1}^h \int_{Y_j'}^{Y_j} \left(y^{1 + \frac s 2} - \frac{s+l-k}{s-l+k} Y_j^{s}y^{1 - \frac s 2}\right)^2 \frac{dy}{y^3} |\omega_j|^2 \vol(\Lambda_j) \\
           &\quad\quad + \int_{Y_j'}^{Y_j} \left( \frac{s-l+k}{s+l-k} y^{1 + \frac s 2} +  Y_j^{s}y^{1 - \frac s 2}\right)^2\frac{dy}{y^3} |\ovl\omega_j'|^2 \vol(\Lambda_j) \\
           &\quad = \frac 1 s \sum_{j=1}^h \left( Y_j^s - (Y_j')^s - \left( \frac{s+l-k}{s-l+k}\right)^2 Y_j^{2s} \left( Y_j^{-s} - (Y_j')^{-s} \right) + 2s\frac{s+l-k}{s-l+k} Y_j^{s}\log\frac{Y_j}{Y_j'} \right) |\omega_j|^2 \vol(\Lambda_j) \\
           &\quad\quad + \left( -\left( \frac{s - l + k}{s + l - k}\right)^2 Y_j^{2s} \left( Y_j^{-s} - (Y_j')^{-s} \right) + Y_j^{s} - (Y_j')^{s} +  2s\frac{s - l + k}{s + l - k} Y_j^{s}\log\frac{Y_j}{Y_j'}\right) |\ovl\omega_j'|^2 \vol(\Lambda_j)
\end{split}
\end{equation}
and we finally deduce that when $(Y_J/Y_j') \gg 1$ we have:
\begin{equation}
\|f_P\|_{L^2(M^Y - M^{Y'})}^2 \gg \sum_{j=1}^h Y_j^s \left(\frac{Y_j}{Y_j'}\right)^s \left( |\omega_j|^2 + |\ovl\omega_j'|^2\right) \vol(\Lambda_j)
\label{mino_tc}
\end{equation}
where the constant depends only on $V$ and $s - l + k, s + l - k$ (it is bounded away from 0 when the latter two are).

It remains to give an upper bound for the norm near the boundary; we have:
\begin{align*}
\| f\|_{L^2(M^Y-M^{Y/3})}^2 &\le 2\| f_P\|_{L^2(M^Y-M^{Y/3})}^2 + 2\| f-f_P\|_{L^2(M^Y-M^{Y/3})}^2 \\
      &\le 2\| f_P\|_{L^2(M^Y-M^{Y/3})} +  2e^{-c\min_j(Y_j'/\alpha_1(\Lambda_j))}\vol(M)\cdot \|f\|_{L^2(M^Y)}
\end{align*}
As in \eqref{comp_constnorm} we can compute the norm of the constant term:
\begin{equation}
\begin{split}
\|f_P\|_{L^2(M^Y- M^{Y'})}^2 &= \frac 1 s \sum_{j=1}^h \left( aY_j^s + bY_j^s + 2\frac{s+l-k}{s-l+k} Y_j^s\log(3) \right) |\omega_j|^2 \vol(\Lambda_j) \\
           &\quad + \left( a'Y_j^s + b Y_j^s +  2\frac{s - l + k}{s + l - k} Y_j^s\log(3) \right) |\ovl\omega_j'|^2 \vol(\Lambda_j)
\end{split}
\end{equation}
where $a = 1 - (2/3)^s$, $b = \left( \frac{s+l-k}{s-l+k}\right)^2((3/2)^s - 1)$, $a' = \left(\frac{s - l + k}{s + l - k}\right)^2((3/2)^s - 1)$. In the end we get the estimate: 
\begin{equation}
\|f\|_{L^2(M^Y-M^{Y/3})} \ll \sum_{j=1}^h Y_j^s \left (|\omega_j|^2 + |\ovl\omega_j'|^2\right) \vol(\Lambda_j)  +  e^{-c\min_j(Y_j'/\alpha_1(\Lambda_j))}\vol(M)\cdot \|f\|_{L^2(M^Y)}.
\label{majobord}
\end{equation}
where the constant stays bounded when $s - l + k, s + l - k$ are both bounded away from 0. 

Now we separate two cases: we will suppose first that $(l, k)$ in the constant term \eqref{thc} is not equal to $(-n_1, n_2)$. In this case, for $s = k - l$ the eigenvalue computed in \eqref{Caseig} is bounded away from $0$; thus there exists a $0 < \lambda_1 \le \lambda_0/2$ depending only on $V$ such that if $f$ is an eigenform of the Laplacian with absolute boundary conditions on $M^Y$ and the eigenvalue of $f$ is less than $\lambda_1$ then $|s + l - k| \ge \delta_0$ where $\delta_0 > 0$ depends only on $V$. On the other hand we have $k - l = n_1 + n_2 > 0$ and thus $s - l + k$ always stays bounded away from 0. Let $W^-$ be the 1-neighbourhood of $\pl M^Y$ in $M^Y$. We have $W^-\subset M^Y-M^{Y/3}$ and for a form $f$ as above \eqref{mino_tc} together with \eqref{mino_tot} and \eqref{majobord} yield that: 
\begin{equation} \label{majo_quotient}
\frac{\| f \|_{L^2(W^-)}^2}{\| f \|_{L^2(M^Y)}^2} \ll \left(\frac 1 2 - e^{-c\min_j(Y_j'/\alpha_1(\Lambda_j))}\vol(M) \right)^{-1} \min_{j=1,\ldots,h}(Y_j'/Y_j)^s  + e^{-c\min_j(Y_j'/\alpha_1(\Lambda_j))}\vol(M)
\end{equation}
with a constant depending only on $V$. 

Now we go back to our sequence $M_n,Y^n$. First we observe that from the condition in the statement of Theorem \ref{Main2} that 
$$
\sum_{j=1}^{h_n} \left(\frac{\alpha_1(\Lambda_{j,n})}{\alpha_2(\Lambda_{j,n})}\right)^2 \ll \frac{\vol(M_n)}{(\log(\vol M_n))^{20}} 
$$
and the defintion of $Y^N$ in Proposition \ref{mainmajo} imply the lower bound $Y_j^n\gg(\log(\vol M_n))^2\alpha_1(\Lambda_{n,j})$. Thus we can choose the $Y_j'$ so that for all $j$ we have $Y_j^n/Y_j' = (\log(\vol M_n))^{1/2}$ and $Y_j'/\alpha_1(\Lambda_{n,j})\gg \log(\vol(M_n))^{3/2}$, hence 
$$
\limsup_{n\to+\infty} (e^{-c\min_j(Y_j'/\alpha_1(\Lambda_{j,n}))}\vol(M_n)) = 0, \quad \limsup(\min_j(Y_j'/Y_j^n)^s) = 0.  
$$
It follows that for $n$ large enough we can use \eqref{majo_quotient} and we obtain that for any $\eps>0$, for $n$ large enough and for any eigenform $f\in L_\abs^2\Omega^1(M_n^{Y^n};V)$ as above we have
$$
\frac{\|f\|_{L^2(W^-)}^2}{\|f\|_{L^2(M^Y)}} \le \eps 
$$
which contradicts Lemma \ref{ns} for $\eps$ small enough (depending on $\lambda_0$).  

For $(l, k) = (-n_1, n_2)$ we get in the same way that there are no eigenvalues of $\Delta_\abs^p[M_n^{Y^n}]$ in the interval $[\eps,\lambda_1]$ where $\eps>0$ can be chose arbitrarily small for $n$ large enough, and this is sufficient to finish the proof of claim \ref{large_eigen}. 


\subsubsection{Proof of \ref{small_eigen}}

If $f$ is a 1-eigenform, its constant term given by \eqref{thc}, the only case where the eigenvalue can be close to zero is when $(l,k)=(-n_1,n_2)$. To prove claim \ref{small_eigen} it thus suffices to show that when an eigenfunction on $M^{\Ups}$ has its constant term equal to \eqref{thc} with $(l,k)=(-n_1,n_2)$ the eigenvalue cannot be too small and nonzero when $\Upsilon$ is large enough. Let $s_1=n_1-n_2$ (the value of $s$ for which a constant term \eqref{thc} is harmonic), $\Ups\ge Y,\delta>0$ and suppose that there is an eigenform $f_0\in \Omega_\abs^1(M^\Ups;V)$ with eigenvalue having parameter $s=s_1 - 2\delta$ (by \eqref{Caseig} the eigenvalue is $\asymp \delta$) and constant term \eqref{thc} with $\omega\in\Omega^+(V_{-n_1,n_2})$. We want to prove that for $\delta$ small enough and $\Ups$ large enough such a form cannot exist; the scheme of proof above is not adaptable to this setting since the holomorphic part of term which dominates the norm has a coefficient that goes to 0 as the eigenvalue does; however we can modify $f$ by an harmonic form to make its holomorphic part small near the boundary. 

The proof will nevertheless be very similar to the one above, and there is one notable simplification: since we can take $\Ups$ as large as we want for any $n$, there is no need to consider the terms coming from the comparison of forms with their constant terms (we shall thus ignore them in all computations below). 

\begin{lem}
For $\Upsilon$ large enough and any $\omega\in\Omega^+(V_{-n_1,n_2})$ there is an $\ovl\omega \in \Omega^-(V_{n_1,-n_2})$ and a 1-form in $\ker \left( \Delta_\abs^1[M^\Ups] \right)$ whose constant term in the $j$th cusp equals $y_j^{1 + \frac{s_1} 2}\omega_j + y_j^{1 - \frac{s_1}2}\ovl\omega_j$. 
\end{lem}

\begin{proof}
First, the $(1,0)$-part of the constant term of a nonzero form $f_1$ in $\ker \left( \Delta_\abs^1[M^\Ups] \right)$ must be equal to $y_j^{1 + \frac{s_1} 2}\omega_j$ for some $\omega\in\Omega^+(V_{-n_1,n_2})$, as satisfying boundary conditions excludes that it contains a nonzero term in $y_j^{1 - \frac{s_1}2}\omega_j'$. Second, it cannot be zero because of Lemma \ref{ns}: indeed, if it were then we would have by computations similar to those above 
$$
\frac{\|f\|_{L^2(M^\Ups-M^{\Ups/3})}^2}{\|f\|_{L^2(M^\Ups-M^{\Ups'})}^2} \ll \left(\frac{\Ups'}{\Ups}\right)^{s_1}
$$
which can be made as small as we want by taking $\Ups/\Ups'$ large enough. So we get an injective map from $\ker\left(\Delta_\abs^1[M^\Ups]\right)$ to $\Omega^+(V_{-n_1,n_2})$ by $f\mapsto (\omega_j)$. Now these two spaces have the same dimension, for well-known topological reasons (see e.g. \cite[Section 4.2]{MFP}) and we can conclude that this map must be surjective. 
\end{proof}

Write $(f_0)_{P_j}$ as in \eqref{thc}. By the lemma above we can pick a $f_1\in\ker\Delta_\abs^1[M^\Ups]$ such that 
$$
(f_1)_{P_j} = y_j^{1 + \frac{s_1} 2} \Ups_j^{-\delta} \omega_j + y_j^{1 - \frac{s_1} 2} \ovl\omega_{j,1} 
$$ 
for all $j$. We put $f = f_0 - f_1$. We will check that for $\Upsilon$ large enough and $\delta$ small enough $f$ satifies the conditions of Lemma \ref{ns}, which yields a contradiction as the spectrum of $\Delta^1[M]$ on square--integrable forms is bounded below by $\lambda_0$. Since $f_0,f_1$ are orthogonal we have 
\begin{equation}
\begin{split}
\| f\|_{L^2(M^\Ups)}^2 &= \|f_0\|_{L^2(M^\Ups)}^2 + \|f_1\|_{L^2(M^\Ups)}^2 \\
                &\gg  \sum_{j=1}^h \Ups_j^{s_1 - 2\delta}|\omega_j|^2\vol(\Lambda_j) + \sum_{j=1}^h Y_j^{-s_1+2\delta}(|\ovl\omega_j|^2 + Y_j^{-2\delta}|\ovl\omega_{j,1}|^2)\vol(\Lambda_j) 
\end{split}
\label{minonorm}
\end{equation}
where the lower bound follows from the same computation as in \eqref{mino_tc}. On the other hand:
\begin{align*}
\|f\|_{L^2(M^\Ups-M^{\Ups/3})}^2 &= \sum_{j=1}^h \int_{\Ups_j/3}^{\Ups_j} \left(y_j^{1 + \frac{s_1} 2}(y_j^{-\delta} - \Ups_j^{-\delta}) - \frac{\delta}{2s_1-\delta}\Ups_j^sy_j^{1 - \frac s 2}\right)^2\frac{dy}y |\omega_j|^2 \vol(\Lambda_j) \\
                        &\phantom{=\sum_{j=1}^h} + \int_{\Ups_j/3}^{\Ups_j} \left| y_j^{1 - \frac s 2}\ovl\omega_j - y_j^{1 - \frac{s_1}2}\ovl\omega_{j,1} - \frac{\delta}{2s_1-\delta}\Ups_j^{-s}y_j^{1 + \frac s 2}\ovl\omega_j\right|^2 \frac{dy}y\vol(\Lambda_j) \\
                        &\ll \sum_{j=1}^h \delta^2 \Ups_j^{s_1 - 2\delta}|\omega_j|^2\vol(\Lambda_j) + \sum_{j=1}^h \Ups_j^{-s_1 + 2\delta}(|\ovl\omega_j|^2 + \Ups_j^{-2\delta}|\ovl\omega_{j,1}|^2)\vol(\Lambda_j) 
\end{align*}
According to \eqref{minonorm} the right-hand side above is $\ll\delta^2\|f\|_{L^2(M^Y)}^2$ as soon as $\max_j(Y_j/\Ups_j)^{s_1}\le\delta$, so for $\delta$ small enough (depending on $\lambda_0$) and $\Ups$ large enough the 1-form $f$ satisfies the assumptions of Lemma \ref{ns}.


\section{Betti numbers}

\label{betti_growth}
Here we prove Proposition \ref{betintro}. Note that we could not deduce it 
immediately from the convergence of the regularised trace because of the 
spectral terms coming from the Eisenstein series. 

\begin{prop}
Let $M_n$ be sequence of finite-volume hyperbolic three-manifolds and suppose 
that $M_n$ BS-converges to $\HH^3$. 
Then we have for $p=1,2$
\begin{equation}
\frac{b_p(M_n)}{\vol(M_n)}\xrightarrow[n\to\infty]{} 0 
\label{betti}
\end{equation}
\label{betprop}
\end{prop}

\subsection{First proof}
For trhis proof we need to assume that \eqref{square!} holds. Let $Y^n\in [1,+\infty[^{h_n}$ be the sequence from Proposition \ref{mainmajo}; for all $n$ and $t>0$ we have 
\[
\dim\ker(\Delta_\abs^p[M_n^{Y^n}]) \le \otr e^{-t\Delta_\abs^p[M_n^{Y^n}]}. 
\]
On the other hand $b_p(M_n) = \dim\ker(\Delta_\abs^p[M_n^{Y^n}])$ and it follows that for any $t>0$ we have:
\[
\limsup_{n\to\infty} \frac{b_p(M_n)}{\vol(M_n)} \le \lim_{n\to\infty} \frac {\otr e^{-t\Delta_\abs^p[M_n^{Y^n}]}}{\vol(M_n)} = \otr_\Gamma e^{-t\Delta^p[\HH^3]}.
\]
The right-hand side goes to 0 as $t\to\infty$ since $b_p^{(2)}(\HH^3)=0$ (cf. \cite[Theorem 1.63]{LuckB}), and \eqref{betti} follows. 


\subsection{Second proof}
\label{dehnBS}

Here we give a complete proof proof of \eqref{betti}. The idea is that 
we can approximate the noncompact manifolds $M_n$ by Dehn surgeries 
so that the sequence of compact manifolds obtained be BS-convergent as well, 
and then the results of \cite{7S} do the work for us. 

\begin{lem} \label{dehn}
Suppose that $M_n$ is a sequence of finite-volume hyperbolic three--manifolds 
which BS-converges to $\HH^3$. Then there is a sequence $M_n'$ of 
compact hyperbolic manifolds such that 
\begin{enumerate}
\item \label{is_surgery} For all $n$, $M_n'$ is obtained by Dehn surgery on $M_n$;
\item \label{same_volume} $\vol(M_n')/\vol(M_n) \xrightarrow[n\to+\infty]{} 1$; 
\item \label{also_BS} The sequence $M_n'$ is BS-convergent to $\HH^3$. 
\end{enumerate}
\end{lem}

We can conclude the proof using this lemma. According to \ref{also_BS} we can apply \cite[Theorem 1.8]{7S} to the sequence $M_n'$ and we get that $b_1(M_n') = o(\vol M_n)$. On the other hand it is easy to see that because of \ref{is_surgery} we have $b_1(M_n) \le b_1(M_n') + h_n$ (where $h_n$ is the number of cusps of $M_n$). Thus we obtain: 
\[
\frac{b_1(M_n)}{\vol(M_n)}  \le \frac{\vol(M_n')}{\vol(M_n)}\cdot \frac{b_1(M_n')}{\vol(M_n')} + \frac{h_n}{\vol(M_n)}. 
\]
By Lemma \ref{nbcuspsBS} and \ref{same_volume} we finally get that the right-hand side is an $o(1)$. 

\begin{proof}[Proof of Lemma \ref{dehn}]
For $(p,q)\in\NN^{h_n}\times\NN^{h_n}$ such that $p_j,q_j$ are coprime for all $j$ let $M_n^{p/q}$ be the compact manifold obtained by $(p,q)$-Dehn surgery on $M_n$. Then $M_n^{p/q}$ converges geometrically to $M_n$ as $\min_j(p_j+q_j)$ goes to infinity, and it follows that for a given $R>0$ there exists a $m_n$ such that when $\min_j(|p_j| + |q_j|)>m_n$ we have
\[
\vol (M_n^{p/q})_{\le R} \le \vol(M_n)_{\le 2R}. 
\]
We can choose a sequence $(p^n,q^n)\in (\NN\times\NN)^{h_n}$ such that \ref{same_volume} holds, and moreover $\min_j(|p_j^n| + |q_j^n|)>m_n$; it then follows from the inequality above that $M_n'$ is BS-convergent to $\HH^3$. 
\end{proof}


\appendix

\section{The heat kernel on truncated manifolds}

\label{heat}
\subsection{Introduction}

Let $\sB$ be a collection of open horoballs in $\HH^3$ whose closures are pairwise disjoint, and let $X$ be the smooth manifold with boundary $\HH^3-\bigcup_{B\in\sB} B$. We will denote 
\[
\delta_X = \inf_{B\not=B'\in\mathcal B} d(B,B')
\]
and we will always suppose that $\delta_X>0$ (this is obviously always the case when $\sB$ comes from a truncated manifold). The aim of this appendix is to show that the proof of \cite[Proposition 5.3]{RS} can be adapted to this setting to yield the following result:

\begin{prop} \label{cly_ha}
For any $\delta>0,t_0>0$ there is a $C>0$ such that for all $X$ as above which satisfy $\delta_X\ge\delta$ and every $x,y\in X$ and $t\in]0,t_0]$ we have
\[
|e^{-t\Delta^p[X]}(x,y)| \le Ct^{-3/2} e^{-\frac{d(x,y)^2}{5t}}. 
\]
\end{prop}

Let us see how this statement implies Proposition \ref{cly_bd}: let $M=\Gamma\bs\HH^3,Y$ be as in its statement, $X=\wdt{M^Y}$. We need to prove that $\delta_X$ is bounded below by a constant not depending on $M$. Let $H,H'$ be horospheres in $\pl X$ such that $d(H,H')=\delta_X$. By the hypothesis on $Y$ there are elements $\eta,\eta'\in\Gamma$ which stabilise $H,H'$ respectively and such that their displacement is smaller than $\eps/10$ where $\eps$ is the Margulis constant of $\HH^3$; for $\delta$ small enough (independent of $X$), if $\delta_X<\delta$ then $\eta\eta'$ displaces of less than $\eps$ on $H$; but since it does not commute with $\eta$ this is impossible by the Margulis Lemma. Thus $\delta_X$ is uniformly bounded away from 0 for such an $X$. 


\subsection{The ``single layer potentials'' construction of the heat kernel}

We recall here the construction of the heat kernel on a manifold with 
boundary $W$ given in \cite[Section 5]{RS}, which starts from an isometric 
embedding $W \subset W'$ into a complete manifold such that the heat kernel on 
$W$ satisfies Gaussian bounds. This reference deals only with compact manifolds 
and thus we cannot apply its results directly to our situation; but on the 
other hand we will see in the next section that the integrals on $\pl X$ which 
we need to converge are indeed absolutely convergent. The results are also 
stated only for bundles with orthogonal monodromy but the arguments work 
for all flat bundles with a euclidean metric. 

The construction goes as follows: let $Q^{(0)}(x,y,t)=e^{-t\Delta^p[W']}(x,y)$ 
and for $m\ge 1$ define by induction:
\begin{equation} \label{intbd}
\begin{split}
Q^{(m)}(x,y,t) &= \int_0^t \int_{\pl W} \bigl( Q^{(0)}(x,z,s)\wedge *dQ^{(m-1)}(z,y,t-s) \\
              &\phantom{= \int_0^t \int_{\pl W} \bigl(} + \delta Q^{(0)}(x,z,s)\wedge *Q^{(m-1)}(z,y,t) \bigr)dz\, ds.
\end{split}
\end{equation}
For $W = X$ we will check that this integral is convergent for all $m$ in Section \ref{check} below. The main result of Ray and Singer with regard to the $Q^{(m)}$ is then stated as follows (\cite[Lemma 5.12]{RS}; the function $D$ on $X$ is the distance to the boundary $\pl X$). 

\begin{prop} \label{Ray_singer_prop} 
Under the hypotheses of Proposition \ref{cly_ha}, for all $m\ge 1$ the kernel $Q^{(m)}$, as well as its differential and co-differential in the variables $x,y$, satisfy the Gaussian bounds:
\begin{equation} \label{gaussm}
|Q^{(m)}(x,y,t)| \le \frac{C^m}{\Gamma(m/2)}e^{-\frac{D(x)^2+D(y)^2}{5t}} t^{-\frac 3 2} e^{-\frac{d(x,y)^2}{5t}}; 
\end{equation}
for some constant $C>0$.
\end{prop}

This is proven by induction on $m$, and to carry the induction step one needs (as is obvious from the formula \eqref{intbd}) also the bounds on the derivatives. 

It follows from Proposition \ref{Ray_singer_prop} that the kernel $K_t^p$ given by 
\[
K_t^p(x,y) = \sum_{m=0}^{+\infty} (-2)^m Q^{(m)}(x,y,t)
\]
is the heat kernel on $p$-forms on $X$ with coefficients in $V$ (see \cite[Corollary 5.14]{RS}), and thus that the latter satisfies the Gaussian bounds stated in Proposition \ref{cly_ha}. Moreover, we also have the bounds:
\begin{equation} \label{Luck_Schick_here}
|K_t^p(x,y) - e^{-t\Delta^p[W']}| \ll e^{-\frac{D(x)^2+D(y)^2}{5t}} t^{-\frac 3 2} e^{-\frac{d(x,y)^2}{5t}} 
\end{equation}
(note that we can apply this to $\wdt{M^Y} \subset \HH^3$, but also to $\wdt{M^Y} \subset \wdt{M^Z}$ for $Z \ge Y$, according to Proposition \ref{cly_ha}). 


\subsection{Convergence of the integrals on the boundary}
\label{check}

If one manages to show that the integrals on the boundary in the definition \eqref{intbd} of $Q^{(m)}$ are uniformly convergent for all $X$ then the original argument of Ray and Singer carries over to yield Proposition \ref{cly_ha}. Indeed, it rests only on local computations (which remain valid for unimodular instead of orthogonal coefficients), and the hypothesis that $\delta_X\ge\delta$ gives uniform bounds on the local geometry (meaning that each point in $X$ has a neighbourhood whose isometry class does not depend on $X$) which Ray and Singer use implicitely in their proof (in the cases they consider it follows immediately from the compactness of the manifolds). 

This uniform convergence can be proven by induction on $m$ (recall that the induction hypothesis carries bounds on both $Q^{(m)}$ and its differentials), which reduces it to show that for  $k, l > 0$ the integral : 
\begin{equation} \label{convergent_integral}
I = \int_{\pl X} d(x, z)^k \cdot d(y, z)^l \cdot e^{-\left(\frac{d(x,z)^2}{5s} + \frac{d(y,z)^2}{5(s-t)}\right)} dz
\end{equation}
converge uniformly for all $x, y \in X$. The polynomial terms may appear when taking derivatives, there are also fectors depending on $s$ but these do not affect the convergence of the integral and thus are dealt with in Ray--Singer's argument. The proof that \eqref{convergent_integral} is convergent will be similar to the study of the term $I_2$ in the proof of Proposition \ref{geom_side}. For $w\in X$ and $B\in\sB$ we denote by $z_B^w$ the projection of $w$ onto $B$. We observe first that the supremum
\[
F(x,y,t) = \sup\left\{ \int_{z\in \pl B} d(z_B^x, z)^k \cdot d(z_B^y, z)^l \cdot e^{-\left(\frac{d(z_B^x,z)^2}{5s} + \frac{d(z_B^y,z)^2}{5(s-t)}\right)} dz\: : B\in\sh\right\}
\]
is finite and actually uniformly bounded in $x,y$. Indeed, for any horosphere $H$ of $\HH^3$ and any $z_0\in H$, we have for $t\le t_0$: 
\[
\int_H d(z_0, z)^K e^{-\frac{d(z_0,z)^2}{5t_0}} dz \le \int_\CC |z|^K e^{-a\left(\log(1+|z|)\right)^2} dzd\ovl z
\]
where $a$ depends only on $t_0$, and the integral at the rightmost above is convergent for any $a > 0$, since the function $r \mapsto e^{-(\log r)^2}$ is $o(r^{-L})$ for any $L > 0$. Since $z_B^x, z_B^y \in \partial B$ the same arguments apply to the integrals $\int_{z\in \pl B}e^{-\left(\frac{d(z_B^x,z)^2}{5s} + \frac{d(z_B^y,z)^2}{5(s-t)}\right)} dz$ to give a bound independent of $x, y$ or $B$. 

Now we remark that for a given $B\in\sB$ and all $z\in \pl B$ we have 
\begin{equation} \label{duh}
d(x,z) \ge \frac{d(x,z_B^x)+d(z_B^x,z)}2
\end{equation}
and similarly for $y$: indeed, for any $z\in\pl B$ we have $d(x,z)\ge d(z_B^x,z)$ and $d(x, z) \ge d(x, z_B^x)$ (because the geodesic triangle $xzz_H^x$ has an obtuse angle at $z_B^x$). Now we consider separately those $z$ for which $d(z,z_B^x)\le d(x,z_B^x)$ and those for which the reverse inequality holds, for both of which \eqref{duh} holds trivially. In addition we have $d(x, z)^k \le (d(x, z_B^x) + d(z_B^x, z))^k$ and for large distances the right-hand side is bounded by $d(x, z_B^x)^k \cdot d(z_B^x, z)^k$. It follows that 
\begin{align*}
I & \le \sum_{B\in\sB} d(x, B)^k \cdot d(y, B)^l \cdot e^{-\left(\frac{d(x,B)^2}{20s}  + \frac{d(y,B)^2}{20(t-s)} \right)} \int_{\pl B} d(z_B^x, z)^k \cdot d(z_B^y, z)^l \cdot e^{-\left(\frac{d(z_B^x,z)^2}{20s} + \frac{d(z_B^y,z)^2}{20(s-t)}\right)} dz \\
  &\le \sum_{B\in\sB} d(x, B)^k \cdot d(y, B)^l \cdot e^{-\left(\frac{d(x,B)^2}{20s} + \frac{d(y,B)^2}{20(t-s)} \right)} F(x,y,t). 
\end{align*}
It follows from Lemma \ref{counthoro} that the series on the last line is absolutely convergent, finishing the proof that the integral \eqref{convergent_integral} is absolutely convergent.


\bibliographystyle{plain}
\bibliography{bib2}

\end{document}